%% file: OCP_revised2.tex
\documentclass[a4paper,11pt]{article}
\usepackage[normalem]{ulem} 
\usepackage{titlesec}

\usepackage{amsmath,amsthm,amssymb,enumerate}
\usepackage[T1]{fontenc}
\usepackage{gensymb}
\usepackage[square,sort&compress,comma,numbers]{natbib}
\usepackage{outlines}
\usepackage[sc]{mathpazo}
\usepackage{multirow} 
\usepackage{newtxtext,newtxmath}
\usepackage{subfig}
\usepackage{graphicx}
\usepackage{tabularx}
\usepackage{epstopdf}
\usepackage{cancel}
\usepackage[left=2.5cm, right=2.2cm, top=2cm]{geometry}
\usepackage{tikz}
\usetikzlibrary{patterns,calc}
\usepackage{caption}
\usetikzlibrary{shapes,calc}
\usepackage{verbatim}
\usepackage{mathrsfs}
\usepackage{accents}
\usepackage[utf8]{inputenc}
\usepackage{appendix}
\usepackage[ruled]{algorithm2e}
\usepackage{chngcntr}

\usepackage{multirow}
\usepackage{array}
\usepackage{tcolorbox}
\usepackage{tcolorbox}
\tcbuselibrary{skins} 

\newcommand{\M}{\text{M}}
\newcommand{\yb}{\bar{y}}
\newcommand{\yd}{y_{\rm d}}
\newcommand{\pb}{\bar{p}}
\newcommand{\ub}{\bar{u}}
\newcommand{\bu}{{\bf{u}}}
\newcommand{\by}{{\bf{y}}}

\def\tp{\widetilde{p}}
\def\ty{\widetilde{y}}

\def\osc{{\rm osc}}
\usepackage[colorinlistoftodos]{todonotes}

\usetikzlibrary{decorations.pathmorphing}
\usetikzlibrary{decorations.pathreplacing}
\usetikzlibrary{positioning}
\usetikzlibrary{shapes}
\usetikzlibrary{arrows}
\usetikzlibrary{patterns}
\usetikzlibrary{fadings}
\usetikzlibrary{plotmarks}
\usetikzlibrary{calc}
\usetikzlibrary{intersections}
\tikzstyle{every picture}+=[font=\footnotesize]
\usepackage{paralist}
\usepackage{bbm}
\usepackage{latexsym}           
\usepackage{enumerate}
\usepackage{enumitem}

\setlist{noitemsep, topsep=0.8ex, partopsep=0pt
	, leftmargin=2em}
\setlist[1]{labelindent=\parindent}

\newlist{axioms}{enumerate}{1}
\setlist[axioms]{font=\bfseries}

\newlist{alphenum}{enumerate}{1}
\setlist[alphenum]{label=\textbf{(\alph*)}, leftmargin=3em}

\newlist{alphienum}{enumerate}{1}
\setlist[alphienum]{label=\textit{(\alph*)}}

\newlist{romanenum}{enumerate}{1}
\setlist[romanenum]{label=\textit{(\roman*)}}

\newlist{romaninenum}{enumerate*}{1}
\setlist[romaninenum]{label=\textit{(\roman*)}}
\usepackage[ruled]{algorithm2e}
\SetKwIF{If}{ElseIf}{Else}{if}{}{else if}{else}{endif}
\SetKwFor{For}{for}{}{endfor}
\usepackage[noabbrev, capitalise]{cleveref}
\usepackage{tabu}
\crefname{equation}{\unskip}{\unskip}
\creflabelformat{equation}{#2(#1)#3}\usepackage{amsmath}
\usepackage{amsfonts}
\usepackage{amssymb}
\usepackage[ruled]{algorithm2e}

\SetFuncSty{textsc}
\SetKwFunction{frun}{Run}
\SetKwHangingKw{arun}{\frun}
\SetKwFunction{fset}{Set}
\SetKwHangingKw{aset}{\fset}
\SetKwFunction{ffind}{Find}
\SetKwHangingKw{afind}{\ffind}
\SetKwFunction{fselect}{Select}
\SetKwHangingKw{aselect}{\fselect}
\SetKwFunction{fcompute}{Compute}
\SetKwHangingKw{acompute}{\fcompute}
\SetKwFunction{fsolve}{Solve}
\SetKwHangingKw{asolve}{\fsolve}
\SetKwFunction{fupdate}{Update}
\SetKwHangingKw{aupdate}{\fupdate}
\SetKwFunction{festimate}{Estimate}
\SetKwHangingKw{aestimate}{\festimate}
\SetKwFunction{fmark}{Mark}
\SetKwHangingKw{amark}{\fmark}
\SetKwFunction{fbreak}{Break}
\SetKwHangingKw{abreak}{\fbreak}
\SetKwFunction{frefine}{Refine}
\SetKwHangingKw{arefine}{\frefine}
\SetKwFunction{compute}{Compute}
\SetKwFunction{set}{Set}

{\algorithm}%
{\endalgorithm}

\input{command_ccnn}

\newcommand{\ad}{{\rm{ad}}}
\newcommand\useleqno{\renewcommand\@eqnnum{\hb@xt@.01\p@{}%
                      \rlap{\normalfont\normalcolor
                        \hskip -\displaywidth(\theequation)}}}

\newcommand{\h}{{\rm h}}

\newcommand{\heta}{\widehat{\eta}}
\newcommand{\rJ}{{\rm J}}
\newcommand{\rM}{{\rm M}}
\newcommand{\rI}{{\rm I}}
\newcommand{\dd}{{\texttt {d}}}
\newcommand{\z}{{\rm z}}
\newcommand{\zy}{y_{\zeta,\rm d}}

\newcommand{\no}{{\nonumber}}

\def \D{{{\rm {\!}\cal D}}}
 \def \Z{{{\mathcal Z}}}

\allowdisplaybreaks

\def\pw{{\rm pw}}

\def\cJ{\mathcal{J}}

\def\cT{\mathcal{T}}

\newcommand{\norm}[1]{{\vert\kern-0.25ex\vert #1 	\vert\kern-0.25ex \vert}}

\newcommand{\enorm}[1]{{\vert\kern-0.25ex\vert\kern-0.25ex\vert #1 	\vert\kern-0.25ex \vert\kern-0.25ex \vert}}

\title{Quasi-optimality of adaptive FEM for optimal control problems involving Dirac measures governed by biharmonic equation}
	\author{Asha K. Dond\footnote{ashadond@iisertvm.ac.in, School of Mathematics, Indian Institute of Science Education and Research Thiruvananthapuram, India.} \quad
 Neela Nataraj\footnote{neela@math.iitb.ac.in, Department of Mathematics, Indian Institute of Technology Bombay, Powai, Mumbai, India.} \quad
   Subham Nayak\footnote{sn20@iisertvm.ac.in, School of Mathematics, Indian Institute of Science Education and Research Thiruvananthapuram, India.}}
	
\begin{document}
		\maketitle
		\begin{abstract} 
			\noindent 
  This article establishes the quasi-optimality of adaptive nonconforming finite element methods for a class of optimal control problems involving Dirac measures governed by the biharmonic equation. The nonconforming Morley finite elements are employed for discretising both the state and adjoint variables. 
  A modified right-hand side through a companion operator that maps Morley finite elements to a conforming space helps to overcome the challenge in handling point sources on the right-hand side. A~priori and a~posteriori error estimates for the optimal control problems are derived. Further, optimal convergence rates for adaptive finite element methods are established using an axiomatic framework: by proving key properties such as
 {\it stability}, {\it reduction}, {\it discrete reliability}, and {\it quasi-orthogonality}. Numerical experiments for three types of optimal control problems are discussed extensively and they validate the theoretical results.
   \end{abstract}
		

\section{Introduction}  
 We consider three types of optimal control problems (OCP) that aim to minimize the cost functional $\cJ$ within the admissible set $U_\ad$ subject to the biharmonic state equation with clamped boundary conditions. Let $\Omega\subset \mathbb{R}^2$ be a bounded domain with Lipschitz boundary $\partial \Omega$. Let $\Z$ and $\D$ be finite-ordered subsets of $\Omega$ with cardinality $n$ and $m$, respectively, and $\delta_{\rm z}$ denote the Dirac distribution at the point ${\rm z}$ in $\Omega$. The first  OCP is governed by a {\it biharmonic equation with point sources} and is described below.
 \begin{align}
\begin{aligned}
\arg\min_{\bu\in U_\ad}\cJ(y,\bu)&:=\frac{1}{2}\norm{y-\yd}^2_{L^2(\Omega)}+\frac{\alpha}{2}\norm{\bu}_{\mathbb{R}^n}^2 
\text{ subject to } \Delta^2 y=\sum_{{\rm z}\in\Z}u_{\rm z}\delta_{\rm z} \text{ in } \Omega \notag \\
\text{ with } U_{\ad}&:=\big\{\bu=\{u_{\rm z}\}_{{\rm z}\in\Z}\in \mathbb{R}^n:~ a_{\rm z}\leq u_{\rm z} \leq b_{\rm z} \fl {\rm z}\in \Z \big\}.
\end{aligned}
\hspace{1em} \left.\rule{0pt}{2.5em}\right\} \tag{\bf{P1}}\label{opt1}
\end{align}
The second OCP  involves {\it point-wise tracking} and is defined by
\begin{align}
\begin{aligned}
\arg\min_{u\in U_\ad}\cJ(y,u)&:=\frac{1}{2}\sum_{\zeta\in\D}|y(\zeta)-y_{ \zeta\rm,d}|^2+\frac{\alpha}{2}\norm{u}_{L^2(\Omega)}^2
\text{ subject to }
\Delta^2 y=f+u\text{ in }\Omega\\
\text{ with } U_{\rm ad}&:=\{u\in L^2(\Omega):u_a\leq u(x)\leq u_b\text{ a.e. in } \Omega\}\no.
\end{aligned}
\hspace{1em} \left.\rule{0pt}{2.5em}\right\} \tag{\bf{P2}}\label{opt1n}
\end{align}
The third OCP focuses on {\it point-wise tracking problem} governed by the biharmonic equation with {\it point sources} defined by
\begin{align} 
\begin{aligned}
\arg\min_{\bu\in U_\ad}\cJ(y,\bu)&:=\frac{1}{2}\sum_{\zeta\in\D}|y(\zeta)-y_{ \zeta\rm,d}|^2+\frac{\alpha}{2}\norm{\bu}_{\mathbb{R}^n}^2 \text{ subject to }\Delta^2 y=\sum_{{\rm z}\in\Z}u_{\rm z}\delta_{\rm z} \text{ in }\Omega\\
 \text{ with } U_{\ad}&:=\big\{\bu=\{u_{\rm z}\}_{{\rm z}\in\Z}\in \mathbb{R}^n:~ a_{\rm z}\leq u_{\rm z} \leq b_{\rm z}\fl {\rm z}\in \Z \big\}.\no
\end{aligned}
\hspace{1em} \left.\rule{0pt}{2.5em}\right\} \tag{\bf{P3}}\label{opt2n}
\end{align}
In $\eqref{opt1}$, the desired state $y_{\rm d}\in L^2(\Omega)$ while in \eqref{opt1n} and \eqref{opt2n}, ${\bf y}_{\rm d}:=\{y_{ \zeta\rm,d}\}_{\zeta\in\D}\in \ell^2(\D)$. The regularization parameter is denoted by $\alpha>0$. In \eqref{opt1} and \eqref{opt2n}, the control variable corresponds to the amplitude of forces modeled as point sources such that the control bounds ${\bf{a}}=\{a_{\rm z}\}_{{\rm z}\in\Z}$ and ${\bf{b}}=\{b_{\rm z}\}_{{\rm z}\in\Z} \in \mathbb R^n$ satisfy $a_{\rm z} \leq b_{\rm z}$  for all  ${\rm z}\in\Z$. In \eqref{opt1n}, distributed control is considered with constraints $u_a,u_b\in {\mathbb R}$ satisfying $u_a\leq u_b$. In \eqref{opt1}, the force is applied at specific points in the domain, causing the state variable to approach the desired state. In contrast, in \eqref{opt1n} and \eqref{opt2n}, the state variable moves towards the desired state at certain points when the force is applied across the entire domain or only at finitely many points, respectively. 
  
\medskip
\noindent  
Some applications of optimal control problems governed by the biharmonic equation that involve Dirac measures are described now. In structural mechanics, they model the deflection of thin elastic plates under point loads, aiming to optimize load distribution for desired plate deformation \cite{JLL1971}. 
In electrostatics, the partial differential equations (PDE) model the potential fields due to point charges, with the control problem focusing on the optimal charge distribution to achieve a specified electric potential \cite{Jackson1999}. The OCP involving Dirac measures have diverse applications; such as in the fields of Caffarelli–Silvestre extension for fractional
diffusion, boundary controllability of parabolic and hyperbolic degenerate
equations, the optimal control of selective cooling of steel, and in the active
control of sound (see ref. \cite{HEJ18} and references therein). The problem of pointwise controlling the motion of a vibrating plate is modeled as an OCP governed by a biharmonic equation with point sources in \cite{AS1991}.


\medskip
\noindent Conforming finite element methods (FEM) for higher-order PDE involve cumbersome $C^1$ finite elements. In this context, nonconforming Morley finite element methods with piecewise $\mathcal{P}_2$ polynomials offer an attractive alternative, especially for fourth-order problems where the exact solution has minimal regularity. Notable contributions to Morley FEM can be found in works such as \cite{BGPS18,CCNN2021,CN2020,DG_Morley_Eigen,CCBGNN24a} and the references therein. For more references on convergence of optimal control problems governed by fourth-order PDE and Morley FEM, see \cite{AKDNNSN2024}.

\medskip
\noindent In \cite{AKDNNSN2024}, the quasi-optimality of the adaptive finite element method (AFEM) for the OCP is established, where the load function in the state equation, as well as the control variable, is in $L^2(\Omega)$. However, this analysis does not apply to the current case, as we focus on AFEM for a class of OCP that involve Dirac measures, and $\delta_{\rm z} \notin L^2(\Omega)$. While there is existing literature on biharmonic equations governed by loads that belong to the dual of the solution space, see for example \cite{DF85,BS,CCNN2021,CCBGNN24}, specific literature for OCP with point sources is notably absent. 

\medskip
\noindent Nevertheless, second-order PDE with point loads, and even OCP featuring point sources governed by the Poisson equation, have been studied using conforming finite element methods. 
A~priori and a~posteriori error estimates are established for OCP with point sources governed by second-order PDE with respect to some weighted norm with Muckenhoupt’s weights of class A2 in \cite{AOES17,AOR18}. The OCP for the Navier-Stokes equations with point sources is discussed in \cite{AFED21}. The point-wise tracking OCP governed by second-order PDE has also been studied in the literature, where the state variable approaches the desired state at certain locations to minimise the cost functional in \cite{AOES17,HEJ18}. Moreover, in \cite{AFED19,NB21,AFE22}, the pointwise tracking OCP governed by Stokes equations and the OCP governed by semi-linear elliptic equations are studied. 

\medskip
\noindent This article discusses a unified approach for adaptive FEM for OCP {\bf (P1)}-{\bf (P3)} governed by the biharmonic equation involving Dirac measures. {\it To the best of our knowledge, this article is the first to study the AFEM for OCP governed by the biharmonic equation involving Dirac measures.} The nonconforming Morley finite elements are employed for discretising both the state and adjoint variables. For the control variable, variational discretization is considered in \eqref{opt1n}. In \eqref{opt1} and \eqref{opt2n}, the control variable is already in a finite-dimensional space. The point load leads to employing the definition of $\langle \delta_{\rm z}, \phi \rangle$ for any function $\phi$ that is continuous at the point ${\rm z}\in \Omega$, using the point-wise evaluation $\phi({\rm z})$. {However, since Morley finite elements are not continuous except at the nodes and the interior of elements,} $\langle \delta_{\rm z}, \phi_\rM \rangle$ is not well-defined for any function $\phi_\rM$ in the Morley finite element space. This necessitates the use of a companion operator, $J$ that maps Morley elements to the conforming space (see \cite{BS, CCNN2021}) and modifies the load function in the discrete formulation as $\delta_{\rm z}\circ J$. 

\medskip
\noindent One of the key ideas in the error analysis is to establish the equivalence between the total error in the solution to the OCP and the errors in the solution of certain auxiliary biharmonic equations with $L^2$ load and/or point loads. The strategy of construction of auxiliary problems is identical for \eqref{opt1}-\eqref{opt2n}. This enables a unified analysis for all the problems \eqref{opt1}-\eqref{opt2n}, as the point loads shift from the state equation in \eqref{opt1} to the adjoint equation in \eqref{opt1n} and appear in both state and adjoint equations in \eqref{opt2n}. The a~priori and a~posteriori error analysis for the corresponding biharmonic equation with $L^2$ load and point loads forms the basis for establishing the a~priori and a~posteriori error estimates for the OCP. 

\medskip \noindent Figure~\ref{no_ref}
 illustrates examples of adaptive mesh refinements for \eqref{opt1}-\eqref{opt2n}. The a~posteriori error estimators of OCP help to identify the critical regions in the domain. As forces are applied at the points in $\Z$ for \eqref{opt1} and \eqref{opt2n}, we expect finer refinements in the mesh at those locations. Similarly, since the state solutions need to approach the desired state only at the locations in $\D$ for \eqref{opt1n} and \eqref{opt2n}, finer adaptive refinements are predicted at those locations. The adaptive mesh refinements for \eqref{opt1}-\eqref{opt2n} show finer refinements at the locations of points in $\Z$ and $\D$ {(ref. Figure~\ref{no_ref})}. 
\begin{figure}[ht!] 
\begin{tabular}{cccc} 
\hspace{5cm}&\includegraphics[width=4.5cm,height=3.5cm]{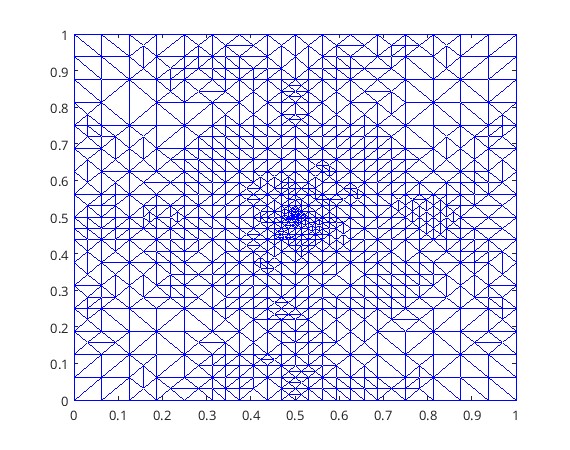}
& \includegraphics[width=4.5cm,height=3.5cm]{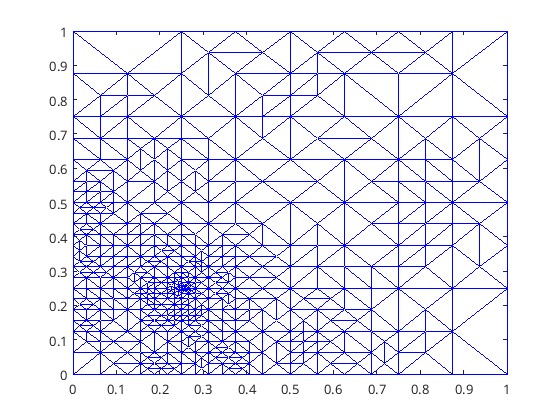}
& \includegraphics[width=4.5cm,height=3.5cm]{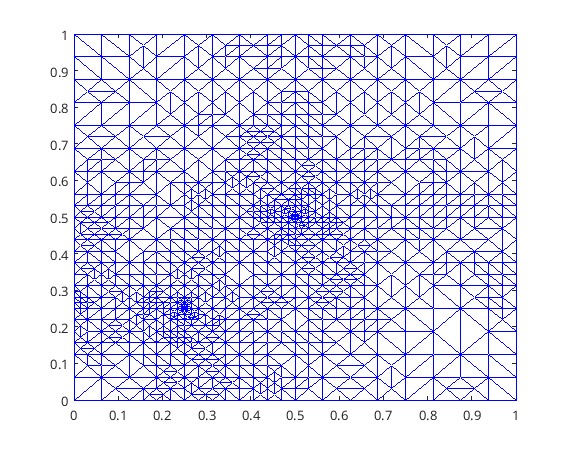}\\
\hspace{5cm}& {\small\eqref{opt1} $\Z=\{(0.5,0.5)\}$} &
 {\small\eqref{opt1n} $\D=\{(0.25,0.25)\}$} & {\small \eqref{opt2n} $\Z=\{(0.5,0.5)\}$, $\D=\{(0.25,0.25)\}$}
\end{tabular}   	
\caption{\small   The adaptive mesh for OCP in \eqref{opt1}-\eqref{opt2n}, having refinements  near point sources $\Z$ and at the point $\D$.}
\label{no_ref}
\end{figure}

\medskip
\noindent  The primary objective of this paper is to establish optimal convergence rates of adaptive nonconforming FEM for the OCP with Dirac measures governed by fourth-order elliptic PDE.
To establish the quasi-optimality of AFEM, an axiomatic framework that comprises {\it stability}, {\it reduction}, {\it discrete reliability}, and {\it quasi-orthogonality} is followed.  
The {\it quasi-orthogonality} is established under the assumption that the point sources are at the nodes of the initial triangulation in the case of \eqref{opt1}. However, this assumption is required throughout the analysis for the problems \eqref{opt1n} and \eqref{opt2n} discussed in Section \ref{p2ocp}-\ref{OCP_new2}. Achieving quasi-optimality for the OCP is challenging due to the nonconformity of Morley finite element spaces and the less regularity of the Dirac measure. 

\medskip \noindent {\bf Contributions}

The novel contributions of this article are summarized below: 
  \begin{itemize}
      \item {\it Nonconforming FEM} are analysed for OCP governed by biharmonic equation with point sources, pointwise tracking OCP governed by biharmonic equations, and pointwise tracking OCP governed by biharmonic equations with point sources. 
    \item {\it a~priori estimates} in the energy norm as well as in the lower-order norms are established under realistic regularity assumptions on the solution of the state and adjoint equations. 
      \item {\it Reliable} and {\it efficient}  a~posteriori error control is derived for the adaptive Morley finite element method.
      \item {\it Quasi-optimality} of adaptive Morley finite element method for OCP in \eqref{opt1}-\eqref{opt2n} is established. The axiomatic framework \cite{CMMD2014,CS24} with the axioms, {\it stability} {\bf (A1)}, {\it reduction} {\bf (A2)}, {\it discrete reliability} $({\rm \bf A3}_{\varepsilon})$, {\it quasi-orthogonality} $({\rm \bf A4}_{\varepsilon})$ are verified to prove quasi-optimality of the AFEM. 
      \item Extensive numerical experiments for \eqref{opt1}-\eqref{opt2n} in convex and non-convex domains with single and multiple sources are presented by constructing many interesting examples, since computational benchmarks are not available in the literature. We observe that the numerical results validate the theoretical results.
  \end{itemize}
\medskip
\noindent This introductory section is followed by Section \ref{prelim}, which presents the preliminaries and establishes crucial embedding results needed in the sequel. In Section \ref{probP1}, a~priori and a~posteriori estimates for the OCP in \eqref{opt1} are derived. The quasi-optimality for the adaptive FEM for \eqref{opt1} is achieved by verifying the axioms, {\it stability}, {\it reduction}, {\it discrete reliability}, {\it quasi-orthogonality} in Section \ref{AOA}. Section \ref{p2ocp} explores \eqref{opt1n}: the pointwise tracking OCP governed by biharmonic equations with source in $L^2(\Omega)$. {The pointwise tracking OCP with point sources \eqref{opt2n} is discussed in Section \ref{OCP_new2}.} Section \ref{sec:num} is dedicated to numerical experiments for OCP in \eqref{opt1}-\eqref{opt2n}. The derivation of the optimality condition for \eqref{opt1} is presented in the appendix in Section \ref{ap:1stcndn}.


		\section{Preliminaries}\label{prelim}
        In this section, we present some preliminaries followed by the definition of the nonconforming Morley finite element space. The interpolation and companion operators and their properties are presented. Utilizing these operators, discrete Sobolev embedding results are established.

        \medskip
\noindent  Throughout the paper, standard notations on Lebesgue and  Sobolev spaces and their norms are used. {Let $X\subset {\mathbb R}^2$ be a polygonal domain. The standard semi-norm (resp. norm) on $H^{\gamma}(X)$ for $\gamma>0$ is denoted by $| \bullet |_{H^{\gamma}(X)}$ (resp. $\|\bullet\|_{H^{\gamma}(X)}$).} The norm and inner product in $L^2(\Omega)$ are denoted by $\|\bullet\|$ and $(\bullet,\bullet)$, respectively. The Euclidean norm on $\mathbb{R}^n$ is defined as $\norm{\bullet}_{\mathbb{R}^n}$. The notation $P\lesssim Q$ means that  $P\leq C~Q$ for some generic constant $C$ independent of the mesh-size. Further, $P\approx Q$ represents  $P\lesssim Q\lesssim P$. The constants defined in the sequel depend on the shape regularity parameter and the regularization parameter $\alpha$, and will be clear from the context. The notations $C_{\rm AE}$ and $C_{\rm AP}$ denote the generic constants in a~priori error estimates in energy norm and lower order norms for \eqref{opt1}-\eqref{opt2n}. Similarly, $C_{\rm REL}$ and $C_{\rm EFF}$ denotes the generic constants in reliability and efficiency of a~posteriori error estimates for \eqref{opt1}-\eqref{opt2n}. For $a,b,\epsilon  >0$, the weighted Young's inequality reads as
$ab\leq \frac{a^2}{2\epsilon}+\frac{\epsilon b^2}{2}.$

\medskip
\noindent Let $\T$ be an admissible and shape regular triangulation of $\Omega$ with the mesh-size defined as $h_T:=|T|^{1/2}$ for all  $T \in \T$, where $|T|$ is the area of $T$. Let  $h:=\max_{T \in \T} h_T$ and  $h_\T|_T:=h_T$. 
{ Let $\bT(\T)$ denote the set of all admissible refinements of $\T$} and $\T_k$ denote the $k^{\text{th}}$ refinement of an initial triangulation $\T_0$. Further, let $\bT(\delta)$ denote the set of all shape regular triangulations with $h\leq \delta$ for $ \delta>0$. Let ${\cal N}$  (resp. $\E$) denote the set of all the nodes (resp. edges) in $\T$ with ${\cal N}={\cal N}^i\cup {\cal N}^b$ (resp. $\E=\E^i\cup\E^b$), where ${\cal N}^i$ and ${\cal N}^b$ (resp. $\E^i$ and $\E^b$) are the set of all interior and boundary nodes (resp. edges) in $\T$, respectively. Let ${\cal N}(T)$ denote the set of vertices of $T \in \T$ and let $\omega_T:=\text { int } \left(\cup_{{\rm z} \in {\cal N}(T)} \cup \T({\rm z})\right)$, where $\T({\rm z})$ are the triangles that share the vertex $\z$ and $\E(\omega_T)$ are the edges in $\omega_T$. 
  Let $\E(T)$  denote the set of edges of $T$ and $h_E:=|E|$ be the length of edge $E$. Since $\T$ is shape regular triangulation, there exist $\kappa_1,\kappa_2>0$ such that $\kappa_{1}^{-1}h_T\leq h_E\leq \kappa_2 h_T$ for all $T\in \T$ and $E\in \E(T)$. Let $\nu_E$ and $\tau_E$ be the unit vectors in the outer normal direction and tangential direction to $E$, respectively. For an interior edge $E\in \E^i$ shared with triangles $T^+$ and $T^-$, define the jump and averaging operators $[\cdot]$ and $\{\cdot\}$ as
			$$\left [ \phi \right ]_E := (\phi|_{T^+})|_E-(\phi|_{T^-})|_E\text{ and }\left\{\phi\right\}_E:=\frac{1}{2}\big((\phi|_{T^+})|_E+(\phi|_{T^-})|_E\big).$$	 
Further, $D_\pw^{m}$ denotes the {piecewise gradient for $m=1$ and the piecewise Hessian for $m=2$}. {Let \( |{\bullet}|_{H^m(T)} \) denote the \(m\)-th order seminorm on \(H^{m}(T)\). For any triangulation \(\T\) and \(0<\gamma<1\), define the broken \(H^{\gamma}\)-norm by 
$\norm{\bullet}_{H^{\gamma}(\T)} := 
\big( \sum_{T \in \T} \norm{\bullet}_{H^{\gamma}(T)}^{2} \big)^{1/2}$, where $H^{\gamma}(\T) := \{ v \in L^{2}(\Omega) : v|_{T} \in H^{\gamma}(T) 
\ \text{for all } T \in \T \}$. The \(H^{\gamma}(T)\)-norm is given by $\norm{v}_{H^{\gamma}(T)}^{2}= \norm{v}_{L^{2}(T)}^{2} + |v|_{H^{\gamma}(T)}^{2}$, where the seminorm is defined as
\begin{align*}
|v|_{H^{\gamma}(T)}^{2}= \int_{T}\int_{T} \frac{|v(x)-v(y)|^{2}}{|x-y|^{2+2\gamma}} \, \dx \dy.
\end{align*}
For more details on fractional order Sobolev spaces and embedding results, refer to \cite{EGE12}.} Let \( \enorm{\bullet}_\pw \) denote the piecewise \( H^2 \)-seminorm, that is, \( \enorm{\bullet}_\pw := \|D^2_\pw \bullet\| \). Let $\mathcal P_k(\T)$ denote the piecewise polynomials of degree at most $k\in \mathbb N$ and $\Pi_{k,\T}$ denote the $L^2$ projection to  $\mathcal P_k(\T)$ defined by $\norm{v-\Pi_{k, {\T}} v}= \min_{p_k\in \mathcal{P}_k(\T)}\norm{v-p_k}\fl v\in L^2(\Omega)$. {Any projection operator $\Pi:H\rightarrow K$ satisfy
\begin{align}\label{FAproj}
    \norm{\Pi x-\Pi y}_{H}\leq \norm{x-y}_{H}\fl x,y\in H,
\end{align}
where $H$ is a Hilbert space, and $K$ is a closed and convex subset of $H$  \cite[Proposition 7.1.2]{KesavanFA}.}

\medskip
\noindent {{\bf Continuous Sobolev embeddings:}~
   Let $X$ be a Lipschitz domain in ${\mathbb R}^2$. For $1<\gamma\leq 2$, there exists an extension operator $E_\gamma:H^\gamma(X)\rightarrow H^{\gamma}({\mathbb R}^2)$ such that $Ew|_{X}=w$ and $\norm{Ew}_{H^{\gamma}({\mathbb R}^2)}\leq C_0 \norm{w}_{H^{\gamma}(X)} \fl w\in H^\gamma(X),$
    where the constant in this inequality depends on $\gamma$ and the Lipschitz continuity of the boundary of the domain \cite[Theorem A.4]{WM2000}. The definition of the extension operator reveals $\norm{w}_{L^\infty(X)}\leq \norm{Ew}_{L^\infty({\mathbb R}^2)}$. The continuous Sobolev embedding $H^\gamma({\mathbb R}^2)\hookrightarrow L^\infty({\mathbb R}^2)$ (ref. \cite[Theorem 2.3]{HT20}) for $Ew\in H^\gamma({\mathbb R}^2)$ shows $\norm{Ew}_{L^\infty({\mathbb R}^2)}\leq C_{\rm em} \norm{Ew}_{H^\gamma({\mathbb R}^2)}$. The above estimates lead to
    \begin{align}\label{inf2gamma}
        \norm{w}_{L^\infty(X)}\leq 
        \norm{Ew}_{L^\infty({\mathbb R}^2)}\leq C_{\rm em}\norm{Ew}_{H^\gamma({\mathbb R}^2)}\leq 
        C_{\rm emb}\norm{w}_{H^\gamma(X)},
    \end{align}
     where $C_{\rm emb}:=C_{\rm em}C_0$ depends on $\gamma$ and the Lipschitz continuity of the boundary of the domain. The second Sobolev embedding result from \cite[Corollary 2.3]{EGE12} shows for $1<\gamma<2$,
    \begin{align}\label{gamma2H2}
        \norm{w}_{H^\gamma(X)}\leq C_{\rm emb1} \norm{w}_{H^2(X)}\fl w\in H^2(X),
    \end{align}
    where $C_{\rm emb1}$ depends on $\gamma$ and the Lipschitz continuity of the boundary of the domain.}
    
        \medskip
 \noindent  The nonconforming Morley element space is defined by\\
$V_\rM :=$
$   \begin{Bmatrix}
 &   v_{\rm M} \text{ and its normal derivative are continuous at the interior nodes and } \\
 v_{\rm M} \in {\mathcal P}_2(\T): & \text{  midpoints of interior edges, respectively, and vanishes at boundary nodes}  \\
 & \text{ and the midpoints of boundary edges, respectively.}   \\
\end{Bmatrix} $  \\
Let $V_{\rm HCT}$ denotes the conforming Hsieh-Clough-Tocher finite element space \cite[Chapter 6]{Ciarlet}.
\begin{lem}[Morley interpolation operator]\label{Interpolation}
 Let  $\widehat{\T}\in \mathbb{T}$ be a shape regular refinement of  $\T$ and $\widehat{V}_\rM$ be the Morley finite element space defined on $\widehat{\T}$. There exist interpolation operators
			$ I_\rM: V+\widehat{V}_\rM \rightarrow V_\rM \text{ and } \widehat{I}_\rM: V \rightarrow \widehat{V}_\rM$ that satisfy
   \begin{itemize}
	\item[(a)] the integral mean property 
			$D_{\rm pw}^2 I_\rM =\Pi_{0, {\T}} D_{\rm pw}^2$ and as a consequence, the orthogonality condition 
				$a_\pw(w,(1-I_\rM)(v+\widehat{v}_\rM))=0$ for all $  w\in  {\mathcal{P}_2(\T)} \text{ and }v+\widehat{v}_\rM\in V+\widehat{V}_\rM$, 
		\item[(b)]	   
			$	\displaystyle{\sum_{j=0}}^2 h_T^{j-2}\|{D_\pw^j(1-I_\rM) v}\|_{L^2(T)}\leq C_\rI   \norm{D_\pw^2(1-I_\rM) v}_{L^2(T)}   =C_\rI \big(\|  D_\pw^2 v\|_{L^2(T)}-\|D_\pw^2 I_\rM v\|_{L^2(T)}\big)$
   for all $v \in H^2(\T)$ or $ v \in V+\widehat{V}_\rM$ and $T \in \T$,
			\item[(c)] 
				$(1-I_\rM)\widehat{v}_\rM=0\text{ and }(\widehat{I}_\rM-I_\rM)(v+\widehat{v}_\rM)=0$ in $\T\cap \widehat{\T} $ for all $v\in V$ and $\widehat{v}_\rM\in \widehat{V}_\rM$.
   \end{itemize}
			 \end{lem}
 {  \noindent  For proofs of $(a)$-$(c)$ refer to  
 \cite[Sect. 5]{CCSP20} and \cite[Lemma 2.2, Remark 2.1]{CN2020}.}\qed
 
 \medskip
\noindent {Define the set of neighboring triangles of $\T\setminus\widehat{\T}$ in $\T$ as $${\mathcal R}(\T,\widehat{\T}):=\{K\in \cT| \, \exists T\in \mathcal{T} \setminus \widehat{\T} \text{ with } \text{dist}(K,T)=0\}.$$} 
       \begin{lem}[companion operator]\label{enrithm}  {\cite[Lemma 5.1]{CCSP20}}
      There exists a companion operator $J:V_\rM \to V_{\rm HCT}+({\mathcal P}_5(\T)\cap V) $ that satisfies
    \begin{itemize}
	\item[(a)] $\enorm{Jv_\rM}_\pw\leq C_{\rm J}\enorm{v_\rM}_\pw$ for all $v_\rM \in V_\rM $,
	\item[(b)]  
   $\displaystyle{\sum_{j=0}}^{2}h_T^{2(j-2)}\|D_\pw^j(1-J)v_\rM\|_{L^2(T)}^2 
    \leq \Lambda_\rJ^2 \min_{v \in V} \|D_{\rm pw}^2(v_{ \rM} -v)\|^2_{L^2(\omega_T)}\no\\
    \leq  C_{\rm J1}^2  \sum_{T\in \omega_T} h_T \sum_{E\in \E(T)} \norm{[D_{\rm pw}^2 v_\rM]_E  \tau_E}_{L^2(E)}^2 \leq C_{\rm J2}^2 \min_{v \in V} \|D_{\rm pw}^2(v_{ \rM} -v)\|^2_{L^2(\omega_T)}$ for all $v_\rM \in V_\rM $,
	\item[(c)]   for $\widehat{v}_\rM^*:=\widehat{I}_{\rM} J v_\rM \in \widehat{V}_\rM$, 	$\displaystyle{\enorm{\widehat{v}_\rM^*-v_\rM}_{\pw}^2 \leq C_{\rm IJ}^2  \sum_{T\in  {\mathcal R}(\T,\widehat{\T})} h_T \sum_{E\in \E(T)} \norm{[D_{\rm pw}^2 v_\rM]_E  \tau_E}_{L^2(E)}^2}$.
      \end{itemize}
   \end{lem}
    \noindent For proofs of $(a)$-$(c)$, we refer to \cite[Theorem 3.2, Sect. 5]{CCSP20}, \cite[Lemma 3.6]{CSDAKNNSD22}, and \cite[Proposition 2.6]{DG_Morley_Eigen}.
    An application of the inverse estimates \cite[Lemma 12.1]{AJ21} and Lemma \ref{enrithm}$(b)$ also yield the approximation estimates below:
\begin{align}\label{sup2pw}
      \|(1-J)v_\rM\|_{L^\infty(T)}\leq C_{\rm inv}h_T^{-1}\norm{(1-J)v_\rM}_{L^2(T)}\leq C_{\rm J3}h_T\min_{v \in V} \|D_{\rm pw}^2(v_{ \rM} -v)\|_{L^2(\omega_T)},
      \end{align}
      {with $C_{\rm J3}:=C_{\rm inv}\Lambda_{\rm J}$, where $C_{\rm inv}$ denotes the constant from the inverse estimate}.  


\begin{lem}[Poincar\'e type inequalities]\label{Poincare}
 Any $w\in V+ V_\rM$, and $w_\rM\in   V_\rM+\widehat{V}_\rM$ satisfy \\
 $(a)~  \norm{w}+\norm{w}_{L^\infty(\Omega)} +\norm{D_\pw w}\leq C_{\rm PJ} \enorm{w}_\pw \text{ and }
 (b)~  \norm{w_\rM}+\norm{w_\rM}_{L^\infty(\Omega)}+\norm{D_\pw w_\rM}\leq C_{\rm PI} \enorm{ w_\rM}_\pw$. 

\end{lem}
\begin{proof}
 For the proof of $(a)$ refer to \cite[Lemma 4.7]{CCGMNN21}. 
The proof of $(b)$ is presented below.




       
\noindent         For any $w_\rM\in V_\rM+\widehat{V}_\rM$, there exist $v_\rM \in V_\rM$ and $\widehat{v}_\rM\in \widehat{V}_\rM$ such that $w_\rM=v_\rM-\widehat{v}_\rM$.    Lemma \ref{Interpolation}$(a)$ and the definition of projection operator $\Pi_{0, {\T}}$ show
         \begin{align}
       \enorm{\widehat{v}_\rM-I_\rM \widehat{v}_\rM}_\pw=\norm{(1-\Pi_{0, {\T}}) D_\pw^2 \widehat{v}_\rM}= \min_{p_0\in \mathcal{P}_0( {\T})}\norm{D_\pw^2 \widehat{v}_\rM-p_0}\leq \enorm{\widehat{v}_\rM -v_\rM}_\pw.\label{call3}
        \end{align}
        Apply Lemma \ref{Interpolation}$(b)$ with $v=\widehat{v}_\rM$ to arrive at
        \begin{align}
        \sum_{i=0}^{1}\norm{D^i_\pw(\widehat{v}_\rM-I_\rM \widehat{v}_\rM)}
        \leq C_\rI \sum_{i=0}^{1} h^{2-i}\enorm{ \widehat{v}_\rM-I_\rM \widehat{v}_\rM}_\pw.\label{call1}
        \end{align}
        Utilize \eqref{call3} in the above step to conclude
        \begin{align}
        \sum_{i=0}^{1} h^{2-i}\enorm{ \widehat{v}_\rM-I_\rM \widehat{v}_\rM}_\pw\leq 2  C_\Omega\enorm{\widehat{v}_\rM-v_\rM}_\pw \label{NNcal04}
        \end{align}
        with $C_\Omega:=\max\{|\Omega|,|\Omega|^{1/2}\}$.
        Since, $v_\rM-I_\rM \widehat{v}_\rM\in V_\rM$, Lemma \ref{Poincare}$(a)$ yields
        $\sum_{i=0}^{1}\norm{D_\pw^i (v_\rM-I_\rM \widehat{v}_\rM)}\leq C_{\rm PJ} \enorm{v_\rM-I_\rM \widehat{v}_\rM}_\pw.$
        This, a triangle inequality, and \eqref{call3} reveal
        \begin{align}	
\sum_{i=0}^{1}\norm{D_\pw^i(v_\rM-I_\rM \widehat{v}_\rM)}\leq
        2C_{\rm PJ}\enorm{v_\rM-\widehat{v}_\rM}_\pw. \label{call2}
	\end{align}
    A triangle inequality plus \eqref{call1}-\eqref{call2} establish $\norm{w_\rM}+\norm{D_\pw w_\rM}\leq C_{\rm p1} \enorm{ w_\rM}_\pw$ with $C_{\rm p1}:=2C_{\rm PJ}+2 C_\rI C_\Omega.$ {The continuous Sobolev embedding 
    \eqref{inf2gamma} (with $\gamma=2$) and the above estimates lead to $\norm{w_\rM}_{L^\infty(\Omega)}\leq C_{\rm emb}(1+C_{\rm p1}) \enorm{w_\rM}_\pw$. 
    This concludes the proof of $(b)$ with $C_{\rm PI}:=C_{\rm p1}+C_{\rm emb}(1+C_{\rm p1})$.}
  \end{proof}   

        \section{Optimal control problem with point sources {\bf(P1)}}\label{probP1}
   This section states the optimality system at the continuous and discrete levels for \eqref{opt1}. Further, a~priori and a~posteriori error estimates for the OCP are derived in Theorems~\ref{apriori} and Theorem~\ref{aposteriori}, respectively, based on the error equivalence result in Theorem~\ref{equivalence} presented in Section \ref{errorap}.

	\medskip	
	\noindent	The weak formulation that corresponds to  the optimal control problem in \eqref{opt1} is
		\\
		$$\displaystyle{\min_{(y,\bu)\in V\times  U_{\rm ad}} \cJ(y,\bu)} \text{ subject to } a(y,\phi)=\sum_{{\rm z}\in\D}u_{\rm z}\langle\delta_{\rm z},\phi\rangle \text{ for all }\phi \in V,$$
		where $V:=H^2_0(\Omega)$  and for all $\phi,w \in V$, the symmetric bilinear form
 $a(\phi,w):=\int_\Omega D^2 \phi:D^2 w \dx$ is bounded and $V$-elliptic. 
  The existence and uniqueness of the solution of the above optimal control problem follow from \cite[Section 5.3]{AAS21}.
  The optimality system seeks $(\yb,\pb,\bar{\bu})\in V\times V\times U_{\rm ad}$ with $\bar{\bu}:=\{\ub_{\rm z}\}_{{\rm z}\in \Z}$ such that
		\begin{subequations}\label{sym1}
			\begin{align}
				&a(\yb,\phi)=\sum_{{\rm z}\in\Z}\ub_{\rm z}\langle\delta_{\rm z},\phi\rangle\text{  for all }\phi\in V,\label{sym11}\\
				&a(\phi,\pb)=(\yb-\yd,\phi)\text{ for all }\phi\in V,\label{sym12}\\
				&\sum_{{\rm z}\in\Z}{(\pb({\rm z})+\alpha \ub_{\rm z})(v_{\rm z}-\ub_{\rm z})\geq 0\text{ for all }\{v_{\rm z}\}_{{\rm z}\in \Z} =: {\bf{v}}\in U_{\rm ad}.}\label{sym13}
			\end{align}
		\end{subequations}
\noindent The optimal control $\ub_z$ in \eqref{sym13} has the representation in terms of the adjoint variable $\pb$ and the projection operator $\Pi_{[a,b]}$ defined by $\Pi_{[a,b]}(g):= \min\{b, \max \{a,g\} \}$ as 
		$	\ub_{\rm z}= \Pi_{[a_{\rm z}, b_{\rm z}]}\big( -  \alpha^{-1} \pb({\rm z}) \big)\fl {{\rm z}\in \Z}.$ Since $H^2_0(\Omega)$ is continuously embedded in $C(\bar{\Omega})$, the above projection operator is well-defined. The projection operator $\Pi_{[a_{\rm z}, b_{\rm z}]}$ is Lipschitz continuous  (ref. \cite[step 3, Proposition 2.5]{KKRK2014}) and
  \begin{align}\label{lipschitz}
		    |{\Pi_{[a_{\rm z}, b_{\rm z}]}\big( g \big)}|\leq |{g}| \fl g\in \mathbb{R},
		\end{align}
where $|\cdot|$ denotes the absolute value. The right-hand side of \eqref{sym11}, $\langle \delta_{\rm z}, \phi\rangle$ with $\phi\in V$ is well-defined as $H^{\gamma}(\Omega)$ is continuously embedded in $C(\bar{\Omega})$ for $\gamma>1$. The Dirac distribution $\delta_{\rm z}:H^{\gamma}(\Omega)\rightarrow {\mathbb R}$ is a linear and bounded operator. The solution of \eqref{sym11}, $\yb\in V$ satisfies $\yb\in H^{4-\gamma}(\Omega)$ and  $\norm{\yb}_{4-\gamma}\lesssim  \norm{\delta_{\rm z}}_{-\gamma}$ for $2-\sigma\leq\gamma\leq2$ with the elliptic regularity index $0<\sigma\leq 1$ \cite{BlumRannacher}. The index of elliptic regularity $\sigma=1$ follows for a convex polygonal domain $\Omega$, while $1/2<\sigma<1$ holds for a non-convex polygonal domain $\Omega$ \cite[Section 2.1]{CCNN2021}. This notation for regularity is used throughout the article.

  \medskip
\noindent    The discrete control problem seeks $(\yb_\rM,\bar{{\bu}}_\h)\in V_\rM\times U_{\rm  ad}$ with $\bar{\bu}_\h:=\{\ub_{\rm z,h}\}_{{\rm z}\in \Z}$ such that
		\begin{subequations}\label{discrete_cost}
			\begin{align} 
				&   \cJ(\yb_\rM, \bar{\bu}_\h)=\min_{(y_\rM, {\bu}) \in  V_\rM \times  U_{\rm ad}} \cJ(y_\rM, {\bu}) \,\,\, \textrm{ subject to } \\
				&  a_\pw(y_\rM,\phi_\rM)=\sum_{{\rm z}\in\Z}{u}_{\rm z}\langle \delta_{\rm z}, J\phi_\rM\rangle \fl \phi_\rM\in V_\rM, \label{dis_cost_b}
			\end{align}
		\end{subequations}
		 where for $\psi_\rM, \phi_\rM \in V_\rM$, the bilinear form $a_\pw(\psi_\rM,\phi_\rM):=\sum_{T\in\T}\int_T D_\pw^2\psi_\rM:D_\pw^2 \phi_\rM \dx$ is elliptic with respect to  $\enorm{\cdot}_{\pw}:=a_\pw(\cdot,\cdot)^{1/2}$. Here $J:V_\rM\rightarrow V$ denotes the companion operator (see Lemma \ref{enrithm} for details). Since, $J\phi_\rM\in V$, $\langle \delta_{\rm z}, J\phi_\rM\rangle=(J\phi_\rM)({\rm z})$ for all $\phi_\rM\in V_\rM$. In the remaining parts of the paper, $(J\phi_\rM)({\rm z})$ is abbreviated as $J\phi_\rM(z)$ for notational simplicity. The existence and uniqueness of the solution to the discrete problem (\ref{discrete_cost}) follow from \cite[Section 5.3]{AAS21}.

\medskip
\noindent	The discrete optimality system seeks $(\bar{y}_\rM,\bar{p}_\rM,\bar{\bu}_{\h})\in V_\rM \times V_\rM \times U_{\rm ad}$ such that 
   \begin{subequations}\label{discrete1}
			\begin{align} 
				&  a_\pw({\bar y_\rM},\phi_\rM)=\sum_{{\rm z}\in\Z}\ub_{\rm z,h}\langle \delta_{\rm z}, J\phi_\rM\rangle \fl \phi_\rM \in V_\rM   \label{state_eq1}
				\\  
				& a_\pw(\phi_\rM,{\pb_\rM})=(\bar y_\rM- \yd, \phi_\rM) \fl \phi_\rM \in V_\rM \label{dadj}
				\\
				& \sum_{{\rm z}\in\Z} ( J\pb_\rM({\rm z})+\alpha \ub_{\rm z,h}) ( w_{\rm z} - {{ \bar u}_{\rm z,h}}) \ge 0   \; \fl {{\bf{w}}:=\{w_{\rm z}\}_{{\rm z}\in \Z} \in U_{\rm ad}}. \label{optimality}
			\end{align}
   \end{subequations}
   The optimal control $\bar{u}_{\rm z,h}$ in \eqref{optimality} is represented as $\ub_{\rm z,h} = \Pi_{[a_{\rm z}, b_{\rm z}]}\big(-\alpha^{-1}  J\pb_\rM({\rm z})\big )\fl  {{\rm z}\in \Z}.$ The derivation of the discrete optimality condition in \eqref{optimality} is  provided in the appendix.
		


\subsection{Auxiliary results}\label{auxresult}
For given discrete control and state solutions $\bar{\bu}_\h:=\{\ub_{\rm z,h}\}_{{\rm z}\in \Z}$ and 
$\yb_\rM$, define auxiliary problems that seek $(\ty,\tp)\in V\times V$
such that
\begin{align}
    a(\ty,\phi)&=\sum_{{\rm z}\in \Z}\ub_{\rm z,h}\langle\delta_{\rm z},\phi\rangle~\text{and}~a(\phi,\tp)=(\yb_\rM-\yd,\phi) \text{ for all }\phi\in V.\label{aux1}
\end{align}
Since $V\hookrightarrow C(\bar{\Omega})$, $\langle \delta_{\rm z}, \phi\rangle$ is well-defined and $|\langle \delta_{\rm z}, \phi\rangle|\leq \norm{\delta_{\rm z}}_{H^{-2}(\Omega)}\norm{\phi}_{V}$. Therefore, the well-posedness of the solutions is assured by the Lax-Milgram Lemma \cite[Section 6.2.1]{Evans}.

 \begin{thm}[equivalence]\label{equivalence}
Let $(\yb,\pb,\bar{\bu})\in V\times V\times U_{\rm ad}$ (resp. $(\yb_\rM,\pb_\rM,\bar{\bu}_\h)\in V_\rM \times V_\rM \times U_{\rm ad}$) solve \eqref{sym1} (resp. \eqref{discrete1}) and $(\ty,\tp)$ solve \eqref{aux1}. Then the following equivalence result holds: 
$$\norm{\bar{\bu}-\bar{\bu}_\h}_{\mathbb{R}^n}+\enorm{\yb-\yb_\rM}_\pw+\enorm{\pb-\pb_\rM}_\pw\approx  \enorm{\ty-\yb_\rM}_\pw+\enorm{\tp-\pb_\rM}_\pw.$$
The constants in $\approx$ depend on the shape regularity parameter, the cardinality of $~\Z$, and the regularization parameter $\alpha$.
\begin{proof}
{\it Step 1.}~(key estimates). For $\phi \in V$, \eqref{sym11}, \eqref{sym12}, and \eqref{aux1} imply
\begin{align}
    a(\yb-\ty,\phi)=\sum_{{\rm z}\in \Z}(\ub_{\rm z}-\ub_{\rm z,h})\langle\delta_{\rm z},\phi\rangle \text{ and } a(\phi,\pb-\tp)=(\yb-\yb_\rM,\phi).\label{NNcal14}
\end{align}
The test function $\phi=\yb-\ty$ in the first equation of \eqref{NNcal14}, the ellipticity of $a(\cdot,\cdot)$, and Lemma \ref{Poincare}$(a)$  show 
\begin{align}
\enorm{\yb-\ty}_\pw^2\leq  \sum_{{\rm z}\in \Z}|\ub_{\rm z}-\ub_{\rm z,h}||\yb({\rm z})-\ty({\rm z})|&\leq  \sum_{{\rm z}\in \Z}|\ub_{\rm z}-\ub_{\rm z,h}|\norm{\yb-\ty}_{L^{\infty}(\Omega)}\leq \sqrt{n}C_{\rm PJ} \norm{\bar{\bu}-\bar{\bu}_\h}_{\mathbb{R}^n}\enorm{\yb-\ty}_\pw. \no\\
{\rm Thus} \; \enorm{\yb-\ty}_\pw &\leq \sqrt{n}C_{\rm PJ} \norm{\bar{\bu}-\bar{\bu}_\h}_{\mathbb{R}^n}. \label{NNcal15}
\end{align}   
The choice $\phi=\pb-\tp$ in the second equation of \eqref{NNcal14} and analogous arguments show
\begin{align}
\enorm{\pb-\tp}_\pw&\leq C_{\rm PJ} \norm{\yb-\yb_\rM}. \label{NNcal16}
\end{align}
\noindent {\it Step 2.}~(State and adjoint w.r.t. control and auxiliary errors). A triangle inequality and \eqref{NNcal15} show
 \begin{align}\label{NNcal18}
&\enorm{\yb-\yb_\rM}_\pw
\leq \sqrt{n}C_{\rm PJ}\norm{\bar{\bu}-\bar{\bu}_\h}_{\mathbb{R}^n}+\enorm{\ty-\yb_\rM}_\pw. 
\end{align}
Similarly, a triangle inequality, \eqref{NNcal16}, Lemma \ref{Poincare}$(a)$, and \eqref{NNcal18} with $C_1:=1+C_{\rm PJ}^2(1+\sqrt{n}C_{\rm PJ})$ lead to
\begin{align}
\enorm{\pb-\pb_\rM}_\pw
&\leq C_1\big( \norm{\bar{\bu}-\bar{\bu}_\h}_{\mathbb{R}^n}+\enorm{\ty-\yb_\rM}_\pw+\enorm{\tp-\pb_\rM}_\pw\big).\label{NNcal21}
\end{align}

\noindent {\it Step 3.}~(upper bound for control error). 
 The choice $v_{\rm z}=\ub_{\rm z,h}$ (resp. $w_{\rm z}=\ub_{\rm z}$) in $\eqref{sym13}$ (resp. $\eqref{optimality}$) yields
\begin{align*}
\sum_{{\rm z}\in \Z}(\alpha \ub_{\rm z}+\pb({\rm z})) (\ub_{\rm z,h}-\ub_{\rm z})\geq 0\quad ({\rm resp. }\quad
\sum_{{\rm z}\in \Z}(\alpha \ub_{\rm z,h}+J\pb_\rM({\rm z}))(\ub_{\rm z}-\ub_{\rm z,h})\geq 0).
\end{align*}
Add the above displayed inequalities and introduce the auxiliary solution $\tp$ to obtain
\begin{align}\label{Nth1}
\alpha \sum_{{\rm z}\in \Z}(\ub_{\rm z}-\ub_{\rm z,h})^2&\leq \sum_{{\rm z}\in \Z}(J\pb_\rM({\rm z})-\pb({\rm z}))(\ub_{\rm z}-\ub_{\rm z,h})\no\\
&=\sum_{{\rm z}\in \Z} (J\pb_\rM({\rm z})-\tp({\rm z}))(\ub_{\rm z}-\ub_{\rm z,h})+\sum_{{\rm z}\in \Z}(\tp({\rm z})-\pb({\rm z}))(\ub_{\rm z}-\ub_{\rm z,h}). 
\end{align}
Apply Young's inequality to the first term of the right side of the above inequality to obtain 
\begin{align}
     \sum_{{\rm z}\in \Z}(J\pb_\rM({\rm z})-\tp({\rm z}))(\ub_{\rm z}-\ub_{\rm z,h})&\leq \frac{1}{\alpha}\sum_{{\rm z}\in \Z}(J\pb_\rM({\rm z})-\tp({\rm z}))^2 +\frac{\alpha}{4} \sum_{{\rm z}\in \Z}(\ub_{\rm z}-\ub_{\rm z,h})^2.\no
\end{align}
   The substitutions $\phi=\tp-\pb$ and $\phi=\yb-\ty$ in the first and second inequalities of $\eqref{NNcal14}$, respectively, the symmetry of $a(\bullet,\bullet)$,  and elementary algebra lead to 
   \begin{align*}
      \sum_{{\rm z}\in \Z} (\ub_{\rm z}-\ub_{\rm z,h})(\tp({\rm z})-\pb({\rm z}))&= a(\yb-\ty,\tp-\pb)= (\yb_\rM-\yb,\yb-\ty)=-\norm{\yb_\rM-\yb}^2+(\yb_\rM-\yb,\yb_\rM-\ty)\\
       &=-\norm{\yb_\rM-\yb}^2+\norm{\yb_\rM-\ty}^2+(\ty-\yb,\yb_\rM-\ty).
   \end{align*}
   An application of the Cauchy Schwarz inequality, Young's inequality, Lemma \ref{Poincare}$(a)$, and \eqref{NNcal15} shows
   \begin{align*}
   (\ty-\yb,\yb_\rM-\ty)&\leq \frac{\alpha}{4nC_{\rm PJ}^4}\norm{\ty-\yb}^2+\frac{nC_{\rm PJ}^4}{\alpha} \norm{\yb_\rM-\ty}^2\leq \frac{\alpha}{4} \sum_{{\rm z}\in \Z}(\ub_{\rm z}-\ub_{\rm z,h})^2+\frac{nC_{\rm PJ}^4}{\alpha} \norm{\yb_\rM-\ty}^2.
   \end{align*}
  A combination of the last three displayed results in \eqref{Nth1} (omitting
  the term $-\norm{\yb_\rM-\yb}^2$) and a triangle inequality leads to \begin{align}
   C_2^{-1}\norm{\bar{\bu}-\bar{\bu}_\h}_{\mathbb{R}^n}\leq \norm{\ty-\yb_\rM}+ \norm{\tp-\pb_\rM}_{L^{\infty}(\Omega)}+\norm{\pb_\rM-J\pb_\rM}_{L^{\infty}(\Omega)} \label{call4}
   \end{align}
   with $C_2^2:=\max\{2\alpha^{-1}(1+n\alpha^{-1}C_{\rm PJ}^4),2\}$.
  Lemma \ref{Poincare}$(a)$ and \eqref{sup2pw} (with $v_\rM=\pb_\rM\in V_\rM$ and $v=\tp\in V$) reveal
 \begin{align}
   \norm{\bar{\bu}-\bar{\bu}_\h}_{\mathbb{R}^n}\leq  & {C_2 C_{\rm PJ}(\enorm{\ty-\yb_\rM}_\pw+\enorm{\tp-\pb_\rM}_\pw)+C_2 |\Omega|^{1/2}C_{\rm J3}C_{\rm sr}\enorm{\pb_\rM-\tp}_\pw}\no\\
   \leq& C_{3}(\enorm{\ty-\yb_\rM}_\pw+ \enorm{\tp-\pb_\rM}_\pw),\label{Ncal7}
   \end{align}
   with {$C_3:=C_2(C_{\rm PJ}+|\Omega|^{1/2}C_{\rm J3}C_{\rm sr})$, where $C_{\rm sr}$ depends only on the shape regularity parameter.}\\
\noindent {\it Step 4.}~(complete upper and lower bounds). Altogether \eqref{Ncal7}, \eqref{NNcal18}, and  \eqref{NNcal21} lead to
\begin{align}
\norm{\bar{\bu}-\bar{\bu}_\h}_{\mathbb{R}^n}+\enorm{\yb-\yb_\rM}_\pw&+\enorm{\pb-\pb_\rM}_\pw\leq C_{\rm c1}\big(\enorm{\ty-\yb_\rM}_\pw+\enorm{\tp-\pb_\rM}_\pw\big)\label{equiv1}
\end{align}
with  $C_{\rm c1}(\alpha^{-1}):=1+C_1+C_3(1+\sqrt{n}C_{\rm PJ}+C_1)$.

\smallskip \noindent The triangle inequality (used twice), \eqref{NNcal15}, \eqref{NNcal16}, and Lemma \ref{Poincare}$(a)$ show
   \begin{align}\label{Cc2}
   \enorm{\ty-\yb_\rM}_\pw+\enorm{\tp-\pb_\rM}_\pw&
   \leq C_{\rm c2} \big(\norm{\bar{\bu}-\bar{\bu}_\h}_{\mathbb{R}^n}+\enorm{\yb-\yb_\rM}_\pw+\enorm{\pb-\pb_\rM}_\pw\big)
   \end{align}
with $C_{\rm c2}:=\max\{ \sqrt{n}C_{\rm PJ},1+C_{\rm PJ}^2 \}$. This concludes the proof.
\end{proof}
\end{thm}
\noindent {The estimates in weaker norms for the state and adjoint solutions in terms of errors of solutions of the auxiliary problems and the discrete solutions are established next.} This is essential for proving the quasi-orthogonality of AFEM.
\begin{thm}[weaker norm error bounds]\label{aprioriHs}
Let $(\yb,\pb,\bar{\bu})\in V\times V\times U_{\rm ad}$ (resp. $(\yb_\rM,\pb_\rM,\bar{\bu}_\h)\in V_\rM \times V_\rM \times U_{\rm ad}$) solve \eqref{sym1} (resp. \eqref{discrete1}) and $(\ty,\tp)$ solve \eqref{aux1}. {Then for $1<\gamma\leq 2$, it holds that }
$$ \norm{\yb-\yb_\rM}_{H^{\gamma}(\T)} + \norm{\pb-\pb_\rM}_{H^{\gamma}(\T)} \leq C_{\rm c3}\big(\norm{\ty-\yb_\rM}_{H^{\gamma}(\T)}+\norm{\tp-\pb_\rM}_{H^{\gamma}(\T)}+h\enorm{\pb_\rM-J\pb_\rM}_\pw\big).$$
\end{thm}
\begin{proof}
A triangle inequality and {the continuous Sobolev embedding from \eqref{gamma2H2}} show
 \begin{align}
        \norm{\yb-\yb_\rM}_{H^{\gamma}(\T)} + &~\norm{\pb-\pb_\rM}_{H^{\gamma}(\T)}  \no\\
        &\leq C_{\rm emb1}(\norm{\yb-\ty}_{H^2(\T)}+ \norm{\pb-\tp}_{H^2(\T)})+\norm{\ty-\yb_\rM}_{H^{\gamma}(\T)}+\norm{\tp-\pb_\rM}_{H^{\gamma}(\T)}.\label{first}
\end{align}
Lemma \ref{Poincare}$(a)$ and the definition of $\|\bullet \|_{H^2(\T)}$ lead to 
 \begin{align}
     \norm{\yb-\ty}_{H^2(\T)}+ \norm{\pb-\tp}_{H^2(\T)} & \le (1+C_{\rm PJ})(\enorm{\yb-\ty}_\pw  + \enorm{\pb-\tp}_\pw). \label{first_part}
 \end{align}
The inequality in \eqref{NNcal16}, a triangle inequality,  and Lemma \ref{Poincare}$(a)$ yield
 $\enorm{\pb-\tp}_\pw\leq C_{\rm PJ}\norm{\yb-\yb_\rM}  \leq C_{\rm PJ}^2\enorm{\yb-\ty}_\pw +C_{\rm PJ}\norm{\ty-\yb_\rM}.$ Substitute this in \eqref{first_part}, utilize \eqref{NNcal16}, a triangle inequality, the Sobolev embedding $H^{\gamma}(T)\hookrightarrow L^2(T)$, and \eqref{NNcal15} to obtain
\begin{align}
     \norm{\yb-\ty}_{H^2(\T)}+ \norm{\pb-\tp}_{H^2(\T)} 
     &\le C_4(\norm{\bar{\bu}-\bar{\bu}_\h}_{\mathbb{R}^n}  + \norm{\ty-\yb_\rM}_{H^{\gamma}(\T)})\label{c01}
\end{align} 
{with $C_4:=\sqrt{n}C_{\rm PJ}(1+C_{\rm PJ})(1+C_{\rm PJ}^2)$.} The continuous Sobolev embedding $H^{\gamma}(T)\hookrightarrow L^2(T)$, \eqref{inf2gamma}, and \eqref{sup2pw} in \eqref{call4} implies 
\begin{align}
      C_5^{-1}\norm{\bar{\bu}-\bar{\bu}_\h}_{\mathbb{R}^n}\leq \norm{\ty-\yb_\rM}_{H^\gamma(\T)}+ \norm{\tp-\pb_\rM}_{H^{\gamma}(\T)}+h\enorm{\pb_\rM-J\pb_\rM}_\pw \label{call4new}
  \end{align}
 with {$C_5:=C_2(1+C_{\rm emb}+C_{\rm sr}C_{\rm J3})$}. Utilize this in \eqref{c01} to obtain 
    \begin{align}
    \norm{\yb-\ty}_{H^2(\T)} + \norm{\pb-\tp}_{H^2(\T)}
    &\leq C_4(1+C_5)\norm{\ty-\yb_\rM}_{H^{\gamma}(\T)}+ C_4 C_5(\norm{\tp-\pb_\rM}_{H^{\gamma}(\T)}+h\enorm{\pb_\rM-J\pb_\rM}_\pw).\no
    \end{align}
{A substitution of the above estimate in \eqref{first} concludes the proof with $C_{\rm c3}(\alpha^{-1}):= 1+C_{\rm emb1}C_4(1+C_5)$}.
\end{proof}

        \subsection{A~priori and a~posteriori error control}\label{errorap}
        In this section, the a~priori and a~posteriori error estimates for OCP are established. The proofs are based on the error estimates from Section \ref{auxresult} and the a~priori and a~posteriori error estimates for the biharmonic problem with load function in $L^2(\Omega)$ or in $H^{-\gamma}(\Omega)$. 
\subsection*{Biharmonic problem}
 Let $\chi\in V$ and $\chi_\rM\in V_\rM$, respectively, solve the continuous and discrete biharmonic problem stated as:
  \begin{align}\label{intermediate}
         a(\chi,v)=\langle g,v\rangle \text{ for all }  v\in V \text{ and } a_\pw(\chi_\rM,v_\rM)=\langle g,Qv_\rM\rangle \text{ for all } v_\rM\in V_\rM
         \end{align}
       with $Q:=\begin{cases}
            J & \text{ for }g\in H^{-\gamma}(\Omega)\setminus L^2(\Omega),\\
            {\rm id} & \text{ for }g\in L^2(\Omega),
        \end{cases}$ $2-\sigma\leq \gamma\leq 2$,  where $0<\sigma\leq 1$ is the index of elliptic regularity of the biharmonic problem. Note that when $g \in L^2(\Omega)$, $\langle g,v\rangle$ is the $L^2$ inner product. 
        It is well known \cite[Eq. 4.1]{CCNN2021} that  the solution $\chi\in V \cap H^{4-\gamma}(\Omega)$ satisfies $\norm{\chi}_{H^{4-\gamma}(\Omega)}\leq C_{\rm reg} \begin{cases}
            \norm{g}_{H^{-\gamma}(\Omega)}&\text{ for }Q=J,\\
               \norm{g} &\text{ for }Q={\rm id}.
        \end{cases}$  
Moreover, we have the following a~priori error estimates (ref.  \cite[Theorem 4.1, Theorem 3.2(a)]{CCNN2021}). 
    \begin{align}
           C_{\rm ae}^{-1} \enorm{\chi-\chi_\rM}_\pw \leq  \begin{cases}
              h^{2-\gamma} \norm{\chi}_{H^{4-\gamma}(\Omega)}&\text{ for }Q=J,\\
              h^{2-\gamma} \norm{\chi}_{H^{4-\gamma}(\Omega)}+\osc(g,\T) &\text{ for }Q={\rm id}
           \end{cases} 
           \text{ and } \label{M_aprioriHs1}\\
           C_{\rm a\ell}^{-1}\norm{\chi-\chi_\rM}_{H^{\gamma}(\T)}\leq h^{2-\gamma}  \begin{cases}
               \enorm{\chi-\chi_\rM}_\pw & \text{ for }Q=J,\\
               \enorm{\chi-\chi_\rM}_\pw + \osc(g,\T)& \text{ for }Q={\rm id}
           \end{cases}\label{M_aprioriHs}
        \end{align}
with $\osc(g,\T):= \norm{h_{\T}^2(1-\Pi_{2,\T})g}$. 

\medskip
\noindent For given $g\in L^2(\Omega)$, the reliability and efficiency of the a~posteriori error estimator for  \eqref{intermediate} are discussed in \cite[Proposition 2.9]{DG_Morley_Eigen} and are stated below:
       \begin{align}\label{M_apost}
       C_{\rm rel}^{-1}\enorm{\chi-\chi_\rM}_\pw\leq \eta_{\rm aux}(g,\chi_\rM) \leq {C_{\rm eff}\big( \enorm{\chi-\chi_\rM}_\pw+\norm{h_\T^2(1-\Pi_{0,\T})g}\big)}
       \end{align}
where $\displaystyle \eta_{\rm aux}^2(g,\chi_\rM):=\sum_{T\in \T}\big(\eta_{{\rm aux},T}^2(g)+\eta_{{\rm aux},\E(T)}^2(\chi_\rM)\big)$ with $$\eta_{{\rm aux},T}^2(g)=
h_T^4\norm{g}_{L^2(T)}^2\text{ and }\eta_{{\rm aux},\E(T)}^2(\chi_\rM):=\sum_{E\in \E(T)}h_T\norm{[D_\pw^2 \chi_\rM]_E \tau_E}_{L^2(E)}^2. $$ 
For $\displaystyle g=\sum_{{\rm z}\in \Z}\varrho_{\rm z}\delta_{\rm z}$, define the a~posteriori error estimators as {$\displaystyle \eta_{\rm aux}^2(g,\chi_\rM):=\sum_{T\in \T}\big(\eta_{{\rm aux},T}^2(g)+\eta_{{\rm aux},\E(T)}^2(\chi_\rM)\big)$ with }
\begin{align}\label{P3:estmpde}
           { \eta_{{\rm aux}, T}^2(g)}:=	
\begin{cases}
    0 & \text{ for } \Z\cap T=\emptyset,\\
    \sum_{{\rm z}\in \Z \cap T} \lambda_T(\z)h_T^2 |\varrho_\z|^2 & \text{ otherwise},
\end{cases}
{ \eta_{{\rm aux},\E(T)}^2(\chi_\rM)}:= \sum_{E\in \E(T)}h_T\norm{[D_\pw^2 \chi_\rM]_E \tau_E}_{L^2(E)}^2,
          \end{align}
where for any $T\in \T$, $\lambda_T(\z):=\begin{cases}
    0& \text{ for }\z\in {\cal N}(T),\\
    1& \text{ otherwise}.
\end{cases}$ \\

 The reliability and efficiency of the above a~posteriori error estimator are established below. 
{In practice, the initial mesh can always be constructed such that $\Z\subset {\mathcal N}^i(\T_0)$, in which case the volume estimator vanishes. Nevertheless, we establish the reliability and the efficiency of the volume estimator in the general setting. To derive the efficiency estimate, the bubble function technique is adapted following the ideas from \cite[Lemma 3.2]{RER06} and \cite[Theorem 3.7]{DS24}, under the following conditions on the mesh:
\begin{enumerate}[label=({\bf C\arabic*}),leftmargin=3em] 
\item[({\bf C1})] For all ${\rm z}_1, {\rm z}_2 \in \Z\setminus {\mathcal N}^i(\T)$,  the sets $\omega(\T({\rm z}_1))$ and $\omega(\T({\rm z}_2))$ should be disjoint, where $\T({\rm z})$ denotes the set of all triangles which contains ${\rm z}$ and $\omega(\T({\rm z}))$ denotes the union of the patches of $T\in \T({\rm z})$ , 
\item[({\bf C2})] ${\rm dist}(\omega(\T({\rm z})),\partial \Omega)> 0$ for all ${\rm z} \in \Z\setminus {\mathcal N}^i(\T)$.
\end{enumerate}
The strict inequality in the second condition $({\bf C2})$ avoids the steep bubble function for $\z\in\Z$ in the efficiency proof. We start with a lemma that will be needed in the proof of efficiency.}
\begin{lem}\label{eff_bubble}
{Let $E\in \E^i$ be an interior edge that satisfies {$\Z\cap (\omega_E\setminus {\mathcal N}^i(\T))\neq\emptyset$}. Then, it holds
$$\norm{[D_\pw^2 \chi_\rM \nu_E]_E }_{L^2(E)}\leq C_{\rm b} h_E^{-1/2} \norm{D^2(\chi_\rM-\chi)}_{L^2(\omega(\omega_E))},$$
where $C_{\rm b}$ depends only on the shape regularity parameter.}
\end{lem}
\begin{proof} {
Any vector $v:E\rightarrow {\mathbb R}^2$ admits a decomposition into normal and tangential parts as
$v=(v\cdot\tau_E)\tau_E+(v\cdot\nu_E)\nu_E$ and using Cauchy Schwarz inequality we obtain
$$\norm{v}_{L^2(E)}\leq \norm{(v\cdot\tau_E)\tau_E}_{L^2(E)}+\norm{(v\cdot\nu_E)\nu_E}_{L^2(E)}\leq \norm{v\cdot\tau_E}_{L^2(E)}+\norm{v\cdot\nu_E}_{L^2(E)}.$$
Choose $v=[D_\pw^2 \chi_\rM \nu_E]_E$ and denote $[D_{\tau_E\nu_E}^2 \chi_\rM]_E:=[D_\pw^2 \chi_\rM\nu_E]_E\cdot\tau_E$ and $[D_{\nu_E\nu_E}^2 \chi_\rM]_E:=[D_\pw^2 \chi_\rM\nu_E]_E\cdot\nu_E$ to obtain 
\begin{align}
    \norm{[D_\pw^2 \chi_\rM \nu_E]_E }_{L^2(E)}\leq \norm{[D_{\tau_E\nu_E}^2 \chi_\rM]_E}_{L^2(E)}+\norm{[D_{\nu_E\nu_E}^2 \chi_\rM]_E}_{L^2(E)}.\label{2rev1}
\end{align}}

\noindent {\it Step I.}~{Consider the first term in the right-hand side of \eqref{2rev1}. Since $[D_{\tau_E\nu_E}^2 \chi_\rM]_E=D_{\tau_E}([D_{\nu_E} \chi_\rM]_E)$, utilizing the inverse estimate results in 
\begin{align*}
\norm{[D_{\tau_E\nu_E}^2 \chi_\rM]_E}_{L^2(E)}= \norm{D_{\tau_E}[D_{\nu_E} \chi_\rM]_E}_{L^2(E)}\leq C_{\rm inv} h_E^{-1}\norm{[D_{\nu_E} \chi_\rM]_E}_{L^2(E)}.
\end{align*}
As $J\chi_\rM\in H^2_0(\Omega)$, $[D_{\nu_E} J\chi_\rM]_E=0$ for any $E\in \E^i$. 
Recall $h_T\leq \kappa_1 h_E$ from Section \ref{prelim}. These results, the trace inequality, and Lemma \ref{enrithm}$(b)$ in the previous estimate yield
\begin{align}
    \norm{[D_{\tau_E\nu_E}^2 \chi_\rM]_E}_{L^2(E)}&\leq C_{\rm inv} h_E^{-1} \norm{[D_\pw (\chi_\rM-J\chi_\rM)\nu_E]_E}_{L^2(E)}\leq C_{\rm inv} h_E^{-1} \norm{[D_\pw (\chi_\rM-J\chi_\rM)]_E}_{L^2(E)}\no\\
    & \leq \kappa_1 C_{\rm inv} C_{\rm tr} \big(h_T^{-3/2}\norm{D_\pw(\chi_\rM-J\chi_\rM)}_{L^2(\omega_E)}+ h_T^{-1/2}\norm{D_\pw^2(\chi_\rM-J\chi_\rM)}_{L^2(\omega_E)}\big)\no\\
    & \leq C_{\rm e3} h_T^{-1/2}\norm{D^2(\chi_\rM-\chi)}_{L^2(\omega(\omega_E))}
\end{align}
with $C_{\rm e3}:=\kappa_1 C_{\rm inv} C_{\rm tr} \Lambda_{\rm J}$.}

\smallskip

\noindent {\it Step II.}~{Consider the second term in the right-hand side of \eqref{2rev1}. With the conditions ({\bf C1})-({\bf C2}), there will be at most one point load in $\Z\cap\omega_E$ and denote it by $\z_1$. Let $T^{+}$ and $T^{-}$ denote the triangles sharing the edge $E$. The estimate for the first term in the right-hand side of \eqref{2rev1} is presented in two cases: (i) the point source $\z_1\in E\setminus {\mathcal N}^i(\T)$ and (ii) the point source $\z_1\in {\rm int}(T^{+})\cup {\rm int}(T^{-})$.}

\noindent {\it Case (i).}~{For $\z_1\in E\setminus {\mathcal N}^i(\T)$, consider the bubble function $b_E:=h_E[D_{\nu_E\nu_E}^2 \chi_\rM]_E \big(\lambda_3^+-\frac{|T^-|}{|T^+|}\lambda_3^-\big)\lambda_1^2\lambda_2^2\in H^2_0(\omega_E)$, where $\lambda_1$ and $\lambda_2$ denote the barycentric coordinates associated with the nodes of the edge $E$ and $\lambda_3^+$ (resp. $\lambda_3^-$) denote the barycentric coordinates associated with the nodes in $T^+$ (resp. $T^-$) away from the edge $E$. The modified bubble function introduced here is based on the construction of the bubble function from \cite[Theorem 3.7]{DS24}.}
{The equivalence of norms, scaling arguments, and estimates for the integration of the product of barycentric coordinates \cite[Theorem 2.2]{FA18} shows
\begin{align}
   &\norm{[D_{\nu_E\nu_E}^2 \chi_\rM]_E}_{L^2(E)}^2\leq C_{\rm b3} ([D_{\nu_E\nu_E}^2 \chi_\rM]_E,Db_E\nu_E)_{L^2(E)}= C_{\rm b3} ([D_\pw^2 \chi_\rM\nu_E]_E,Db_E)_{L^2(E)}\label{2rev2}\\
   &\text{and }\norm{D^2 b_E}_{L^2(\omega_E)}\leq C_{b4} h_E^{-1/2} \norm{[D_{\nu_E\nu_E}^2 \chi_\rM]_E}_{L^2(E)}.\label{2rev3}
\end{align}
The integration by parts in \eqref{2rev2} and the equation \eqref{intermediate} with $g=\sum_{{\rm z}\in \Z}\varrho_{\rm z}\delta_{\rm z}$ leads to
\begin{align}
   C_{\rm b3}^{-1} \norm{[D_{\nu_E\nu_E}^2 \chi_\rM]_E}_{L^2(E)}^2\leq ([D_\pw^2 \chi_\rM\nu_E]_E,Db_E)_{L^2(E)} = a_\pw(\chi_\rM-\chi,b_E)+ \sum_{{\rm z}\in \Z\cap\omega_E}\varrho_{\rm z}\langle\delta_{\rm z},b_E\rangle.
\end{align}
For any $\z\in E\cup {\mathcal N}^i(\omega_E)$, the construction of the bubble function reveals $b_E(\z)=0$ and hence $\langle \delta_\z,b_E\rangle=0$. Utilizing this, the continuity of the bilinear form $a_\pw(\cdot,\cdot)$, $\text{support}(b_E)=\omega_E$, and \eqref{2rev3} in the above estimate show
\begin{align}
    \norm{[D_{\nu_E\nu_E}^2 \chi_\rM]_E}_{L^2(E)}^2 &\leq C_{\rm b3}\norm{D^2_\pw(\chi_\rM-\chi)}_{L^2(\omega_E)}\norm{D^2b_E}_{L^2(\omega_E)}\no\\
    & \leq C_{\rm b3}C_{\rm b4} h_E^{-1/2}\norm{D^2_\pw(\chi_\rM-\chi)}_{L^2(\omega_E)} \norm{[D_{\nu_E\nu_E}^2 \chi_\rM]_E}_{L^2(E)}.\label{2rev4} 
\end{align}
This implies $$\norm{[D_{\nu_E\nu_E}^2 \chi_\rM]_E}_{L^2(E)}\leq C_{\rm b3}C_{\rm b4} h_E^{-1/2}\norm{D^2_\pw(\chi_\rM-\chi)}_{L^2(\omega_E)}.$$
\noindent {\it Case (ii).}~For $\z_1\in {\rm int}(T^{+})\cup {\rm int}(T^{-})$, consider the bubble function $b_{E}(x):=h_E^{-1/2}[D_{\nu_E\nu_E}^2 \chi_\rM]_E \norm{x-\z_1}_{{\mathbb R}^2}^2$ $\big(\lambda_3^+(x)-\frac{|T^-|}{|T^+|}\lambda_3^-(x)\big)\lambda_1^2(x)\lambda_2^2(x)$ which vanishes at $\z_1$.  The idea of construction of this bubble function is inspired by \cite[Lemma 3.7]{HYNP25} and \cite[Theorem 3.7]{DS24}. The rest of the proof follows analogous as in the previous case and hence is skipped.}
\end{proof}

\begin{lem}[a posteriori error control]\label{bih_apost}
Let $\chi\in V$ and $\chi_\rM\in  V_\rM $ solve $\eqref{intermediate}$ with $g=\sum_{\z\in \Z}\varrho_\z\delta_\z$, respectively. Then there exist  $C_{\rm rel}, C_{\rm eff}>0$ such that
			$	 C_{\rm rel}^{-2}\enorm{\chi -\chi_\rM}_{\pw}^2 \leq  \eta_{\rm aux}^2(g,\chi_\rM) 
     \leq  C_{\rm eff}^2\enorm{\chi-\chi_\rM}_{\pw}^2.$
\end{lem}
    \begin{proof}
         Denote {$\T_{\rm d}(\T):=\{T\in \T:{\rm z}\in T \text{ for any }{\rm z} \in \Z\setminus{\mathcal N}^i(\T)\}$ with $|\T_{\rm d}(\T)|\leq 2n$.} 
         
\noindent {\it Step 1.}~(reliability) 
The residual of \eqref{intermediate} with load $g=\sum_{\z\in \Z}\varrho_\z\delta_\z$  and the test function $v\in V$ is defined by ${\rm Res}_\pw(v):=a_\pw(\chi-\chi_\rM,v).$ The definitions of the $\pw$-norm and the residual, together with Young's inequality, lead to
	\begin{align}
	    \enorm{\chi-\chi_\rM}_\pw^2
     &={\rm Res}_\pw(\chi-v)+a_\pw(\chi-\chi_\rM,v-\chi_\rM)\no\\
     &\leq {\rm Res}_\pw(\chi-v)+\frac{1}{4}\enorm{\chi-\chi_\rM}_{\pw}^2+\enorm{v-\chi_\rM}_{\pw}^2. \label{res}
	\end{align}
 Define $w:=\chi-v \in V$. The equation \eqref{intermediate} with test function $I_\rM w$ and the orthogonality, $a_\pw(\chi_\rM,w-I_\rM w)=0$ from Lemma \ref{Interpolation}$(a)$ yield  
 \begin{align}
	{\rm Res}_\pw(w)=\sum_{{\rm z}\in\Z}\varrho_\z\langle\delta_{\rm z},w\rangle-a_\pw(\chi_\rM,w)=\sum_{{\rm z}\in\Z} \varrho_\z\langle\delta_{\rm z},w-JI_\rM w\rangle =\sum_{{\rm z}\in\Z}\varrho_\z(w-JI_\rM w)({\rm z}).\label{call100}
\end{align}
The bound for $\enorm{\chi-\chi_\rM}_\pw$ 
is achieved in the following two cases.

\medskip
\noindent {\it Case I.}~($\z\in {\cal N}^i(\T)$)~For ${\rm z}$ in the nodes of the triangulation, the definition of the interpolation operator and the companion operator reveals \cite[Section 2.2 and Section 2.3]{DG_Morley_Eigen}
\begin{align}
(w-JI_\rM w)({\rm z})=0\fl w\in V.\label{res01}
\end{align}
In the case of $\Z\subseteq{\cal N}^i(\T)$, ${\rm Res}_\pw(\chi-v)=0$. Substitute this and utilize Lemma \ref{enrithm}$(b)$ in \eqref{res} to obtain 
{ \begin{align*}
     \enorm{\chi-\chi_\rM}_{\pw}^2 \leq \frac{4}{3}\min_{v\in V}\enorm{v-\chi_\rM}_{\pw}^2\leq \frac{4}{3}C_{\rm J1}^2\Lambda_{\rm J}^{-2}\sum_{T\in \T}\sum_{E\in \E(T)}h_T\norm{[D_\pw^2 \chi_\rM]_E \tau_E}_{L^2(E)}^2\leq C_{\rm rel}^2\eta_{\rm aux}^2(g,\chi_\rM) 
 \end{align*}}
 with $C_{\rm rel}:=3^{-1/2}2C_{\rm J1}\Lambda_{\rm J}^{-1}$. 

\medskip
\noindent {\it Case II.}~($\z\notin {\cal N}^i(\T)$)~The triangle inequality shows 
 \begin{align*} 
 |(w-JI_\rM w)({\rm z})|\leq\norm{w-JI_\rM w}_{L^{\infty}(T)}\leq\norm{w-I_\rM w}_{L^{\infty}(T)}+\norm{I_\rM w-JI_\rM w}_{L^{\infty}(T)}.
 \end{align*}
 Let $T_{\rm ref}$ denote the reference triangle. The scaling argument (used twice) \cite[Theorem 3.1.2]{Ciarlet} and the Bramble-Hilbert Lemma \cite[Theorem 4.1.3]{Ciarlet} yield 
 \begin{align*}
 \norm{w-I_\rM w}_{L^{\infty}(T)}\leq C_{\rm sc1} \norm{w_{\rm ref}-I_\rM w_{\rm ref}}_{L^{\infty}(T_{\rm ref})}\leq C_{\rm sc1}C_{\rm bh} \norm{D^2 w_{\rm ref}}_{L^2(T_{\rm ref})}\leq C_{\rm sc1}C_{\rm bh}C_{\rm sc2} h_T\norm{D^2 w}_{L^2(T)}
 \end{align*}
 where $C_{\rm sc1}$ and $C_{\rm sc2}$ are scaling constants and $C_{\rm bh}$ is the constant in the Bramble-Hilbert lemma, which are independent of mesh-size. The inequality in \eqref{sup2pw}, Lemma \ref{Interpolation}$(a)$, and the stability of the projection operator $\Pi_{0,\T}$ \cite[Lemma 1.131]{ErnJLU_2004} imply $$\norm{I_\rM w-JI_\rM w}_{L^{\infty}(T)}\leq C_{\rm J3} h_T\norm{D^2_\pw I_\rM w}_{L^2(\omega_T)}\leq C_{\rm J3}h_T\norm{D^2 w}_{L^2(\omega_T)}.$$ 
 A combination of last three estimates reveals $|(w-JI_\rM w)({\rm z})|\leq C_{\rm rel1} h_T\norm{D^2 w}_{L^2(\omega_T)}$ with $C_{\rm rel1}:=\max\{C_{\rm sc1}C_{\rm bh}C_{\rm sc2},C_{\rm J3}\}$.
 
 Utilizing these results in \eqref{call100}, we obtain 
 $${\rm Res}_\pw(\chi-v)\leq C_{\rm rel1} \sum_{\substack{\z \in \Z \cap T \\
\z \notin \mathcal{N}^i(\T)}} |\varrho_{\rm z}|h_T\norm{D^2(\chi-v)}_{L^2(\omega_T)}.$$
  This, a triangle inequality, and Young's inequality (used twice) yield
	\begin{align}
	  {\rm Res}_\pw(\chi-v)&
      \leq C_{\rm rel1} \sum_{\substack{\z \in \Z \cap {T} \\
\z \notin \mathcal{N}^i(\T)}}|\varrho_{\rm z}| h_T\big (\norm{D_\pw^2(\chi-\chi_\rM)}_{L^2(\omega_T)}+\norm{D_\pw^2(\chi_\rM-v)}_{L^2(\omega_T)}\big )\nonumber \\
	    &\leq \frac{1}{4} \enorm{\chi-\chi_\rM}_{\pw}^2+\frac{3}{2} C_{\rm rel1}^2 \sum_{\substack{\z \in \Z \cap {T} \\
\z \notin \mathcal{N}^i(\T)}} \varrho_{\rm z}^2h_T^2+\frac{1}{2}\enorm{v-\chi_\rM}_{\pw}^2.\label{res2}
	\end{align}
{Notice that, utilizing the conditions ({\bf C1})-({\bf C2}) in \eqref{res2}, we obtain the constant $C_{\rm rel1}$ independent of the cardinality of the set $\Z$.} {A substitution of \eqref{res2} in \eqref{res} and utilization of Lemma \ref{enrithm}$(b)$ show}
	\begin{align}
	    \frac{1}{2}\enorm{\chi-\chi_\rM}_{\pw}^2&\leq \frac{3}{2} C_{\rm rel1}^2 \sum_{\substack{\z \in \Z \cap {T} \\
\z \notin \mathcal{N}^i(\T)}}\varrho_{\rm z}^2h_T^2+\frac{3}{2}\min_{v\in V}\enorm{v-\chi_\rM}_{\pw}^2\no\\
     &\leq \frac{3}{2}C_{\rm rel1}^2 \sum_{\substack{\z \in \Z \cap {T} \\
\z \notin \mathcal{N}^i(\T)}} \varrho_{\rm z}^2h_T^2+\frac{3}{2}C_{\rm J1}^2\Lambda_{\rm J}^{-2}\sum_{T\in \T}\sum_{E\in \E(T)}h_T\norm{[D_\pw^2 \chi_\rM]_E \tau_E}_{L^2(E)}^2.\label{res1}
	\end{align}
This yields, $\enorm{\chi-\chi_\rM}_{\pw}\leq C_{\rm rel} \eta_{\rm aux}(g,\chi_\rM)$ with {$C_{\rm rel}^2:= \max\{3C_{\rm rel1}^2,3C_{\rm J1}^2\Lambda_{\rm J}^{-2}\}$.}

\medskip
\noindent {\it Step 2.}~(efficiency)~ The ideas are inspired by \cite[Lemma 3.2]{RER06}. For ${\rm z}\in T\in \T_{\rm d}(\T)$, denote $\omega_T:=\bigcup\{K\in \T:K\cap T\neq \emptyset\}$ and $d:={\rm dist}({\rm z},\partial \omega_T)$. Consider the bubble function $b_{\rm z}\in V$ satisfying $0\leq b_{\rm z}(x)\leq 1$ for all $\Omega$, $b_{\rm z}(x)=1$ in $|x-{\rm z}|\leq d/4$, $b_{\rm z}(x)=0$ in $|x-{\rm z}|\geq 3d/4$, and $|b_{\rm z}|_{W^{m,\infty}(\omega_T)}\leq C_{\rm b1} d^{-m}$ for $m=1,2$. It also satisfies $|b_{\rm z}|_{H^m(\omega_T)}\leq C_{\rm b2} h_T^{1-m}$. Here, the constants $C_{\rm b1}$ and $C_{\rm b2}$ only depend on the shape regularity parameter.
{To observe the constructions of the bubble functions for different locations of $\z$, refer to \cite[Figure 2]{RER06}. Unlike the first two figures shown there, in the third figure, a steep bubble function may occur, and then the value of $d$ can be very small. The condition  {\bf (C2)} avoids this steep bubble function and this ensures $h_T\lesssim d$.}

The above properties of the bubble function and \eqref{intermediate} yield
 \begin{align}
     \varrho_{\rm z}=\varrho_{\rm z}\langle \delta_{\rm z},b_{\rm z}\rangle= \sum_{\z\in \Z}\varrho_{\rm z}\langle \delta_{\rm z},b_{\rm z}\rangle&=  \int_\Omega D_\pw^2(\chi-\chi_\rM):D^2 b_{\rm z}\dx +\int_\Omega D_\pw^2\chi_\rM:D^2b_{\rm z}\dx.\label{eff2}
     \end{align}
     The Cauchy-Schwarz inequality and the estimate $|b_{\rm z}|_{H^2(\omega_T)}\leq C_{\rm b2} h_T^{-1}$ show
     \begin{align}
     \int_\Omega D_\pw^2(\chi-\chi_\rM):D^2 b_{\rm z}\dx\leq \enorm{\chi-\chi_\rM}_\pw \enorm{b_{\rm z}}_\pw  \leq C_{\rm b2} h_T^{-1}\enorm{\chi-\chi_\rM}_\pw.\label{sn7}
 \end{align}
Utilizing the integration by parts in the second term of \eqref{eff2}, we obtain
\begin{align}\label{eff3}
   \sum_{T\in \T} \int_T D_\pw^2\chi_\rM:D^2b_{\rm z}\dx={\sum_{E\in \E^i(\omega_T)}}\int_{E}[D_\pw^2\chi_\rM \nu_E]_E Db_{\rm z}{\rm ds}. 
\end{align}
The Cauchy-Schwarz inequality, trace inequality \cite[Eq. 10.3.8]{Brenner}, and the boundedness property of the bubble function shows
{\begin{align}
   {\int_{E}[D_\pw^2\chi_\rM\cdot \nu_E]_E Db_{\rm z}{\rm ds}}
   &\leq C_{\rm tr}\norm{[D_\pw^2\chi_\rM \nu_E]_E }_{L^2(E)}(h_T^{-1/2} \norm{Db_{\rm z}}_{L^2(T)}+h_T^{1/2}\norm{D^2b_{\rm z}}_{L^2(T)})\no\\
   & \leq 2C_{\rm tr} C_{\rm b2} h_T^{-1/2}\norm{[D_\pw^2 \chi_\rM\nu_E]_E }_{L^2(E)}\leq C_{\rm e} h_T^{-1} \norm{D^2_\pw(\chi_\rM-\chi)}_{L^2(\omega_E)},\no
   \end{align}
where the last step utilizes the estimate from Lemma \ref{eff_bubble} with $C_{\rm e}:=2C_{\rm tr} C_{\rm b2}\kappa_1^{1/2} C_{\rm b}$.
Substituting this in \eqref{eff3} bounds the second term on the right side of \eqref{eff2}. This with \eqref{sn7} in \eqref{eff2} show
 $$    |\varrho_{\rm z}|=  |\varrho_{\rm z}\langle \delta_{\rm z},b_{\rm z}\rangle| 
      \leq (C_{\rm b2}+C_{\rm e}) {\sum_{E\in \E^i(\omega_T)}}h_T^{-1} \norm{D^2_\pw(\chi_\rM-\chi)}_{L^2(\omega(\omega_E))}.$$
      A summation over ${\rm z}\in\Z$ in the above  estimates and Lemma \ref{enrithm}$(b)$ for the efficiency of the edge estimator in \eqref{P3:estmpde} yield
  $ \eta_{\rm aux}\leq C_{\rm eff}\enorm{\chi-\chi_\rM}_{\pw}$
 with $C_{\rm eff}:= C_{\rm sr}(C_{\rm b2}+C_{\rm e})+C_{\rm J2}C_{\rm J1}^{-1}$.} For $\Z\subseteq {\cal N}^i(\T)$, the volume estimator vanishes. Hence, utilizing Lemma \ref{enrithm}$(b)$ for the edge estimator implies the efficiency of $\eta_{\rm aux}$ with {$C_{\rm eff}:=C_{\rm J2}C_{\rm J1}^{-1}$.}
    \end{proof}

\begin{remark}[condition on initial mesh]
{For every source point ${\rm z}\in \Z\cap T$, the efficiency of the volume estimator $\eta_{{\rm aux},T}$ requires the bubble function  whose support is a subset of $\omega(\T({\rm z}))$. Hence, the first equality in \eqref{eff2} requires the assumptions {\bf (C1)}-{\bf (C2)} considered on the triangulation. }

{In case of ${\Z}\subseteq {\mathcal N}^i(\T_0)$, {the volume estimator $\eta_{{\rm aux},T}$ vanishes,} so there is no restriction on the mesh-size. For a given the set $\Z$, constructing the initial mesh so that ${\Z}\subseteq {\mathcal N}^i(\T_0)$ is a reasonable assumption in practice. Otherwise, the initial triangulation need to satisfy the assumptions {\bf (C1)}-{\bf (C2)}. Considering a suitable initial triangulation is not very restrictive as the set $\omega(\T({\rm z}))$ contains a patch of one triangle if ${\rm z}$ is inside a triangle or a patch of two triangles when ${\rm z}$ is on an interior edge.} 
\end{remark}

\subsection*{Error control for OCP}
  \begin{thm}[a~priori error control]\label{apriori}
   The solutions $(\yb,\pb,{\bar\bu})\in (V\cap H^{4-\gamma}(\Omega))^2\times U_{\ad}$ to \eqref{sym1} {with $2-\sigma< \gamma \leq 2$} 
   and $(\yb_\rM,\pb_\rM,\bar{\bu}_\h)\in V_\rM\times V_\rM\times U_{\ad}$ to \eqref{discrete1} satisfy 
    \begin{align*}
      (a) & \norm{\bar{\bu}-\bar{\bu}_\h}_{\mathbb{R}^n}+\enorm{\yb-\yb_\rM}_\pw+\enorm{\pb-\pb_\rM}_\pw\leq C_{\rm AE}h^{2-\gamma}\big(\max\{\norm{{\bf a}}_{L^\infty({\mathbb R}^n)},\norm{{\bf b}}_{L^\infty({\mathbb R}^n)}\}+\norm{y_{\rm d}}\big) , \\
    (b) & \norm{\bar{\bu}-\bar{\bu}_\h}_{\mathbb{R}^n}+\norm{\yb-\yb_\rM}_{H^{\gamma}(\T)}+\norm{\pb-\pb_\rM}_{H^{\gamma}(\T)}\no\\
      &\quad\leq C_{\rm AL} h^{2-\gamma}\big( \norm{\bar{\bu}-\bar{\bu}_\h}_{\mathbb{R}^n}+\enorm{\yb-\yb_\rM}_\pw+\enorm{\pb-\pb_\rM}_\pw+h^2\big(\max\{\norm{{\bf a}}_{L^\infty({\mathbb R}^n)},\norm{{\bf b}}_{L^\infty({\mathbb R}^n)}\}+\norm{y_{\rm d}}\big)\big).
    \end{align*}
    \end{thm}
    \begin{proof}[Proof of (a)]
     The error equivalence result in Theorem~\ref{equivalence} (see \eqref{equiv1}), and \eqref{M_aprioriHs1} for the state and adjoint equations with $g=\sum_{{\rm z}\in \Z}\ub_{\rm z,h}\delta_{\rm z}$ and $g=\yb_\rM-\yd$, respectively, reveal
        \begin{align}
          \norm{\bar{\bu}-\bar{\bu}_\h}_{\mathbb{R}^n}+\enorm{\yb-\yb_\rM}_\pw &+\enorm{\pb-\pb_\rM}_\pw \leq C_{\rm c1} \big(\enorm{\ty-\yb_\rM}_\pw+\enorm{\tp-\pb_\rM}_\pw \big) \no\\
          &\leq C_{\rm c1}C_{\rm ae}\big( h^{2-\gamma} (\norm{\ty}_{H^{4-\gamma}(\Omega)} +\norm{\tp}_{H^{4-\gamma}(\Omega)}) +\osc(\yb_\rM-\yd,\T)\big).  \label{0cal11}
        \end{align}
       { Lemma \ref{Poincare}$(a)$ (applied twice),  \eqref{state_eq1} with test function $\yb_\rM$, and Lemma \ref{enrithm}$(a)$ yield $\norm{\yb_\rM}\leq C_{\rm PJ}\enorm{\yb_\rM}_{\pw}$ and}
        \begin{align}
        \enorm{\yb_\rM}_{\pw}^2&\leq n\max\{\norm{{\bf a}}_{L^\infty({\mathbb R}^n)},\norm{{\bf b}}_{L^\infty({\mathbb R}^n)}\}\norm{J\yb_\rM}_{L^\infty(\Omega)}\no\\
        &\leq  nC_{\rm PJ}\max\{\norm{{\bf a}}_{L^\infty({\mathbb R}^n)},\norm{{\bf b}}_{L^\infty({\mathbb R}^n)}\}\enorm{J\yb_\rM}_\pw \leq  nC_{\rm PJ}C_{\rm J} \max\{\norm{{\bf a}}_{L^\infty({\mathbb R}^n)},\norm{{\bf b}}_{L^\infty({\mathbb R}^n)}\}\enorm{\yb_\rM}_\pw, \no
        \end{align}
       where $n$ is the cardinality of $\Z$. This with {$C_{\rm ap0}^2:={n}C_{\rm PJ}^{3} C_{\rm J}$} implies
       \begin{align}
           \norm{\yb_\rM}\leq C_{\rm ap0}\max\{\norm{{\bf a}}_{L^\infty({\mathbb R}^n)},\norm{{\bf b}}_{L^\infty({\mathbb R}^n)}\}.\label{call5}
       \end{align}
The regularity result for the solution in  \eqref{intermediate} with $g=\sum_{{\rm z}\in \Z}\ub_{\rm z,h}\delta_{\rm z}$ (resp. $\yb_\rM-y_{\rm d}$) plus \eqref{call5} show 
\begin{align*}
&\norm{\ty}_{H^{4-\gamma}(\Omega)}\leq {nC_{\rm reg} C_{\delta}} \max\{\norm{{\bf a}}_{L^\infty({\mathbb R}^n)},\norm{{\bf b}}_{L^\infty({\mathbb R}^n)}\} \\
\big(\text{resp. }&\norm{\tp}_{H^{4-\gamma}(\Omega)}\leq C_{\rm reg} \norm{\yb_\rM-y_{\rm d}}\leq C_{\rm reg}(C_{\rm ap0}+1) \big(\max\{\norm{{\bf a}}_{L^\infty({\mathbb R}^n)},\norm{{\bf b}}_{L^\infty({\mathbb R}^n)}\}+\norm{y_{\rm d}}\big)\big),
\end{align*}
{where $C_{\delta}:=\max\{\norm{\delta_\z}_{H^{-\gamma}(\Omega)}:\z\in\Z\}$.} To control the third term on the right-hand side of \eqref{0cal11}, utilize the continuity of the projection operator $\Pi_{2,\T}$, a triangle inequality, and \eqref{call5} to obtain
\begin{align}\label{osc_bound}
\osc(\yb_\rM-\yd,\T)\leq h^{2}\norm{\yb_\rM-\yd}\leq h^{2}(C_{\rm ap0}+1) \big(\max\{\norm{{\bf a}}_{L^\infty({\mathbb R}^n)},\norm{{\bf b}}_{L^\infty({\mathbb R}^n)}\}+\norm{y_{\rm d}}\big).
\end{align}
A substitution of all this in \eqref{0cal11} {and $h\leq |\Omega|^{1/2}$ concludes the proof of $(a)$ with $C_{\rm AE}(\alpha^{-1}):= C_{\rm c1}C_{\rm ae}(nC_{\rm reg}C_{\delta}+(C_{\rm ap0}+1)(C_{\rm reg}+|\Omega|^{\gamma/2}))$.} 

\medskip
\noindent {\it Proof of (b).}~ 
The lower order a~priori error estimate for OCP is established here. The inequality in \eqref{call4new}, the lower order error bounds in Theorem~\ref{aprioriHs}, \eqref{M_aprioriHs}, and Lemma \ref{enrithm}$(b)$ lead to 
\begin{align}
\norm{\bar{\bu}-\bar{\bu}_\h}_{\mathbb{R}^n}+\norm{\yb-\yb_\rM}_{H^{\gamma}(\T)} &+\norm{\pb-\pb_\rM}_{H^{\gamma}(\T)} \no\\
&\leq (C_5+C_{\rm c3}) \big(\norm{\ty-\yb_\rM}_{H^{\gamma}(\T)}+h\enorm{\pb_\rM-J\pb_\rM}_\pw+\norm{\tp-\pb_\rM}_{H^{\gamma}(\T)}\big)\no \\
&\leq C_{\rm ap} h^{2-\gamma}\big(\enorm{\ty-\yb_\rM}_\pw+\enorm{\tp-\pb_\rM}_\pw+\osc(\yb_\rM-y_{\rm d},\T)\big)\no 
\end{align} 
{with $C_{\rm ap}:=(C_5+C_{\rm c3})\max\{C_{\rm a\ell}, |\Omega|^{(\gamma-1)/2}C_{\rm sr}\Lambda_{\rm J}\}$.} This and Theorem~\ref{equivalence} with $C_{\rm AP}:=(C_{\rm c2}+1)C_{\rm ap}$ show
\begin{align}\label{A4apriori}
  \norm{\bar{\bu}-\bar{\bu}_\h}_{\mathbb{R}^n}+ \norm{\yb-\yb_\rM}_{H^{\gamma}(\T)} &+\norm{\pb-\pb_\rM}_{H^{\gamma}(\T)} \\
  &\leq C_{\rm AP} h^{2-\gamma}\big(\norm{\bar{\bu}-\bar{\bu}_\h}_{\mathbb{R}^n} +\enorm{\yb-\yb_\rM}_\pw +\enorm{\pb-\pb_\rM}_\pw+\osc(\yb_\rM-y_{\rm d},\T)\big)\no
\end{align} 
A combination with \eqref{osc_bound} concludes the proof of $(b)$ with $C_{\rm AL}(\alpha^{-1}):=C_{\rm AP}(C_{\rm ap0}+2)$.
  \end{proof}
\noindent Now, the a~posteriori error estimators are introduced for \eqref{opt1}, and their reliability and efficiency estimates are derived. 
\begin{tcolorbox}[colback=gray!5, colframe=black, boxrule=0.4pt, arc=2pt,
  left=6pt, right=6pt, top=6pt, bottom=6pt]
\begin{subequations}\label{estmSA}
The {\it complete error estimator}  is defined as
$\displaystyle  \eta^2: =  {  \sum_{T \in \T}{\eta^2(T)} }=\eta_S^2+\eta_A^2. $
Here, the error estimators for the state and adjoint equations are 
 \begin{align}\label{esti_com}
 \eta_S^2:=\sum_{T\in \T}(\eta_{S,T}^2+\eta_{S,\E(T)}^2) \text{ and } \eta_A^2:=\sum_{T\in \T}(\eta_{A,T}^2+\eta_{A,\E(T)}^2),
 \end{align}
 where for any $T\in \T$,
\begin{align}
\eta_{S,T}^2 := 
\begin{cases}
    0 & \text{ for }  \Z \cap T = \emptyset, \\
    \displaystyle\sum_{{\rm z}\in \Z \cap T} \lambda_T(\z)\, h_T^2\, |\ub_{{\rm z},h}|^2 & \text{ otherwise },
\end{cases}
 &\quad
\eta_{S,\E(T)}^2 := \sum_{E \in \E(T)} h_T \norm{[D_{\rm pw}^2 \yb_\rM]_E \tau_E}_{L^2(E)}^2, \\
\eta_{A,T}^2 := h_T^4 \norm{\yb_\rM - \yd}_{L^2(T)}^2,\hspace{2.7cm} &\quad
\eta_{A,\E(T)}^2 := \sum_{E \in \E(T)} h_T \norm{[D_{\rm pw}^2 \pb_\rM]_E \tau_E}_{L^2(E)}^2,
\end{align}
\end{subequations}
and for any $T\in \T$, $\lambda_T:\Z\rightarrow \{0,1\}$ takes value 0 when $\z\in {\cal N}(T)$, and 1 otherwise. In the sequel,  the element estimator in a triangle $T$ is denoted as
\begin{align}
\displaystyle \eta^2(T):=\eta_{S,T}^2+\eta_{S,\E(T)}^2+\eta_{A,T}^2+\eta_{A,\E(T)}^2.\label{elemest}
  \end{align}
\end{tcolorbox}

\begin{thm}[a~posteriori error control]\label{aposteriori}
			Let $(\yb ,\pb,\bar{\bu} ) \in V\times V\times U_{ad}$ and $ ( \yb_\rM,\pb_\rM,\bar{\bu}_\h)\in  V_\rM \times V_\rM\times U_{\rm ad}$ solve $\eqref{sym1}$ and $\eqref{discrete1}$, respectively. Then there exist constants $C_{\rm REL}, C_{\rm EFF}>0$ such that
			\begin{align*}
				 &C_{\rm REL}^{-2}(\norm{\bar{\bu}-\bar{\bu}_\h}_{\mathbb{R}^n}^2+\enorm{\yb -\yb_\rM}_{\pw}^2+\enorm{\pb -\pb_\rM}_{\pw}^2) \leq  \eta^2(\yb_\rM,\pb_\rM,\bar{\bu}_\h,\yd) \\
     &\leq  C_{\rm EFF}^2\big( \norm{\bar{\bu}-\bar{\bu}_\h}_{\mathbb{R}^n}^2+\enorm{\yb -\yb_\rM}_{\pw}^2 + \enorm{\pb-\pb_\rM}_{\pw}^2 +h^{4}\max\{\norm{{\bf a}}_{L^\infty({\mathbb R}^n)}^2,\norm{{\bf b}}_{L^\infty({\mathbb R}^n)}^2\}+\norm{h_\T^2(1-\Pi_{0,\T})\yd}^2\big). 
			\end{align*}
			\end{thm}
    
\begin{proof}
The error equivalence from Theorem~\ref{equivalence} shows 
    \begin{align}
     \norm{\bar{\bu}-\bar{\bu}_\h}_{\mathbb{R}^n}^2+\enorm{\yb -\yb_\rM}_{\pw}^2+\enorm{\pb -\pb_\rM}_{\pw}^2  \approx \enorm{\ty -\yb_\rM}_{\pw}^2+\enorm{\tp-\pb_\rM}_{\pw}^2 \label{eq}
    \end{align}
    with the constants from \eqref{equiv1} and \eqref{Cc2}. 

\noindent Employ the reliability and efficiency estimates in Lemma \ref{bih_apost} with $g=\sum_{\z\in \Z}\ub_{\rm z,h}\delta_\z$ in the first term of the right-hand side of \eqref{eq} to obtain
$$C_{\rm rel}^{-2}\enorm{\ty-\yb_\rM}_{\pw}^2\leq \eta_S^2\leq C_{\rm eff}^2\enorm{\ty-\yb_\rM}_{\pw}^2.$$ 
For the second term in the right-hand side of \eqref{eq}, \eqref{M_apost} with $g=\yb_\rM-\yd$ yields
\begin{align}
    C_{\rm rel}^{-2}\enorm{\tp-\pb_\rM}_{\pw}^2\leq \eta_A^2\leq C_{\rm eff}^2\big( \enorm{\tp-\pb_\rM}_\pw+\norm{h_\T^2(1-\Pi_{0,\T})(\yb_\rM-\yd)}\big)^2.\label{newref001}
\end{align}
A triangle inequality and  \eqref{call5} reveal $$\norm{h_\T^2(1-\Pi_{0,\T})(\yb_\rM-\yd)}^2\leq 2C_{\rm ap0}^2h^{4}\max\{\norm{{\bf a}}_{L^\infty({\mathbb R}^n)}^2,\norm{{\bf b}}_{L^\infty({\mathbb R}^n)}^2\}+2\norm{h_\T^2(1-\Pi_{0,\T})\yd}^2.$$ 
Substituting these estimates in \eqref{eq} proves the reliability and efficiency of the a~posteriori error estimator with $C_{\rm REL}^2(\alpha^{-1}):=2C_{\rm c1}^{2}C_{\rm rel}^2$ and $C_{\rm EFF}^2(\alpha^{-1}):=12C_{\rm c2}^{2}C_{\rm eff}^2(C_{\rm ap0}^2+1)$.
  \end{proof}

\begin{remark}[alternate approach 1]
    The proofs of reliability and efficiency of the a~posteriori error estimator are based on the reliability and efficiency of the a~posteriori error estimator for the biharmonic equation with point sources.
    Here, another a~posteriori error estimator is discussed based on the characterization of $\delta_{\rm z}\in H^{-2}(\Omega)$. Without loss of generality, we consider $\Z=\{\rm z\}$. There is a non-unique representation of $\delta_{\rm z}\in H^{-2}(\Omega)$ for $\phi\in H_0^2(\Omega)$ (ref \cite[Theorem 1, p. 23]{Evans})
    $\langle \delta_{\rm z}, \phi\rangle= \sum_{|\alpha|\leq 2}(f_\alpha,\partial^{\alpha} \phi)$
    with $f_\alpha\in L^2(\Omega)$ for all $|\alpha|\leq 2$. This representation also leads to a reliable and efficient a~posteriori error estimator (ref. \cite[Section 7.2]{CCBGNN24}).
{Since we do not have specific $f_\alpha$ for $|\alpha|\leq 2$,} this a~posteriori error estimator is difficult to compute in practice. 
\end{remark}
\begin{remark}[alternate approach 2]
    We have a specific representation of $\delta_{\rm z}$ as
    $\langle \delta_{\rm z}, \phi\rangle=\frac{1}{2\pi}\int_{\Omega}\log|x-{\rm z}|\Delta \phi (x)\dx$ for all $\phi\in V$ \cite[Section 5.2]{CCNN2021}. For every ${\rm z}\in \Z$, define $f_{\rm z}(x):=\frac{1}{2\pi}\log|x-{\rm z}|$. By replacing $\eta_{S,T}$ in \eqref{estmSA} with $\sum_{{\rm z}\in \Z}|\ub_{\rm z,h}|^2\norm{f_{\rm z}}_{L^2(T)}^2$ provides the a~posteriori error estimator corresponding to this representation. Although we have the reliability and efficiency of the a~posteriori error estimator, proving estimator reduction is challenging due to the absence of $h_T$ in the coefficient.
\end{remark}
\section{Quasi-optimality of AFEM for {\bf (P1)}}\label{AOA}
 Let $\widehat{\T}\in \mathbb{T}(\T)$ be a shape regular refinement of $\T$, and {$(\widehat{\yb}_\rM,\widehat{\pb}_\rM,\widehat{\bar{{\bu}}}_{\rm h})$ and $(\yb_\rM,\pb_\rM,\bar{{\bu}}_\h)$} be the solutions of \eqref{discrete1} corresponding to the triangulations $\widehat{\T}$ and $\T$, respectively. Let $\heta$ and $\eta$ (resp. $\widehat{\mu}$ and ${\mu}$) be the corresponding complete error estimator (resp. volume estimator) in \eqref{esti_com}  with contributions from state, adjoint, and control variables with respect to the triangulations $\widehat{\T}$ and $\T$. Similarly, denote the complete error estimator (resp. volume estimator) with respect to the triangulation $\T_k$ as $\eta_k$ (resp. $\mu_k$) for $k\in \mathbb N_0$.
 The distance function is defined as 
$$\dd(\T,\widehat{\T}):=\big(\norm{\widehat{\bar{\bu}}_\h-\bar{\bu}_\h}_{\mathbb{R}^n}^2+\enorm{\widehat{\yb}_\rM-\yb_\rM}_{\pw}^2 +\enorm{\widehat{\pb}_\rM-\pb_\rM}_{\pw}^2\big)^{1/2}.$$ 
The axioms of adaptivity \cite{CMMD2014,CCRH17} {\bf (A1)}-{\bf(A4)} guarantee the quasi-optimal convergence of the adaptive FEM. {\it Stability} axiom {\bf (A1)} establishes that the difference between the estimators associated with the triangulations $\widehat{\T}$ and $\T$ over the common triangles in $\widehat{\T} \cap \T$ is bounded by the distance function. {\it Reduction} axiom {\bf (A2)} proves that up to the distance function, the estimator for the triangulation $\widehat{\T}$ over the refined triangles is bounded by a fixed fraction (referred to as the reduction parameter) multiplied by the complete estimator at the $\T$ level for elements in $\T$ that have been refined. The {\it discrete reliability} {\bf (A3)} proves that the distance function is bounded by the complete estimator over  $ {\mathcal R}(\T,\widehat{\T}):=\{K\in \cT| \, \exists T\in \mathcal{T} \setminus \widehat{\T} \text{ with } \text{dist}(K,T)=~0\}$:
\begin{align}
\dd^2(\T,\widehat{\T}) \leq \Lambda_3 \eta^2(\mathcal{R}(\T,\widehat{\T})).\tag{\bf A3}
\end{align}
 The set ${\mathcal R}(\T,\widehat{\T})$ consists of triangles in $\mathcal{T} \setminus \widehat{\T}$ plus one layer of triangles around it satisfying $\big|\mathcal R(\T,\widehat{\T})\big| \leq \Lambda \big|\T\setminus\widehat{\T}\big|$. More recently, \cite[Theorem 4.6]{CS24} introduces a weak form of {\bf (A3)} by adding the complete estimator with the coefficient $\varepsilon$, over the whole domain. {The fourth axiom {\bf (A4)} is {\it quasi-orthogonality} that shows that the distance function converges to 0.} Along with {\it stability} and {\it reduction}, it also establishes the R-linear convergence for the error estimator \cite[Proposition 4.10]{CMMD2014}. 
 A weak form of {\bf (A4)} denoted as $({\rm \bf A4}_{\varepsilon})$ is introduced by adding $\varepsilon$ times the series of complete estimators, and it is shown that {\bf (A1)} and {\bf (A2)}, together with $({\rm \bf A4}_{\varepsilon})$, imply {\bf (A4)} for sufficiently small $\varepsilon > 0$ (ref. \cite[Theorem 3.1]{CCRH17}). Further, \cite[Theorem A.1]{CS24} proves that the axioms {\bf (A1)}, {\bf (A2)}, $({\rm \bf A3}_{\varepsilon})$, $({\rm \bf A4}_{\varepsilon})$ establish quasi-optimality of AFEM. The axioms are as follows.
For $0<\varepsilon\leq \varepsilon_0$, there exists $\delta_0>0$ such that for all $\T\in {\mathbb T}(\delta_0)$ satisfy
\[
\begin{array}{ll}
 ({\rm \bf A1}) & \text{\bf Stability. } \forall \widehat{\T} \in \mathbb{T}(\T), \big|\widehat{\eta}(\T\cap \widehat{\T})-\eta(\T\cap\widehat{\T})\big| \leq \Lambda_1 \dd(\T,\widehat{\T}) \\
 
({\rm \bf A2}) & \text{\bf Reduction. } \forall \widehat{\T} \in \mathbb{T}(\T),~ 0<\rho_2<1, \widehat{\eta}(\widehat{\T} \setminus \T) \leq \rho_2 \eta(\T\setminus \widehat{\T}) + \Lambda_2 \dd(\T,\widehat{\T}) \\

({\rm \bf A3}_{\varepsilon}) & \text{\bf Discrete Reliability. } \forall \widehat{\T} \in \mathbb{T}(\T),~ \dd^2(\T,\widehat{\T}) \leq \Lambda_3 \eta^2(\mathcal{R}(\T,\widehat{\T})) + \varepsilon \eta^2(\T) \\

({\rm \bf A4}_{\varepsilon}) & \text{\bf Quasi-orthogonality. } \forall \ell\in \mathbb{N}_0, \sum_{k=\ell}^\infty \dd^2(\T_k, \T_{k+1}) \leq \Lambda_{4(\varepsilon)} \eta^2_\ell+\varepsilon \sum_{k=\ell}^\infty \eta_k^2.
\end{array}
\]

 \noindent  {Here $\Lambda_i,i=1,2,3$ are universal constants, $\Lambda_{4(\varepsilon)}$ depends on $\varepsilon$, and $\rho_2$ is independent of mesh size.} In this section, the axioms, {\bf (A1)}, {\bf (A2)}, $({\rm \bf A3}_{\varepsilon})$, $({\rm \bf A4}_{\varepsilon})$ are verified for the OCP considered in \eqref{opt1}. The choices of $\varepsilon_0$ and $\delta_0$ are discussed in the respective proofs.  
\begin{thm}[optimal convergence rates in adaptive FEM]\label{thm:2.2}
Assume that the nodes in the initial triangulation include the source points, ${\rm z}\in \Z$ and 
      there exist  $\delta_0,\theta >0$ such that $0< \theta  < \theta_0:= 1/(1 + \Lambda_1^2\Lambda_3)$. Then for all $s>0$, the AFEM  with outputs  {$(\T_\ell )_{\ell \in {\mathbb N}_0}\in \mathbb{T}(\delta_0)$} and $ (\eta_\ell )_{\ell \in {\mathbb N}_0}$ satisfies 
 	    \begin{align*}
 	        \sup_{\ell\in {\mathbb N}_0}(1+|\T_\ell|-|\T_0|)^{s}\eta_\ell\approx  
          \sup_{N\in {\mathbb N}_0}  \min_{\T\in {\mathbb T}(\T_0,N)} (1+N)^{s}\eta(\T).
 	    \end{align*}
      Here for $N\in \mathbb{N}$, define ${\mathbb T}(\T_0,N):=\{\T\in \bT(\T_0):|\T|\leq N+|\T_0|\}.$
    	\end{thm}
  \noindent The proof of Theorem~\ref{thm:2.2} follows from the four axioms {\it stability} {\bf (A1)}, {\it reduction} {\bf (A2)}, {\it discrete reliability} $({\rm \bf A3}_{\varepsilon})$, and {\it quasi-orthogonality} $({\rm \bf A4}_{\varepsilon})$\cite{CMMD2014,CS24}. The axioms of adaptivity {\bf (A1)}-{\bf (A2)} are verified in Section \ref{A1}, $({\rm \bf A3}_{\varepsilon})$ in Section \ref{disrel}, and $({\rm \bf A4}_{\varepsilon})$ in Section \ref{axioma4}. The axiom $({\rm \bf A4}_{\varepsilon})$ is established under the assumption that the point sources are included in the nodes of the initial triangulation, whereas {\bf (A1)}, {\bf (A2)}, $({\rm \bf A3}_{\varepsilon})$ are established without this assumption. The choice of $\delta_0$ is also discussed in Remark \ref{delta}.\qed
 \subsection{Proofs of (A1) and (A2)}\label{A1}
		The proofs of {\bf (A1)} and {\bf (A2)} utilize triangle and Cauchy inequalities and the discrete jump control estimates.	
\begin{lem}(discrete jump control)\label{Ldjc}\cite[Lemma 5.2]{CCRH17}
For $k\in\mathbb{N}$, there exists $0<C_{\rm djc}<\infty$ depending on the shape regularity of $\T$ such that, for all $g\in\mathcal{P}_k(\T)$ and $\T\in \mathbb{T}$, we have \begin{equation}
    \sum_{T\in\T}h_T\sum_{E\in\E(T)}\norm{[g]_E}_{L^2(E)}^2\leq C_{\rm djc}^2\norm{g}_{L^2(\Omega)}^2.\tag*{\qed} 
\end{equation}
\end{lem}

\begin{figure}[htbp!]
\centering
\begin{tikzpicture}[scale=3]

\begin{scope}
  \coordinate (A) at (0,0);
  \coordinate (B) at (1,0);
  \coordinate (C) at (1,1);
  \coordinate (D) at (0,1);
  \coordinate (M) at (0.5,0.5);

  \fill[pattern=north east lines, pattern color=black!60] (A) -- (M) -- (B) -- cycle;
  \fill[pattern=north east lines, pattern color=black!60] (B) -- (M) -- (C) -- cycle;
  \fill[pattern=north east lines, pattern color=black!60] (C) -- (M) -- (D) -- cycle;

  \draw[thick] (A) -- (M);
  \draw[thick] (B) -- (M);
  \draw[thick] (C) -- (M);
  \draw[thick] (D) -- (M);

  \draw[thick] (0,0) rectangle (1,1);

  \node at (0.5,-0.15) {\small Triangulation $(\T)$ };
\end{scope}

\begin{scope}[xshift=1.6cm]

  \coordinate (A) at (0,0);
  \coordinate (B) at (1,0);
  \coordinate (C) at (1,1);
  \coordinate (D) at (0,1);
  \coordinate (M) at (0.5,0.5);
  \coordinate (P) at ($(D)!0.5!(A)$);  

  \fill[pattern=north east lines, pattern color=black!60] (A) -- (M) -- (B) -- cycle;
  \fill[pattern=north east lines, pattern color=black!60] (B) -- (M) -- (C) -- cycle;
  \fill[pattern=north east lines, pattern color=black!60] (C) -- (M) -- (D) -- cycle;

  \fill[pattern=crosshatch, pattern color=black]
    (D) -- (M) -- (P) -- cycle;
  \fill[pattern=crosshatch, pattern color=black]
    (P) -- (M) -- (A) -- cycle;

  \draw[thick] (A) -- (M);
  \draw[thick] (B) -- (M);
  \draw[thick] (C) -- (M);
  \draw[thick] (D) -- (M);

  \draw[thick] (D) -- (P);
  \draw[thick] (A) -- (P);
  \draw[thick] (P) -- (M);

  \draw[thick] (0,0) rectangle (1,1);

  \node at (0.5,-0.15) {\small Refined triangulation $(\widehat{\T})$};
\end{scope}

\end{tikzpicture}
\caption{ Common triangles $(\T\cap \widehat{\T})$ are shown with north-east hatching; the refined triangle $(\T\setminus\widehat{\T})$ are unmarked; new triangles $( \widehat{\T}\setminus \T)$ are shown with crosshatch pattern}
\end{figure}

\begin{proof}[{\bf Proof of (A1)}]
A reverse triangle inequality shows
\begin{align}
|\heta(\T\cap\widehat{\T})-\eta(\T\cap\widehat{\T})|^2
\leq  \sum_{T\in \T\cap\widehat{\T}}\big(\heta(T)-\eta(T)\big)^2 .\label{NNcal07}
\end{align}
For $T\in \T\cap\widehat{\T}$, the definition of $\eta(T)$ in \eqref{elemest}, and a reverse triangle inequality in  $\mathbb{R}^4 $ lead to 
\begin{align}
\big(\heta(T)-\eta(T)\big)^2
\leq & \big(\heta_{S,T}-\eta_{S,T}\big)^2+\big(\heta_{S,\E(T)}-\eta_{S,\E(T)}\big)^2+\big(\heta_{A,T}-\eta_{A,T}\big)^2+\big(\heta_{A,\E(T)}-\eta_{A,\E(T)}\big)^2.\label{NNcal06}
\end{align}
The reverse triangle inequality once more and the definitions of estimators \eqref{estmSA} yield 
\begin{align}\label{NNcal08}
\begin{aligned}
\big(\heta_{S,T}-\eta_{S,T}\big)^2 \leq & {\sum_{\z\in \Z}h_T^2 |{\widehat{\ub}_{\rm z,h}-\ub_{\rm z,h}}|^2}, ~\big(\heta_{A,T}-\eta_{A,T}\big)^2\leq h_T^{4} \norm{\widehat{\yb}_\rM-\yb_\rM}_{L^2(T)}^2,\\
\big(\heta_{S,\E(T)}-\eta_{S,\E(T)}\big)^2 \leq & h_T\sum_{E\in \E(T)} \norm{[D_{\rm pw}^2(\widehat{\yb}_\rM-\yb_\rM)]_E \tau_E}_{L^2(E)}^2,   \\
\text{ and } 
\big(\heta_{A,\E(T)}-\eta_{A,\E(T)}\big)^2
\leq &    h_T\sum_{E\in \E(T)} \norm{[D_{\rm pw}^2(\widehat{\pb}_\rM-\pb_\rM)]_E  \tau_E}_{L^2(E)}^2.
\end{aligned}
\hspace{1em} \left.\rule{0pt}{4em}\right\} 
\end{align}
A summation over $T\in \T\cap \widehat{\T}$, and utilizing Lemma \ref{Poincare}$(b)$ (for the first inequality below) and Lemma \ref{Ldjc} (for the second inequality below) show  
\begin{align}
 &\sum_{T\in \T\cap\widehat{\T}}\big(\heta_{S,T}-\eta_{S,T}\big)^2 + \big(\heta_{A,T}-\eta_{A,T}\big)^2 
 \leq (1+|\Omega|C_{\rm PI}^2) h^{2} ( \norm{\widehat{\bar{\bu}}_\h-\bar{\bu}_\h}_{\mathbb{R}^n}^2+ \enorm{\widehat{\yb}_\rM-\yb_\rM}_\pw^2),\no\\
&\sum_{T\in \T\cap\widehat{\T}}\big(\heta_{S,\E(T)}-\eta_{S,\E(T)}\big)^2+\big(\heta_{A,\E(T)}-\eta_{A,\E(T)}\big)^2\leq C_{\rm djc}^2 \big (\enorm{\widehat{\yb}_\rM-\yb_\rM}_\pw^2+\enorm{\widehat{\pb}_\rM-\pb_\rM}_\pw^2\big). \no
\end{align}
A substitution of the above estimates in \eqref{NNcal06} and \eqref{NNcal07} result in 
$|\heta(\T\cap\widehat{\T})-\eta(\T\cap\widehat{\T})|
		\leq  \Lambda_1 \dd(\T,\widehat{\T}).
		$
  This concludes the proof of {\bf (A1)} with $\Lambda_1^2:= (1+|\Omega|C_{\rm PI}^2)|\Omega|+C_{\rm djc}^2$.
		\end{proof}
	\begin{proof}[{\bf Proof of (A2)}]
A reverse triangle inequality yields 
		\begin{align}\label{a2pf}
		&\widehat{\eta}(\widehat{\T}\setminus \T)=\bigg( \sum_{K\in \T\setminus \widehat{\T}} \sum_{T\in \widehat{\T}(K)} \big(\heta_{S,T}^2 +\heta_{S,\E(T)}^2 +\heta_{A,T}^2+\heta_{A,\E(T)}^2 \big)\bigg )^{1/2} \nonumber \\
		 &\leq \bigg( \sum_{K\in \T\setminus\widehat{\T}}\sum_{T\in\widehat{\T}(K)}(\heta_{S,T}-\eta_{S,T})^2+(\heta_{S,\E(T)}-\eta_{S,\E(T)})^2+(\heta_{A,T}-\eta_{A,T})^2+(\heta_{A,\E(T)}-\eta_{A,\E(T)})^2 \bigg)^{1/2}	\nonumber \\	
&\quad +\bigg( \sum_{K\in \T\setminus\widehat{\T}}\sum_{T\in\widehat{\T}(K)}\eta_{S,T}^2+\eta_{S,\E(T)}^2+\eta_{A,T}^2+\eta_{A,\E(T)}^2 \bigg )^{1/2}=:S_1+S_2 ~~{\text{say.}}
		\end{align}
  	To estimate  $S_1$, proceed as in the proof of {\bf (A1)} by employing \eqref{NNcal08}, Lemma \ref{Poincare}$(b)$, and Lemma \ref{Ldjc} to establish
		\begin{align}		
  S_1^2\leq &~ (1+|\Omega|C_{\rm PI}^2) h^{2}\big( \norm{\widehat{\bar{\bu}}_\h-\bar{\bu}_\h}_{\mathbb{R}^n}^2
  +\enorm{\widehat{\yb}_\rM-\yb_\rM}_{\pw}^2\big)+ C_{\rm djc}^2\dd^2(\T,\widehat{\T}).	\label{NNcal2}
  \end{align}
 Note that for any $\phi_\rM\in V_\rM$, $[D_{\rm pw}^2\phi_\rM]_E=0$ for the edges $E\in \E(T)$ which are in the interior of the triangle $K$. Utilize  $h_T\leq 2^{-1/2}h_K$ and \eqref{estmSA} to obtain
	\begin{align}
  S_2^2&\leq \sum_{K\in \T\setminus\widehat{\T}} \left ( \frac{\eta_{S,K}^2+\eta_{A,K}^2}{2}+\frac{\eta_{S,\E(K)}^2+\eta_{A,\E(K)}^2}{\sqrt{2}}  \right ) \leq \frac{1}{\sqrt{2} } \sum_{K\in \T\setminus\widehat{\T}} \big (\eta_{S,K}^2+\eta_{A,K}^2+\eta_{S,\E(K)}^2+\eta_{A,\E(K)}^2\big ).\label{NNcal1}
	\end{align}
  A combination of \eqref{a2pf}-\eqref{NNcal1} leads to {\bf (A2)} with $\rho_2=2^{-1/4}$ and $\Lambda_2^2:= (1+|\Omega|C_{\rm PI}^2)|\Omega|+C_{\rm djc}^2.$ 
   \end{proof}
  \begin{remark}[{\bf (A1)} and {\bf (A2)} for volume estimators]\label{a12_rem}
     With the assumption, ${\rm z}\in \mathcal{N}^i(\T_0)\fl {\rm z}\in\Z$, the estimator $\eta_{S,T}$ vanishes and the volume estimator will have only contribution from the adjoint solution, $\mu^2=\sum_{T\in \T}\eta_{A,T}^2$. Consider the estimator, $\mu$ (resp. $\widehat{\mu}$) instead of the complete estimator, $\eta$ (resp. $\widehat{\eta}$) and revisit the arguments in {\bf (A1)} and {\bf (A2)} to arrive at  
      \begin{align}
 | \widehat{\mu}(\T\cap\widehat{\T})-\mu(\T\cap \widehat{\T})|\leq  \Lambda_{0}h^2\dd(\T,\widehat{\T}) \text{ and } 
\widehat{\mu}(\widehat{\T}\setminus\T)\leq 2^{-1/2}\mu(\T\setminus\widehat{\T})+\Lambda_{0}h^2\dd(\T,\widehat{\T})\no
\end{align}
with $\Lambda_{0}:=C_{\rm PI}$. Along with these estimates following the same arguments as in \cite[Step 3, Proof of Theorem 5.3]{AKDNNSN2024} yield
\begin{align}
    \mu^2(\T\setminus \widehat{\T}) \leq 4(\mu^2-\widehat{\mu}^2)+4\Lambda_0 h^2\dd(\T,\widehat{\T}) \big ( \widehat{\mu}+\mu \big )+3\times 2^2  \Lambda_0^2 h^4 \dd^2(\T,\widehat{\T}).\label{cora12}
\end{align}
These estimates will be used in verifying the axiom $({\rm \bf A4}_{\varepsilon})$.
  \end{remark}

	\subsection{Proof of $({\rm \bf A3}_{\varepsilon})$}\label{disrel}
   In this section, a weak form of {\it discrete reliability} is established \cite[Theorem 4.6]{CS24}. 
For a given discrete control and state $\bar{\bu}_\h$ and 
$\yb_\rM$ over $\mathcal T$, define the auxiliary problems that seek $(\widehat{\ty}_\rM,\widehat{\tp}_\rM)\in \widehat{V}_\rM\times \widehat{V}_\rM$ such that
\begin{align}
a_\pw(\widehat{\ty}_\rM,\widehat{\phi}_\rM)&=\sum_{{\rm z}\in \Z}\ub_{\rm z,h}\langle\delta_{\rm z},\widehat{J}\widehat{\phi}_\rM\rangle~\text{and}~a_\pw(\widehat{\phi}_\rM,\widehat{\tp}_\rM)=(\yb_\rM-\yd,\widehat{\phi}_\rM) \text{ for all }\widehat{\phi}_\rM\in \widehat{V}_\rM,\label{daux1}
\end{align}
where the companion operator $\widehat{J}:\widehat{V}_\rM\rightarrow V$. The well-posedness of the solutions of \eqref{daux1} follows from the Lax-Milgram lemma \cite[Section 6.2.1]{Evans}.
 The following lemma will be used in the proof of the $({\rm \bf A3}_{\varepsilon})$. {Recall that $\bT(\delta)$ denotes the set of all shape regular triangulations $\T$ with $\max_{T\in \T}h_T\leq \delta$.}
 \begin{lem}[discrete error in control]\label{cont_a3}
     Let $(\yb_\rM,\pb_\rM,\bar{\bu}_\h)\in V_\rM \times V_\rM \times U_{\rm ad}$ (resp. $(\widehat{\yb}_\rM,\widehat{\pb}_\rM,\widehat{\bar{\bu}}_\h)\in \widehat{V}_\rM \times \widehat{V}_\rM \times U_{\rm ad}$) solve \eqref{discrete1} with respect to the triangulation $\T\in \bT(\delta)$ (resp. $\widehat{\T}$) and $(\widehat{\ty}_\rM,\widehat{\tp}_\rM)$ solve \eqref{daux1}. Then, the control error satisfies the following bound:
     \begin{align*}
    C_{10}^{-1} \norm{\widehat{\bar{\bu}}_\h-\bar{\bu}_\h}_{\mathbb{R}^n}\leq \enorm{\widehat{\ty}_\rM-\yb_\rM}_\pw+ \enorm{\widehat{\tp}_\rM-\pb_\rM}_\pw + \delta\enorm{\tp-\pb_\rM}_\pw.
\end{align*}
 \end{lem}
\begin{proof}
The proof is presented in two steps. In the first step, key intermediate estimates are derived, followed by the discrete error estimates in the control variable in the second step.

\medskip
\noindent {\it Step 1.}~(intermediate estimates). For $\widehat{\phi}_\rM \in \widehat{V}_\rM$, \eqref{state_eq1} and \eqref{dadj} over $\widehat{\T}$, and \eqref{daux1} imply
\begin{align}
    a_\pw(\widehat{\yb}_\rM-\widehat{\ty}_\rM,\widehat{\phi}_\rM)= \sum_{{\rm z}\in \Z}(\widehat{\ub}_{\rm z,h}-\ub_{\rm z,h})\langle\delta_{\rm z},\widehat{J}\widehat{\phi}_\rM\rangle \text{ and } a_\pw(\widehat{\phi}_\rM,\widehat{\pb}_\rM-\widehat{\tp}_\rM)=(\widehat{\yb}_\rM-\yb_\rM,\widehat{\phi}_\rM).\label{dNNcal14}
\end{align}
A choice of the test function $\widehat{\phi}_\rM=\widehat{\yb}_\rM-\widehat{\ty}_\rM$ in the first equation of \eqref{dNNcal14} 
shows
\begin{align}
\enorm{\widehat{\yb}_\rM-\widehat{\ty}_\rM}_\pw^2&\leq   \sum_{{\rm z}\in \Z}|\widehat{\ub}_{\rm z,h}-\ub_{\rm z,h}||\widehat{J}\widehat{\yb}_\rM({\rm z})-\widehat{J}\widehat{\ty}_\rM({\rm z})|\leq  \sqrt{n}\norm{\widehat{\bar{\bu}}_\h-\bar{\bu}_\h}_{\mathbb{R}^n}\norm{\widehat{J}\widehat{\yb}_\rM-\widehat{J}\widehat{\ty}_\rM}_{L^{\infty}(\Omega)}.\nonumber
\end{align}
Utilizing Lemma \ref{Poincare}$(a)$ and Lemma \ref{enrithm}$(a)$ in the above estimate reveals
\begin{align}
 \enorm{\widehat{\yb}_\rM-\widehat{\ty}_\rM}_\pw& \leq C_6   \norm{\widehat{\bar{\bu}}_\h-\bar{\bu}_\h}_{\mathbb{R}^n} \label{dNNcal15}
\end{align}
with $C_6:=\sqrt{n}C_{\rm PJ}C_{\rm J} $. The choice of the test function $\widehat{\phi}_\rM=\widehat{\pb}_\rM-\widehat{\tp}_\rM\in \widehat{V}_\rM$ in the second equation of \eqref{dNNcal14} and Lemma \ref{Poincare}$(a)$-(b) show
 \begin{align} 
 \enorm{\widehat{\pb}_\rM-\widehat{\tp}_\rM}_\pw&\leq C_{\rm PJ} \norm{\widehat{\yb}_\rM-\yb_\rM}\leq C_{\rm PJ} C_{\rm PI}\enorm{\widehat{\yb}_\rM-\yb_\rM}_\pw. \label{dNNcal16}
\end{align} 

\noindent {\it Step 2.}~(discrete error estimate in control).~The choice $w_{\rm z}=\ub_{\rm z,h}$ (resp. $w_{\rm z}=\widehat{\ub}_{\rm z,h}$) in \eqref{optimality} with respect to the triangulation $\widehat{\T}$ (resp. $\T$) yields
\begin{align*}
 \sum_{{\rm z}\in \Z}(\alpha \widehat{\ub}_{\rm z,h}+\widehat{J}\widehat{\pb}_\rM({\rm z})) (\ub_{\rm z,h}-\widehat{\ub}_{\rm z,h})\geq 0\quad ({\rm resp. }\quad
 \sum_{{\rm z}\in \Z}(\alpha \ub_{\rm z,h}+J\pb_\rM({\rm z}))(\widehat{\ub}_{\rm z,h}-\ub_{\rm z,h})\geq 0).
\end{align*}
Add the above displayed inequalities and introduce the auxiliary solution $\widehat{\tp}_\rM$ to obtain
\begin{align}\label{dNth1}
\alpha  \sum_{{\rm z}\in \Z}(\widehat{\ub}_{\rm z,h}-\ub_{\rm z,h})^2
&\leq  \sum_{{\rm z}\in \Z} (J\pb_\rM({\rm z})-\widehat{J}\widehat{\pb}_\rM({\rm z}))(\widehat{\ub}_{\rm z,h}-\ub_{\rm z,h})\no\\
&=  \sum_{{\rm z}\in \Z}(J\pb_\rM({\rm z})-\widehat{J}\widehat{\tp}_\rM({\rm z}))(\widehat{\ub}_{\rm z,h}-\ub_{\rm z,h})+ \sum_{{\rm z}\in \Z}(\widehat{J}\widehat{\tp}_\rM({\rm z})-\widehat{J}\widehat{\pb}_\rM({\rm z}))(\widehat{\ub}_{\rm z,h}-\ub_{\rm z,h}). 
\end{align}
 The triangle inequality, \eqref{sup2pw}(used twice), and Lemma \ref{Poincare}$(b)$ show for $\z\in\Z\cap T$,
    \begin{align}
       |{J\pb_\rM(\z)-\widehat{J}\widehat{\tp}_\rM(\z)}|
       &\leq  \norm{J\pb_\rM-\pb_\rM}_{L^{\infty}({T})}+\norm{\pb_\rM-\widehat{\tp}_\rM}_{L^{\infty}({T})} +\norm{\widehat{\tp}_\rM-\widehat{J}\widehat{\tp}_\rM}_{L^{\infty}({T})} \label{sn11}\\
        &\leq  C_{\rm J3}  h_{T}\enorm{\tp-\pb_\rM}_\pw+C_{\rm PI}\enorm{\pb_\rM-\widehat{\tp}_\rM}_\pw + C_{\rm J3} h_{T}\enorm{\widehat{\tp}_\rM-\tp}_\pw\no\\
        &\leq  2C_{\rm J3}  h_{T}\enorm{\tp-\pb_\rM}_\pw+C_{\rm PI}\enorm{\pb_\rM-\widehat{\tp}_\rM}_\pw + C_{\rm J3} h_{T}\enorm{\widehat{\tp}_\rM-\pb_\rM}_\pw.\no
    \end{align}
    A summation over ${{\rm z}\in \Z}$ yields $$C_8^{-1}\sum_{{\rm z}\in \Z} |{J\pb_\rM(\z)-\widehat{J}\widehat{\tp}_\rM(\z)}| \leq \delta\enorm{\tp-\pb_\rM}_\pw+\enorm{\pb_\rM-\widehat{\tp}_\rM}_\pw$$
    with $C_8:=n\max\{2C_{\rm J3},C_{\rm PI}+|\Omega|^{1/2}C_{\rm J3}\}$ and $h_T\leq\delta$. Utilizing the Young's inequality and the above estimate in the first term of the right-hand side of \eqref{dNth1} leads to 
\begin{align}
     \sum_{{\rm z}\in \Z} (J\pb_\rM({\rm z})-\widehat{J}\widehat{\tp}_\rM({\rm z}))(\widehat{\ub}_{\rm z,h}-\ub_{\rm z,h})\leq  C_9^2\big(\delta^2 \enorm{\tp -\widehat{\tp}_\rM}_\pw^2+ \enorm{\pb_\rM -\widehat{\tp}_\rM}_\pw^2 \big)+ 2^{-2}\alpha \norm{\widehat{\bar{\bu}}_\h-\bar{\bu}_\h}_{\mathbb{R}^n}^2\no
\end{align} 
with $C_9^2:=2\alpha^{-1}C_8^2$. The substitutions $\widehat{\phi}_\rM=\widehat{\tp}_\rM-\widehat{\pb}_\rM\in \widehat{V}_\rM$ and $\widehat{\phi}_\rM=\widehat{\yb}_\rM-\widehat{\ty}_\rM\in \widehat{V}_\rM$ in the first and second inequalities of $\eqref{dNNcal14}$, respectively, and elementary algebra lead to 
   \begin{align*}
      &  \sum_{{\rm z}\in \Z} (\widehat{J}\widehat{\tp}_\rM({\rm z})-\widehat{J}\widehat{\pb}_\rM({\rm z}))(\widehat{\ub}_{\rm z,h}-\ub_{\rm z,h})= a_\pw(\widehat{\yb}_\rM-\widehat{\ty}_\rM,\widehat{\tp}_\rM-\widehat{\pb}_\rM)= (\yb_\rM-\widehat{\yb}_\rM,\widehat{\yb}_\rM-\widehat{\ty}_\rM)\no\\
       &=-\norm{\yb_\rM-\widehat{\yb}_\rM}^2+\norm{\yb_\rM-\widehat{\ty}_\rM}^2+(\widehat{\ty}_\rM-\widehat{\yb}_\rM,\yb_\rM-\widehat{\ty}_\rM)
\leq\norm{\yb_\rM-\widehat{\ty}_\rM}^2+(\widehat{\ty}_\rM-\widehat{\yb}_\rM,\yb_\rM-\widehat{\ty}_\rM)
   \end{align*} 
   by omitting the first term on the right side in the last inequality.
   An application of the Cauchy Schwarz inequality, Young's inequality, Lemma \ref{Poincare}$(b)$ (twice), and \eqref{dNNcal15} shows
   \begin{align*}
   (\widehat{\ty}_\rM-\widehat{\yb}_\rM,\yb_\rM-\widehat{\ty}_\rM)&\leq  2^{-2}{\alpha}{C_6^{-2}C_{\rm PI}^{-2}}\norm{\widehat{\ty}_\rM-\widehat{\yb}_\rM}^2+ {\alpha}^{-1}{C_6^2C_{\rm PI}^2} \norm{\yb_\rM-\widehat{\ty}_\rM}^2\no\\
   &\leq 2^{-2}{\alpha} \norm{\widehat{\bar{\bu}}_\h-\bar{\bu}_\h}_{\mathbb{R}^n}^2+{\alpha}^{-1}{C_6^2}C_{\rm PI}^4 \enorm{\yb_\rM-\widehat{\ty}_\rM}_\pw^2.
   \end{align*}
A combination of last three displayed estimates in \eqref{dNth1} with Lemma \ref{Poincare}$(b) $ establishes
   \begin{align}
    C_{10}^{-1} \norm{\widehat{\bar{\bu}}_\h-\bar{\bu}_\h}_{\mathbb{R}^n}\leq\enorm{\widehat{\ty}_\rM-\yb_\rM}_\pw+ \enorm{\widehat{\tp}_\rM-\pb_\rM}_\pw +\delta\enorm{\tp-\pb_\rM}_\pw\no
\end{align}
with $C_{10}^2(\alpha^{-1}):=2\alpha^{-1}\max\{C_9^2,(1+\alpha^{-1}C_6^2C_{\rm PI}^2)C_{\rm PI}^2\}$.
\end{proof}
\begin{remark}
    {The approach of the proof in Lemma \ref{cont_a3} is similar to Step 3 in the proof of Theorem \ref{equivalence}; however, Lemma \ref{cont_a3} requires handling companion operators at two different triangulations, which introduces additional difficulties. This also results in an extra term in the estimate, necessitating the use of a weaker form of discrete reliability $({\rm \bf A3}_{\varepsilon})$.}
\end{remark}

\begin{lem}[discrete function bound]\label{disequiv}
Let $(\yb_\rM,\pb_\rM,\bar{\bu}_\h)\in V_\rM \times V_\rM \times U_{\rm ad}$ (resp. $(\widehat{\yb}_\rM,\widehat{\pb}_\rM,\widehat{\bar{\bu}}_\h)\in \widehat{V}_\rM \times \widehat{V}_\rM \times U_{\rm ad}$) solve \eqref{discrete1} with respect to  $\T\in \bT(\delta)$ (resp. $\widehat{\T}$) and $(\widehat{\ty}_\rM,\widehat{\tp}_\rM)$ solve \eqref{daux1}. Then the following discrete error bounds hold.
   \begin{align*}
       (a) &~\norm{\widehat{\bar{\bu}}_\h-\bar{\bu}_\h}_{\mathbb{R}^n}+\enorm{\widehat{\yb}_\rM-\yb_\rM}_\pw+\enorm{\widehat{\pb}_\rM-\pb_\rM}_\pw\leq C_{\rm d1} \big(\enorm{\widehat{\ty}_\rM-\yb_\rM}_\pw+\enorm{\widehat{\tp}_\rM-\pb_\rM}_\pw + \delta\enorm{\tp-\pb_\rM}_\pw\big)\\
       (b) & ~\enorm{\widehat{\ty}_\rM-\yb_\rM}_\pw+\enorm{\widehat{\tp}_\rM-\pb_\rM}_\pw
   \leq C_{\rm d2} \big( \norm{\widehat{\bar{\bu}}_\h-\bar{\bu}_\h}_{\mathbb{R}^n}+\enorm{\widehat{\yb}_\rM-\yb_\rM}_\pw+\enorm{\widehat{\pb}_\rM-\pb_\rM}_\pw\big).
   \end{align*}
   \end{lem}
\begin{proof}[Proof of $(a)$]
\noindent  A triangle inequality for $\enorm{\widehat{\pb}_\rM-\pb_\rM}_\pw$ first with \eqref{dNNcal16} followed by a triangle inequality  for $\enorm{\widehat{\yb}_\rM-\yb_\rM}_\pw$ with \eqref{dNNcal15} show
\begin{align}
    \enorm{\widehat{\yb}_\rM-\yb_\rM}_\pw+\enorm{\widehat{\pb}_\rM-\pb_\rM}_\pw
    &\leq (1+C_{\rm PJ}C_{\rm PI})\enorm{\widehat{\yb}_\rM-\yb_\rM}_\pw+\enorm{\widehat{\tp}_\rM-\pb_\rM}_\pw\no\\
   & \leq C_7\big( \norm{\widehat{\bar{\bu}}_\h-\bar{\bu}_\h}_{\mathbb{R}^n}+\enorm{\widehat{\ty}_\rM-\yb_\rM}_\pw+\enorm{\widehat{\tp}_\rM-\pb_\rM}_\pw\big)\no
\end{align}
 with $C_7:=(1+C_{\rm PJ}C_{\rm PI})(1+C_6)+1.$
 For $C_{\rm d1}(\alpha^{-1}):= (1+C_7) C_{10}$,  Lemma \ref{cont_a3} and the above estimate yield
\begin{align*}
 \norm{\widehat{\bar{\bu}}_\h-\bar{\bu}_\h}_{\mathbb{R}^n}+\enorm{\widehat{\yb}_\rM-\yb_\rM}_\pw+\enorm{\widehat{\pb}_\rM-\pb_\rM}_\pw&\leq C_{\rm d1} \big(\enorm{\widehat{\ty}_\rM-\yb_\rM}_\pw +\enorm{\widehat{\tp}_\rM-\pb_\rM}_\pw+\delta\enorm{\tp-\pb_\rM}_\pw\big).
\end{align*}


\noindent {\it Proof of $(b)$.} The triangle inequality and \eqref{dNNcal15}-\eqref{dNNcal16} conclude the proof with $C_{\rm d2}:=1+C_6(1+C_{\rm PJ}C_{\rm PI})$. 
\end{proof}

\noindent {\bf Proof of $({\rm \bf A3}_{\varepsilon})$.}~~Recall that $\widehat{\ty}_{\rM},\widehat{\tp}_{\rM} \in \widehat{V}_\rM$ solve \eqref{daux1}. Lemma \ref{disequiv}$(a)$ yields 
  \begin{align}
  \norm{\widehat{\bar{\bu}}_\h-\bar{\bu}_\h}_{\mathbb{R}^n}+\enorm{\widehat{\yb}_{\rM}-\yb_\rM}_{\pw} +\enorm{\widehat{\pb}_{\rM}-\pb_\rM}_{\pw}  \leq C_{\rm d1} \big(\enorm{\widehat{\ty}_{\rM}-\yb_\rM}_{\pw}  +\enorm{\widehat{\tp}_{\rM}-\pb_\rM}_{\pw}+\delta\enorm{\tp-\pb_\rM}_\pw\big).\label{equiref3}
  \end{align}
  \noindent {\it Step 1.}~The definition of the piecewise $\pw$-norm shows
        \begin{align}
        \enorm{\widehat{\ty}_{\rM}-\yb_\rM}_{\pw}^2&=a_\pw(\widehat{\ty}_{\rM},\widehat{\ty}_{\rM}-\yb_\rM)-a_\pw(\yb_\rM,\widehat{\ty}_{\rM}-\yb_\rM)\nonumber \\
         & =a_\pw(\widehat{\ty}_{\rM},\widehat{\yb}_{\rM}^*-\yb_\rM)+a_\pw(\widehat{\ty}_{\rM},\widehat{\ty}_{\rM}-\widehat{\yb}_{\rM}^*)-a_\pw(\yb_\rM,I_\rM\widehat{\ty}_{\rM}-\yb_\rM) \no
        \end{align}
        {with $\widehat{\yb}_{\rM}^*:=\widehat{I}_\rM J\yb_\rM\in \widehat{V}_\rM$ and Lemma \ref{Interpolation}$(a)$ employed in the last term. 
Lemma \ref{Interpolation}$(a)$ once again with $\yb_\rM=I_\rM\widehat{\yb}_{\rM}^*$  implies  $a_\pw(\yb_\rM,\widehat{\yb}_{\rM}^*-\yb_\rM)=0$ (see \cite[(C4)]{CCSP20}). Employ this with \eqref{daux1} and \eqref{discrete1} in the above estimate to arrive at } 
		\begin{align}
		 \enorm{\widehat{\ty}_{\rM}-\yb_\rM}_{\pw}^2 =
  a_\pw(\widehat{\ty}_{\rM}-\yb_\rM,\widehat{\yb}_{\rM}^*-\yb_\rM)-\sum_{\rm z\in \Z}{\ub}_{\rm z,h}\langle \delta_z,(\widehat{J}-JI_\rM)(\widehat{\yb}_{\rM}^*-\widehat{\ty}_{\rM})\rangle.\label{Ncal16}
	\end{align}
{For any $\z\in {\mathcal N}^i(\widehat{\T})$, the second expression on the right-hand side of the above estimate vanishes. Otherwise, for $T\in\T_{\rm d}(\widehat{\T})$ and $\widehat{w}_\rM\in \widehat{V}_\rM$,} utilize the triangle inequality, \eqref{sup2pw} (twice), inverse estimate, and Lemma \ref{Interpolation}$(b)$ to obtain 
\begin{align*}
 \norm{(\widehat{J}-JI_\rM)\widehat{w}_\rM}_{L^\infty(T)}
 &\leq \norm{(\widehat{J}-1)\widehat{w}_\rM}_{L^\infty(T)}+\norm{(1-I_\rM)\widehat{w}_\rM}_{L^\infty(T)}+\norm{(I_\rM-JI_\rM)\widehat{w}_\rM}_{L^\infty(T)} \\&\leq C_{S1} h_T|\widehat{w}_\rM|_{H^2(\omega_T)}
 \end{align*}
where $C_{S1}:=2C_{\rm J3}+C_{\rm inv}C_{\rm I}$. Since by definitions of interpolation and companion operator,  $(\widehat{J}\widehat{I}_\rM-JI_\rM)\widehat{w}_\rM|_T=0$ for all $T\in \T\setminus {\mathcal R}(\T,\widehat{\T})$ and $\widehat{J}\widehat{I}_\rM\widehat{w}_\rM=\widehat{J}\widehat{w}_\rM$, we obtain $(\widehat{J}-JI_\rM)\widehat{w}_\rM|_T=0$ for all $T\in \T\setminus {\mathcal R}(\T,\widehat{\T})$. These results with $\widehat{w}_\rM=\widehat{\yb}_{\rM}^*-\widehat{\ty}_{\rM}$ and a Cauchy-Schwarz inequality in \eqref{Ncal16} show
\begin{align} 
\enorm{\widehat{\ty}_{\rM}-\yb_\rM}_{\pw}^2 
  \leq & \enorm{\widehat{\ty}_{\rM}-\yb_\rM}_{\pw}\enorm{\widehat{\yb}_{\rM}^*-\yb_\rM}_{\pw}+C_{\rm S1}\enorm{\widehat{\yb}_{\rM}^*-\widehat{\ty}_{\rM}}_\pw\sum_{T \in {\mathcal R}(\T,\widehat{\T})}\sum_{\rm z\in \Z\cap T}|{\ub}_{\rm z,h}|h_T\nonumber\\
  \leq & \enorm{\widehat{\ty}_{\rM}-\yb_\rM}_{\pw}\enorm{\widehat{\yb}_{\rM}^*-\yb_\rM}_{\pw}+\sqrt{2n}C_{\rm S1}(\enorm{\widehat{\yb}_{\rM}^*-\yb_\rM}_\pw+\enorm{\yb_\rM-\widehat{\ty}_{\rM}}_\pw)\eta_S({\mathcal R}(\T,\widehat{\T})),\nonumber
  \end{align} 
  where the triangle inequality with \eqref{estmSA} is utilized in the last step. Lemma \ref{enrithm}$(c)$ and Young's inequality (used thrice with $\epsilon=2$) in the above displayed estimate reveal
  \begin{align}
  \enorm{\widehat{\ty}_{\rM}-\yb_\rM}_{\pw}^2
  \leq C_{\rm S} \eta_S^2( {\mathcal R}(\T,\widehat{\T})) \label{discreterel1}
 \end{align}
{ with $C_{\rm S}:=n C_{\rm S1}^2(C_{\rm IJ}^2+4)+2(C_{\rm IJ}^2+1)$.}\\
  \noindent {\it Step 2.}~For establishing the adjoint estimates, steps 
 similar to {\it Step 1} lead to 
		\begin{align}
		 \enorm{\widehat{\tp}_{\rM}-\pb_\rM}_{\pw}^2 =  a_\pw(\widehat{\tp}_{\rM}-\pb_\rM, \widehat{\pb}_{\rM}^*-\pb_\rM)-(\yb_\rM-y_{\rm d},(1-I_\rM)(\widehat{\pb}_{\rM}^*-\widehat{\tp}_{\rM})).\no  
	\end{align}
From Lemma \ref{Interpolation}$(c)$, $(1-I_\rM)(\widehat{\pb}_{\rM}^*-\widehat{\tp}_{\rM})=0$ holds for all $T\in \T\cap \widehat{\T}$. This,  a Cauchy-Schwarz inequality, and Lemma \ref{Interpolation}$(b)$ show
		\begin{align}
        \enorm{\widehat{\tp}_{\rM}-\pb_\rM}_{\pw}^2 
		\leq & \enorm{\widehat{\tp}_{\rM}-\pb_\rM}_{\pw}\enorm{\widehat{\pb}_{\rM}^*-\pb_\rM}_{\pw}+C_\rI\sum_{T\in \T\setminus \widehat{\T}}h_T^2\norm{\yb_\rM-y_{\rm d}}_{L^2(T)}\norm{D_\pw^2(\widehat{\pb}_{\rM}^*-\widehat{\tp}_{\rM})}_{L^2(T)}.\nonumber
  \end{align}
 Lemma \ref{enrithm}$(c)$, H\"older's inequality, and \eqref{estmSA} in the last displayed estimate reveal  
  \begin{align}
 C^{-1}_{\rm A1} \enorm{\widehat{\tp}_{\rM}-\pb_\rM}_{\pw}^2		\leq & (\enorm{\widehat{\tp}_{\rM}-\pb_\rM}_{\pw} +\enorm{\widehat{\pb}_{\rM}^*-\widehat{\tp}_{\rM}}_{\pw})\eta_A( {\mathcal R}(\T,\widehat{\T})) \no
 \end{align}
  with  $C_{\rm A1}:=\max\{ C_{\rm IJ},C_\rI\}$. A triangle inequality and Lemma \ref{enrithm}$(c)$ lead to 
  \begin{align*}
   C_{\rm A2}^{-1}\enorm{\widehat{\tp}_{\rM}-\pb_\rM}_{\pw}^2\leq \enorm{\widehat{\tp}_{\rM}-\pb_\rM}_{\pw} \: \eta_A( {\mathcal R}(\T,\widehat{\T}))+ \eta_A ^2( {\mathcal R}(\T,\widehat{\T}))
  \end{align*}
  with $C_{\rm A2}:=C_{\rm A1}\max\{2,  C_{\rm IJ}\}$. A Young's inequality with {$\epsilon=C_{\rm A2}$} shows 
  \begin{align}
      \enorm{\widehat{\tp}_{\rM}-\pb_\rM}_{\pw}^2\leq C_{\rm A}\eta_A^2( {\mathcal R}(\T,\widehat{\T}))\label{sn12}
  \end{align}
  with $C_{\rm A}:=C_{\rm A2}(2+C_{\rm A2})$.
  Substitute \eqref{discreterel1} and \eqref{sn12} in \eqref{equiref3} along with \eqref{newref001} (for $g=\yb_\rM-\yd$)
  to {conclude that for every $\varepsilon>0$ selecting $0<\delta^2\leq \delta_1^2:=3^{-1}C_{\rm rel}^{-2}C_{\rm d1}^{-2}\varepsilon$ leads to }
  $$\dd^2(\T,\widehat{\T}) \leq \Lambda_3 \eta^2(\mathcal{R}(\T,\widehat{\T})) + \varepsilon \eta^2(\T)$$
  with {$\Lambda_3(\alpha^{-1}):=3C_{\rm d1}^2\max\{C_{\rm S},C_{\rm A}\}$.} This concludes the proof of $({\rm \bf A3}_{\varepsilon})$ for all $\T\in {\mathbb T}(\delta_1)$.  \qed
\begin{remark}[$\Z\subseteq {\mathcal N}^i(\T_0)$]\label{directa3}
  In the case that the point sources are included in the nodes of the initial triangulation, the definition of the companion operator show $\langle \delta_{\rm z},J \phi_\rM\rangle=J\phi_\rM({\rm z})= \phi_\rM({\rm z})$ for ${\rm z}\in \Z$ and $\phi_\rM\in V_\rM$ (ref. \cite[Definition 6.5]{CCSP20}). Utilize this in \eqref{sn11} to observe that the first and third terms on the right side of the inequality vanish. With this, re-visit the arguments in {\it Step 2} in Lemma \ref{cont_a3} and substitute in Lemma \ref{apriori}$(a)$ to obtain $$\norm{\widehat{\bar{\bu}}_\h-\bar{\bu}_\h}_{\mathbb{R}^n}+\enorm{\widehat{\yb}_{\rm nc}-\yb_\rM}_{\pw} +\enorm{\widehat{\pb}_{\rm nc}-\pb_\rM}_{\pw}  \leq C_{\rm d1} \big(\enorm{\widehat{\ty}_{\rM}-\yb_\rM}_{\pw}  +\enorm{\widehat{\tp}_{\rM}-\pb_\rM}_{\pw}\big).$$
  Utilizing this instead of \eqref{equiref3} with \eqref{discreterel1} and \eqref{sn12} will prove directly {\it discrete reliability} {\bf (A3)} instead of $({\rm \bf A3}_{\varepsilon})$. Thus, the mesh-size restriction can be avoided by considering $\Z\subseteq {\mathcal N}^i(\T_0)$.
\end{remark}
  \subsection{Proof of $({\rm \bf A4}_{\varepsilon})$} \label{axioma4}
  This subsection verifies the axiom $({\rm \bf A4}_{\varepsilon})$.  
The proof adopts the assumption that {\it the point sources are located at the nodes of the initial triangulation}.
   Note that by the definition of the companion operator, for any $\phi_\rM\in V_\rM$, $\langle \delta_{\rm z},J\phi_\rM\rangle =J\phi_\rM({\rm z})= \phi_\rM({\rm z})\fl {\rm z}\in \Z$ \cite[Definition 6.5]{CCSP20}.  Let $(\yb_\ell,\pb_\ell,{\bf\ub}_\ell)$ be the solution of the optimality system \eqref{discrete1} with respect to the $\ell^{th}$ refinement of $\T_0$ denoted as  $\T_\ell\in\mathbb{T}(\delta)$. For all $\ell\in \mathbb{N}\cup \{0\}$, define $$e_\ell^2:=\norm{\bar{\bu}-\bar{\bu}_\ell}_{{\mathbb R}^n}^2+\enorm{\yb-\yb_\ell}_{\pw}^2+\enorm{\pb-\pb_\ell}_{\pw}^2 , \;\; \osc_\ell:=\osc(\yb_\ell-\yd,\T_\ell), \text{ and } \widehat{e}_{\ell}:= e_\ell+\osc_\ell.$$ 
  \begin{lem}[intermediate results]\label{lem_A4}
      Let $(\yb_\ell,\pb_\ell,{\bf\ub}_\ell)$ be the solution of the optimality system \eqref{discrete1} with respect to $\T_\ell\in\mathbb{T}(\delta)$. Then for $2-\sigma< \gamma \leq 2$, there exist $C_{\rm c}$, $C_{\rm s}$, and $C_{\rm a}$ $>0$ such that 
      \begin{itemize}
      \item[(a)] $ ({\bf\ub}_{k+1}-{\bf\ub}_k,{\bf\ub}_{k+1}-{\bf\ub}) \leq C_{\rm c}   \delta^{2-\gamma}\enorm{\pb_{k+1}-\pb_k}_{\pw}\widehat{e}_{k+1}$,
          \item[(b)] $a_\pw(\yb_{k+1}-\yb_k,\yb_{k+1}-\yb) \leq C_{\rm s}   \delta^{2-\gamma}\enorm{\pb_{k+1}-\pb_k}_{\pw}\widehat{e}_{k+1}$,  
    \item[(c)] $a_\pw(\pb_{k+1}-\pb_k,\pb_{k+1}-\pb)\leq C_{\rm a} \big ( \mu_{k}(\T_k\setminus \T_{k+1})+\delta^{2-\gamma}\enorm{\yb_{k+1}-\yb_k}_{\pw}\big )\widehat{e}_{k+1}$.
      \end{itemize}
  \end{lem}
  \begin{proof}[Proof of (a)]
The Cauchy-Schwarz inequality, {the property of the projection operator $\Pi_{[a_\z,b_\z]}$ from \eqref{FAproj}} (used twice), Lemma \ref{Poincare}$(b)$, \eqref{inf2gamma}
and \eqref{A4apriori} show for $C_{\rm c}(\alpha^{-3}):=n^2\alpha^{-2}C_{\rm PI}C_{\rm emb} C_{\rm AP}(\alpha^{-1})$ that
\begin{align}
({\bf\ub}_{k+1}-{\bf\ub}_k,{\bf\ub}_{k+1}-{\bf\ub})&\leq n^2\alpha^{-2}C_{\rm PI}C_{\rm emb} \enorm{\pb_{k+1}-\pb_{k}}_\pw\norm{\pb_{k+1}-\pb}_{H^{\gamma}(\T)}\no\\
&\leq C_{\rm c}h^{2-\gamma}\enorm{\pb_{k+1}-\pb_{k}}_\pw\widehat{e}_{k+1}.\no
\end{align}
 \noindent {\it Proof of (b).}~The orthogonality condition in Lemma \ref{Interpolation}$(a)$ shows  $$a_\pw(\yb_k, 
    (1-I_k)(\yb_{k+1}-\yb))=0 \text{ and }  a_\pw(\yb_{k+1},(1-I_{k+1})(\yb_{k+1}-\yb))=0. $$ This,  (\ref{state_eq1}),  and an algebraic manipulation of terms yield
    \begin{align}			
    a_\pw&(\yb_{k+1}-\yb_k,\yb_{k+1}-\yb)= a_\pw(\yb_{k+1},I_{k+1}(\yb_{k+1}-\yb))-a_\pw(\yb_k,I_k(\yb_{k+1}-\yb))\nonumber \\
	& =\sum_{{\rm z}\in \Z}({\ub}_{{\rm z},k+1}-{\ub}_{{\rm z},k})\langle\delta_{\rm z},J_{k+1}I_{k+1}(\yb_{k+1}-\yb)\rangle+\sum_{{\rm z}\in \Z}{\ub}_{{\rm z},k}\langle \delta_{\rm z},(J_{k+1}I_{k+1}-J_kI_k)(\yb_{k+1}-\yb)\rangle. \nonumber
\end{align} 
 Since $\z\in \mathcal{N}^i(\T_0)$, 
 \begin{align}
     (J_{k+1}I_{k+1}-J_kI_k)(\yb_{k+1}({\rm z})-\yb({\rm z}))=0,\label{obstacle}
 \end{align}
and this implies 
$ \displaystyle  a_\pw(\yb_{k+1}-\yb_k,\yb_{k+1}-\yb)
    =\sum_{{\rm z}\in \Z}({\ub}_{{\rm z},k+1}-{\ub}_{{\rm z},k})(\yb_{k+1}-\yb)({\rm z}).$
{Recall the definition of $\T_{\rm d}(\T)$ from the proof of Lemma \ref{bih_apost}.} For all ${\rm z}\in \Z$, the definition of control in terms of the projection and \eqref{FAproj} show	
    \begin{align}			
    a_\pw(\yb_{k+1}-\yb_k,\yb_{k+1}-\yb) 
\leq n\alpha^{-1} \norm{\pb_{k+1}-\pb_k}_{L^{\infty}(\T_d(\T_{k+1}))}\norm{\yb_{k+1}-\yb}_{L^{\infty}(\T_d(\T_{k+1}))}.\label{42}
\end{align} 
{The Sobolev embedding from \eqref{inf2gamma}, \eqref{A4apriori}, and \eqref{lipschitz} yield}
    \begin{align}\label{NNcal24}
  { \norm{\yb_{k+1}-\yb}_{L^{\infty}(\T_{\rm d}(\T_{k+1}))}\leq \norm{\yb_{k+1}-\yb}_{L^\infty(\Omega)}\leq  C_{\rm emb}\norm{\yb_{k+1}-\yb}_{H^{\gamma}(\T)}\leq C_{\rm emb} C_{\rm AP}\delta^{2-\gamma} \widehat{e}_{k+1}}.
    \end{align}
 A combination of \eqref{42}-\eqref{NNcal24}, and Lemma \ref{Poincare}$(b)$ in \eqref{42} shows
\begin{align}
	a_\pw(\yb_{k+1}-\yb_k,\yb_{k+1}-\yb) \leq C_{\rm s}   \delta^{2-\gamma}\enorm{\pb_{k+1}-\pb_k}_{\pw}\widehat{e}_{k+1}\label{43}
\end{align}
   { with $C_{\rm s}(\alpha^{-2}):=n\alpha^{-1} C_{\rm emb} C_{\rm PI}C_{\rm AP}(\alpha^{-1})$.}\\
    \noindent {\it Proof of (c).}~The orthogonality condition in Lemma \ref{Interpolation}$(a)$ shows $$a_\pw(\pb_k, 
    (1-I_k)(\pb_{k+1}-\pb))=0 \text{ and }  a_\pw(\pb_{k+1},(1-I_{k+1})(\pb_{k+1}-\pb))=0. $$ The above displayed equations and \eqref{dadj} yield	
    \begin{align}			
    a_\pw(\pb_{k+1}-\pb_k,\pb_{k+1}-\pb)&= a_\pw(\pb_{k+1},I_{k+1}(\pb_{k+1}-\pb))-a_\pw(\pb_k,I_k(\pb_{k+1}-\pb))\nonumber \\
   &=(\yb_{k+1}-\yd,I_{k+1}(\pb_{k+1}-\pb))-(\yb_k-\yd,I_k(\pb_{k+1}-\pb)).\nonumber 
   \end{align}
This and some algebraic manipulation by rearranging the terms result in 
\begin{align}
    a_\pw&(\pb_{k+1}-\pb_k,\pb_{k+1}-\pb) =(\yb_{k+1}-\yb_k,I_{k+1}(\pb_{k+1}-\pb))+(\yb_k-\yd,(I_{k+1}-I_k)(\pb_{k+1}-\pb))\nonumber\\
    & =(\yb_{k+1}-\yb_k,(I_{k+1}-1)(\pb_{k+1}-\pb))+(\yb_{k+1}-\yb_k,\pb_{k+1}-\pb)+(\yb_k-\yd,(I_{k+1}-I_k)(\pb_{k+1}-\pb)). \nonumber
\end{align}
 From Lemma \ref{Interpolation}$(c)$, $(I_k-I_{k+1}) (\pb_{k+1}-\pb)=0$ for all $T\in \T_k\cap \T_{k+1}$. This, Cauchy-Schwarz inequality, Lemma \ref{Interpolation}$(b)$ {(used twice for $(I_k-1)(\pb_{k+1}-\pb)$ and $(1-I_{k+1})(\pb_{k+1}-\pb)$)}, and the stability of the nonconforming interpolation operator in Lemma \ref{Interpolation}$(b)$ show	
    \begin{align}			
    a_\pw(\pb_{k+1}-\pb_k,\pb_{k+1}-\pb) \leq &~ C_\rI \delta^2\norm{\yb_{k+1}-\yb_k}\enorm{\pb_{k+1}-\pb}_\pw+\norm{\yb_{k+1}-\yb_k}\norm{\pb_{k+1}-\pb}\nonumber\\
&+2C_\rI \sum_{T\in\T_k\setminus\T_{k+1}}h_T^2\norm{\yb_k-\yd}_{L^2(T)}\norm{D_\pw^2(\pb_{k+1}-\pb)}_{L^2(T)}.\label{420}
\end{align}
 The a~priori estimate in \eqref{A4apriori} yields
    \begin{align}\label{NNcal240}
   \norm{\pb_{k+1}-\pb}\leq \norm{\pb_{k+1}-\pb}_{H^{\gamma}(\T)}\leq  C_{\rm AP}\delta^{2-\gamma}\widehat{e}_{k+1}.
    \end{align}
    The H{\"o}lder's inequality in the third term of \eqref{420} leads to
    \begin{align}
    \sum_{T\in \T_k\setminus\T_{k+1}}h_T^2\norm{\yb_k-\yd}_{L^2(T)}\norm{D_\pw^2(\pb_{k+1}-\pb)}_{L^2(T)}\leq \mu_k(\T_k\setminus \T_{k+1})\enorm{\pb-\pb_{k+1}}_{\pw}.\label{NNcal230}
    \end{align}
    A combination of \eqref{420}-\eqref{NNcal230} and Lemma \ref{Poincare}$(b)$ in the first and second terms of \eqref{420} shows
\begin{align}
a_\pw(\pb_{k+1}-\pb_k,\pb_{k+1}-\pb) \leq C_{\rm a}  \big ( \delta^{2-\gamma} \enorm{\yb_{k+1}-\yb_k}_{\pw}+\mu_{k}(\T_k\setminus \T_{k+1})\big )\widehat{e}_{k+1}\no
\end{align}
    with {$C_{\rm a}(\alpha^{-1}):=\max\{2C_\rI, C_\rI C_{\rm PI}|\Omega|^{\gamma},C_{\rm PI}C_{\rm AP}(\alpha^{-1})\}$} and $\mu_{k}(\T_k\setminus \T_{k+1})$ denoting the volume estimator with respect to the triangulation $\T_k$ over the set $\T_k \setminus \T_{k+1}$. This completes the proof.
  \end{proof}

\noindent To establish $({\rm \bf A4}_{\varepsilon})$, {for given  $0<\varepsilon \leq \varepsilon_0:=\min \{2^3(1+C^2_{\rm REL}), \max\{2,2^{1/2} (1+C^2_{\rm REL})^{1/2}\}  2^3 \Lambda_6 \}$,} select maximum $\delta_2>0$ such that
     \begin{align}
     \max\{ 2^4\Lambda_6^2\delta_2^{2(2-\gamma)}, 3\times 2^6 \Lambda_6^2 \Lambda_0^2 \delta_2^4 ,{2^{13/3}} \Lambda_6^{4/3} \Lambda_0^{2/3}\delta_2^{4/3}\} \leq \varepsilon \label{cndn2}
     \end{align} 
     with $\Lambda_6^2:=C_{\rm sa}^2(1+ C_{\rm REL}^2)$  where $C_{\rm sa}:=2^{3/2}\max\{C_{\rm c},C_{\rm s},C_{\rm a}\}$.
     
		\begin{thm}[${\rm \bf A4}_{\varepsilon}$]\label{thma4e}
		Let $(\yb_\ell,\pb_\ell,\bar{\bu}_\ell)$ be the solution of the optimality system \eqref{discrete1} with respect to $\T_\ell\in\mathbb{T}(\delta_2)$ and $\dd_{\ell,\ell+1}:=\dd(\T_\ell,\T_{\ell+1})$. {Then for $0<\varepsilon \leq \varepsilon_0$,  there exists $0<\delta\leq \delta_2$ satisfying \eqref{cndn2} and $0<\Lambda_{4(\varepsilon)}<\infty$ such that for all $\ell,m\in \mathbb{N}\cup \{0\}$,  }
		\begin{align}
		\sum_{k=\ell}^{\ell+m}\dd_{k,k+1}^2 \leq \Lambda_{4(\varepsilon)}\eta_\ell^2+\varepsilon\sum_{k=\ell}^{\ell+m}\eta_k^2 .\tag{${\rm \bf A4}_{\varepsilon}$}\label{A4ep}
		\end{align}
		\end{thm}
\begin{proof}
For any $k\in \mathbb{N}\cup \{0\}$, the definitions of $\dd_{k,k+1}$ and $e_k$ plus a re-grouping of terms show 
        \begin{align}
			  \dd_{k,k+1}^2&+e_{k+1}^2-e_k^2 =2({\bf\ub}_k-{\bf\ub}_{k+1},{\bf\ub}-{\bf\ub}_{k+1})+2a_\pw(\yb_{k+1}-\yb_k,\yb_{k+1}-\yb)+2a_\pw(\pb_{k+1}-\pb_k,\pb_{k+1}-\pb). \no
			\end{align}   
The rest of the proof is divided into {\it two} steps for easy readability.

\noindent {\it Step 1.} (Crucial estimates).~ 
A substitution of Lemma \ref{lem_A4} in the above equality reveals
\begin{align}\label{45} 
    \dd_{k,k+1}^2+e_{k+1}^2-e_k^2\leq  C_{\rm sa} \big( \mu_k(\T_k\setminus \T_{k+1})+\delta^{2-\gamma} \dd_{k,k+1}\big)\widehat{e}_{k+1}
\end{align}
   with $C_{\rm sa}:=2^{3/2}\max\{C_{\rm c},C_{\rm s},C_{\rm a}\}$. Denote $C_{\rm RL}^2:=1+C_{\rm REL}^2$.
  The Young's inequality (used twice) with $\epsilon=2^{-3}\varepsilon C_{\rm RL}^{-2}$,
  $a=C_{\rm sa} \mu_k(\T_k\setminus\T_{k+1})$, and $b=\widehat{e}_{k+1}$ and $\epsilon=2$, $a=\dd_{k,k+1}$, and $b=C_{\rm sa}\delta^{2-\gamma} \widehat{e}_{k+1}$ for the first and second terms on the right-hand side of \eqref{45}, respectively,  yield
	\begin{align*}	\dd_{k,k+1}^2+e_{k+1}^2-e_k^2\leq & ~2^2\varepsilon^{-1}\Lambda_6^2\mu_k^2(\T_k\setminus\T_{k+1})+ \big(2^{-4}\varepsilon C_{\rm RL}^{-2} +C_{\rm sa}^2\delta^{2(2-\gamma)}\big)\widehat{e}_{k+1}^2 +\dd_{k,k+1}^2/4
	\end{align*}
 with $\Lambda_6^2:=C_{\rm sa}^2 C_{\rm RL}^2$.
	The estimate $C_{\rm sa}^2\delta^{2(2-\gamma)}\leq 2^{-4} \varepsilon C_{\rm RL}^{-2}$ from \eqref{cndn2} shows
\begin{align*}
	\frac{3}{4}\dd_{k,k+1}^2+e_{k+1}^2-e_k^2\leq & ~2^2\varepsilon^{-1}\Lambda_6^2\mu_k^2(\T_k\setminus\T_{k+1})+2^{-3}\varepsilon C_{\rm RL}^{-2} \widehat{e}_{k+1}^2.
\end{align*}
Substitute \eqref{cora12} in the first term of the right-hand side in the above displayed inequality to obtain
\begin{align*}
	\frac{3}{4}\dd_{k,k+1}^2+e_{k+1}^2-e_k^2\leq & ~2^4\varepsilon^{-1}\Lambda_6^2\big(\mu_{k}^2-\mu_{k+1}^2\big)+ 2^4\varepsilon^{-1} \Lambda_6^2 \Lambda_0 \delta^2\dd_{k,k+1}\big ( \mu_{k+1}+\mu_k \big ) \\
	&+3\times2^4 \varepsilon^{-1}\Lambda_6^2\Lambda_0^2 \delta^4\dd_{k,k+1}^2+2^{-3}\varepsilon C_{\rm RL}^{-2} \widehat{e}_{k+1}^2.
\end{align*}
 The estimate $3\times 2^4 \varepsilon^{-1}\Lambda_6^2\Lambda_0^2 \delta^4\leq 1/4$ from \eqref{cndn2} leads to
\begin{align}
	\frac{1}{2}\dd_{k,k+1}^2+e_{k+1}^2-e_k^2\leq  2^4\varepsilon^{-1}\Lambda_6^2\big(\mu_{k}^2-\mu_{k+1}^2\big)+ 2^4\varepsilon^{-1}\Lambda_6^2\Lambda_0 \delta^2\dd_{k,k+1}\big ( \mu_{k+1}+\mu_k \big )+2^{-3}\varepsilon C_{\rm RL}^{-2} \widehat{e}^2_{k+1}.\label{NNcal30}
 \end{align}
 A  Young's inequality with  $\epsilon=2^{-4}\varepsilon$, $a=2^4 \varepsilon^{-1} \Lambda_6^2 \Lambda_0 \dd_{k,k+1}$, and $b=\mu_{k+1}+\mu_k$ for the second term in the above inequality shows
\begin{align} \label{eqnew}
2^4\varepsilon^{-1}\Lambda_6^2\Lambda_0 \delta^2\dd_{k,k+1}\big ( \mu_{k+1}+\mu_k \big )\leq {2^{11}}\varepsilon^{-3}\Lambda_6^4\Lambda_0^2 \delta^4\dd_{k,k+1}^2 +2^{-4}\varepsilon \big ( \mu_{k+1}^2+\mu_k^2 \big ).
   \end{align}
Utilize the bound ${2^{11}}\varepsilon^{-3}\Lambda_6^4\Lambda_0^2\delta^4\leq 1/4$ from \eqref{cndn2} and substitute \eqref{eqnew} in \eqref{NNcal30} to establish
\begin{align*}
\dd_{k,k+1}^2+4\big(e_{k+1}^2-e_k^2\big)\leq & 2^6\varepsilon^{-1}\Lambda_6^2\big(\mu_{k}^2-\mu_{k+1}^2\big)+2^{-2}\varepsilon \big (\mu_{k+1}^2+\mu_k^2 \big )+\frac{\varepsilon}{2} C_{\rm RL}^{-2} \widehat{e}_{k+1}^2.
\end{align*}
A (partly telescoping) summation  from $\ell$ to $\ell+n$ on both sides lead to $e_{\ell+n+1}^2-e_\ell^2$ on the left-hand side and $\mu_{\ell}^2-\mu_{\ell+n+1}^2$ on the right-hand side in
\begin{align}
\sum_{k=\ell}^{\ell+n}\dd_{k,k+1}^2+&4\big(e_{\ell+n+1}^2-e_\ell^2\big)\leq 2^6\varepsilon^{-1}\Lambda_6^2\big (\mu_{\ell}^2-\mu_{\ell+n+1}^2\big )+\frac{\varepsilon}{4} \sum_{k=\ell}^{\ell+n}\big (\mu_{k+1}^2+\mu_k^2 \big )+\frac{\varepsilon}{2}   C_{\rm RL}^{-2} \sum_{k=\ell}^{\ell+n}\widehat{e}_{k+1}^2. \label{eqn:depart}
\end{align}

 \medskip \noindent {\it Step 2.}~(Conclusion of \eqref{A4ep}).~
Rearrange the terms in \eqref{eqn:depart} to obtain
\begin{align}
\sum_{k=\ell}^{\ell+n}\dd_{k,k+1}^2\leq &~ 4e_\ell^2 +2^6\varepsilon^{-1} \Lambda_6^2\mu_{\ell}^2 + \big(\frac{\varepsilon}{2} C_{\rm RL}^{-2}-4\big) e_{\ell+n+1}^2 +\frac{\varepsilon}{2} C_{\rm RL}^{-2}\osc_{\ell+n+1}^2+\big(\frac{\varepsilon}{4}-2^6\varepsilon^{-1}\Lambda_6^2\big)\mu_{\ell+n+1}^2\nonumber\\
&~+\frac{\varepsilon}{2}  \sum_{k=\ell}^{\ell+n}\mu_k^2+\frac{\varepsilon}{2}  C_{\rm RL}^{-2} \sum_{k=\ell}^{\ell+n-1}
\widehat{e}_{k+1}^2.\label{nn01}
    \end{align}
    The definition of the projection operator $\Pi_{2,k+1}$ shows $\norm{v-\Pi_{2,k+1}v}=\inf_{q\in \mathcal{P}_{2}(\T_{k+1})}\norm{v-q}$. The choice $v=\yb_{k+1}-\yd$ and $q=0$ lead to 
   \begin{align*}
    {\rm osc}^2(\yb_{k+1}-\yd,\mathcal{T}_{k+1})&=\norm{h_{\mathcal{T}_{k+1}}^2(1-\Pi_{2,k+1})(\yb_{k+1}-\yd)}^2\leq \norm{h_{\mathcal{T}_{k+1}}^2(\yb_{k+1}-\yd)}^2\leq \mu_{k+1}^2
    \end{align*}
    with $h_{\mathcal{T}_{k+1}}|_T:=|T|^{1/2}$ for $T\in \T_{k+1}$. 
    Utilize the bound $\osc_{\ell+n+1}\leq \mu_{\ell+n+1}$  in the fourth term of the right-hand side of \eqref{nn01} to obtain
    \begin{align}
\sum_{k=\ell}^{\ell+n}\dd_{k,k+1}^2 \leq &~ {4e_\ell^2+2^6\varepsilon^{-1}\Lambda_6^2 \mu_{\ell}^2+(\frac{\varepsilon}{4} C_{\rm RL}^{-2}-4)e_{\ell+n+1}^2+\big(\frac{\varepsilon}{2} C_{\rm RL}^{-2}+\frac{\varepsilon}{4}-2^6\Lambda_6^2\varepsilon^{-1}\big)\mu_{\ell+n+1}^2}\nonumber\\
&~+\frac{\varepsilon}{2} \sum_{k=\ell}^{\ell+n}\mu_k^2+\frac{\varepsilon}{2} C_{\rm RL}^{-2} \sum_{k=\ell}^{\ell+n-1}  \widehat{e}_{k+1}^2.\nonumber
\end{align}
This with  $\frac{1}{2}\varepsilon C_{\rm RL}^{-2}-4\leq 0$ and $\frac{1}{2}\varepsilon C_{\rm RL}^{-2}+\frac{1}{4}\varepsilon-2^6\varepsilon^{-1}\Lambda_6^2 \leq 0$ from \eqref{cndn2}, $e_{k+1}^2\leq C_{\rm REL}^2 \eta_{k+1}^2$, $\osc_{k}^2\leq\mu_{k}^2\leq \eta_{k }^2$,  and $\Lambda_{4(\varepsilon)}(\alpha^{-2}):=4C_{\rm RL}^2+2^6\varepsilon^{-1}\Lambda_6^2$ in the above displayed estimate  show
\begin{align}
\sum_{k=\ell}^{\ell+n}\dd_{k,k+1}^2\leq & \Lambda_{4(\varepsilon)} \eta_{\ell}^2+\frac{\varepsilon}{2} \sum_{k=\ell}^{\ell+n}\mu_k^2+ \frac{\varepsilon}{2} C_{\rm RL}^{-2} \sum_{k=\ell}^{\ell+n-1}  \widehat{e}_{k+1}^2
 \nonumber \\
	\leq & \Lambda_{4(\varepsilon)} \eta_{\ell}^2+\frac{\varepsilon}{2} \sum_{k=\ell}^{\ell+n}\mu_k^2 +\frac{\varepsilon}{2}  \sum_{k=\ell}^{\ell+n-1}\eta_{k+1}^2\nonumber 
	\leq \Lambda_{4(\varepsilon)} \eta_{\ell}^2+\varepsilon \sum_{k=\ell}^{\ell+n}\eta_k^2.
			\end{align}
   This concludes the proof.
\end{proof}
\begin{remark}[assumptions for axioms]\label{delta}
\noindent {\rm (i).}~The axioms, {\bf (A1)}-{\bf (A2)}, $({\rm \bf A3}_{\varepsilon})$-$({\rm \bf A4}_{\varepsilon})$ with $\delta_0:=\min\{\delta_1,\delta_2\}$ proves Theorem~\ref{thm:2.2} which shows the optimal convergence of the AFEM for the optimal control problem \cite[Theorem A.1]{CS24}. {\rm (ii).}~Lemma~\ref{lem_A4} provides the key estimates to establish $({\rm \bf A4}_{\varepsilon})$; however, when $\z\notin {\mathcal N}^i(\T_0)$, it is difficult to handle the term $(J_{k+1}I_{k+1}-J_kI_k)(\yb_{k+1}({\rm z})-\yb({\rm z}))$ in \eqref{obstacle} and hence the initial triangulation is constructed such that point sources are included in the nodes. Notice that only the axiom, $({\rm \bf A4}_{\varepsilon})$ is verified under this assumption $\Z\subseteq {\mathcal N}^i(\T_0)$.
\end{remark}

\section{Point-wise tracking problem {\bf (P2)}}\label{p2ocp}
The a~priori and a~posteriori error control and the convergence of the adaptive finite element method for the point-wise tracking OCP are discussed in this section. 
 {Recall the OCP in \eqref{opt1n} with the desired state ${\bf y}_{\rm d}:=\{y_{ \zeta\rm,d}\}_{\zeta\in\D}$} that seeks $(\yb,\ub)$ which minimizes the cost functional
\begin{align}  
\cJ(y,u):=&\frac{1}{2}\sum_{\zeta\in\D}|y(\zeta)-\zy|^2+\frac{\alpha}{2}\norm{u}_{L^2(\Omega)}^2\no\\
\text{ subject to }
\Delta^2 y=&f+u\text{ in }\Omega\text{ and } y=\frac{\partial y}{\partial \nu} =0 \text{ on } \partial \Omega\no
\end{align}	
in the admissible set $U_{\rm ad}=\{u\in L^2(\Omega):u_a\leq u(x)\leq u_b\text{ a.e. in } \Omega\}$.  For all $\phi\in V$, $\phi_\rM\in V_\rM$, and $v\in U_{\rm ad}$, the weak and discrete problems seek 
$(\yb,\pb,\ub)\in V\times V\times U_{\rm ad}$ and $(\bar{y}_\rM, \bar{p}_\rM, \bar{u}_{\h})\in V_\rM \times V_\rM \times U_{\rm ad}$ such that
\begin{equation}
\label{eq2:discrete}
\left.
\begin{array}{l@{\hspace{1cm}}l}
a(\yb,\phi) = (f+\ub,\phi) &
a_\pw(\bar{y}_\rM, \phi_\rM) = (f+\ub_\h, \phi_\rM) \\[1ex]
a(\phi,\pb) = \sum_{\zeta\in\D} (\yb(\zeta) - \zy) \langle \delta_\zeta, \phi \rangle &
a_\pw(\phi_\rM, \pb_\rM) = \sum_{\zeta\in\D} (\yb_\rM(\zeta) - \zy) \langle \delta_\zeta, \phi_\rM \rangle \\[1ex]
(\pb + \alpha \ub, v - \ub) \geq 0, &
(\pb_\rM + \alpha \ub_{\rm h}, v - \bar{u}_{\rm h}) \ge 0.
\end{array}
\right\}
\end{equation}
 \noindent {\it Throughout the analysis, we assume that the atoms of the Dirac measure, $\zeta\in \D$ are located at the nodes of the initial triangulation, ${\mathcal{N}}^i(\T_0)$.} {The Morley function $\phi_\rM$ is evaluated at the vertices $\zeta$ of $\T_0$. Choosing $\zeta\in {\cal N}^i(\T_0)$ guarantees that $\zeta$ will continue to be a vertex in the further refinements, thus leading to $\phi_\rM(\zeta)$ remaining single-valued.}
This ensures that the expression on the right-hand side of the discrete adjoint equation is well-defined.
The existence and uniqueness of the solution of the optimal control problems in \eqref{eq2:discrete} follow from \cite[Section 5.3]{AAS21}.

\medskip
\noindent 
The proof of the error estimates is based on an equivalence result (analogous to Theorem \ref{equivalence}) and a~priori and a~posteriori error estimates for the new auxiliary problems defined by: Seek $\ty,\tp\in V$ such that
$$a(\ty,\phi)=(f+\ub_\h,\phi) \text{ and }a(\phi,\tp)=\sum_{\zeta\in\D}(\yb_\rM(\zeta)-y_{\zeta\rm,d}) \langle\delta_\zeta ,\phi\rangle ~\text{for all }\phi\in V.$$
{In the case of \eqref{opt1}, the Dirac measure appears in the discrete state equation \eqref{state_eq1}, while the right-hand side of the discrete adjoint equation \eqref{dadj} lies in $L^2(\Omega)$. In contrast, for \eqref{opt1n}, the discrete state equation has an $L^2(\Omega)$ source term, whereas the right-hand side of the discrete adjoint equation consists of a finite sum of Dirac measures (see \eqref{eq2:discrete}). Accordingly, the roles of Dirac measure and $L^2(\Omega)$ data on the right-hand side of the continuous and discrete versions of auxiliary equations 
in \eqref{opt1} and 
in \eqref{opt1n} are interchanged. Therefore, to estimate the errors in the state variable for \eqref{opt1n}, the procedure used for the adjoint variable in \eqref{opt1} is applied, and vice versa.} 

 \begin{thm}[error control] \label{OCPap}
    The continuous and discrete solutions $(\yb,\pb,\ub)$ and
   $(\yb_\rM,\pb_\rM,\ub_\h)$
   to \eqref{eq2:discrete} satisfy
    (a) the a~priori error estimates below for $(\yb,\pb,\ub)
   \in  (V \cap H^{4-\gamma}(\Omega))^2\times U_{\ad}$:
    \begin{align*}
      &\norm{\bar{u}-\bar{u}_\h}+\enorm{\yb-\yb_\rM}_\pw+\enorm{\pb-\pb_\rM}_\pw\leq  {C_{\rm AE}h^{2-\gamma}\big(\norm{f}+\max\{|{u_a}|,|{u_b}|\} +\norm{\by_{\rm d}}_{L^\infty({\mathbb R}^m)}\big),}\\
      &\norm{\bar{u}-\bar{u}_\h}+\norm{\yb-\yb_\rM}_{H^{\gamma}(\T)} +\norm{\pb-\pb_\rM}_{H^{\gamma}(\T)}\leq  C_{\rm AP} h^{2-\gamma}\big(\norm{\bar{u}-\bar{u}_\h} +\enorm{\yb-\yb_\rM}_\pw \\
      &\hspace{8.8cm}+\enorm{\pb-\pb_\rM}_\pw+h^2(\norm{f}+\max\{|{u_a}|,|{u_b}|\})\big),
    \end{align*}
    (b) reliable and efficient a~posteriori error estimates stated as: 
   \begin{align*}
&C_{\rm REL}^{-2}(\norm{\bar{u}-\bar{u}_\h}^2+\enorm{\yb -\yb_\rM}_{\pw}^2+\enorm{\pb -\pb_\rM}_{\pw}^2) \leq  \eta^2 \\
     &\leq  C_{\rm EFF}^2\big( \norm{\bar{u}-\bar{u}_\h}^2+\enorm{\yb -\yb_\rM}_{\pw}^2 + \enorm{\pb-\pb_\rM}_{\pw}^2 +h^{4}\max\{|u_a|,|u_b|\} +\norm{h_\T^2(1-\Pi_{k,\T}) f)}^2\big).
\end{align*}
    \end{thm}
 \noindent  The proofs follow similar steps as in the proofs of Theorem \ref{apriori}-Theorem \ref{aposteriori}.
 \begin{tcolorbox}[colback=gray!5, colframe=black, boxrule=0.4pt, arc=2pt,
  left=6pt, right=6pt, top=6pt, bottom=6pt]
For any $T \in \T$, define the volume and edge estimators as 
\begin{align*}
\eta_{S,T}^2 & := h_T^4 \norm{f + \ub_{\rm h}}_{L^2(T)}^2,
\;
\eta_{S,\E(T)}^2 := \sum_{E \in \E(T)} h_T \norm{[D_{\rm pw}^2 \yb_\rM]_E \tau_E}_{L^2(E)}^2, \text{ and } \;\\ 
 \eta_{A,\E(T)}^2 & := \sum_{E \in \E(T)} h_T \norm{[D_{\rm pw}^2 \pb_\rM]_E \tau_E}_{L^2(E)}^2.
\end{align*}
The a~posteriori error estimator for \eqref{opt1n} reads 
$
\eta^2 := \eta_S^2 + \eta_A^2
$
with
\[
\eta_S^2 := \sum_{T\in \T} \big( \eta_{S,T}^2 + \eta_{S,\E(T)}^2 \big) \text{ and }
\eta_A^2 := \sum_{T\in \T} \eta_{A,\E(T)}^2.
\]
\end{tcolorbox}

\medskip
\noindent  The quasi-optimality of AFEM for $\eqref{opt1n}$ is established by verifying the axioms {\bf (A1)}, {\bf (A2)}, {\bf (A3)}, and $({\rm \bf  A4}_{\varepsilon})$ with the distance function $\dd^2(\T,\widehat{\T}):=\norm{\widehat{\bar{u}}_\h-\bar{u}_\h}^2+\enorm{\widehat{\yb}_\rM-\yb_\rM}_{\pw}^2 +\enorm{\widehat{\pb}_\rM-\pb_\rM}_{\pw}^2$(ref. \cite[Theorem 3.1]{CCRH17} and \cite[Theorem 4.1]{CMMD2014}). The proofs of axioms are similar to the proofs from Section \ref{AOA}, and hence only a brief sketch is provided below. {Note that, since $\D\in {\mathcal{N}}^i(\T_0)$, the proof of {\it discrete reliability} becomes more simpler (see Remark \ref{directa3}).}

\begin{lem}[intermediate results for $({\rm \bf  A4}_{\varepsilon})$]
      Let $(\yb_\ell,\pb_\ell,\ub_\ell)$ be the solution of the discrete optimality system \eqref{eq2:discrete} with respect to $\T_\ell\in\mathbb{T}(\delta)$. Then there exist $C_{\rm c}$, $C_{\rm s}$, and $C_{\rm a}$ $>0$ such that 
\begin{itemize}
\item[(a)] $ ({\ub}_{k+1}-{\ub}_k,{\ub}_{k+1}-{\ub}) \leq C_{\rm c}   \delta^{2-\gamma}\enorm{\pb_{k+1}-\pb_k}_{\pw}\widehat{e}_{k+1}$,
\item[(b)] $a_\pw(\yb_{k+1}-\yb_k,\yb_{k+1}-\yb) \leq C_{\rm s}  \big ( \mu_{k}(\T_k\setminus \T_{k+1})+\delta^{2-\gamma}\enorm{\pb_{k+1}-\pb_k}_{\pw}\big )\widehat{e}_{k+1}$ ,  
\item[(c)] $a_\pw(\pb_{k+1}-\pb_k,\pb_{k+1}-\pb)\leq C_{\rm a} \delta^{2-\gamma}\enorm{\yb_{k+1}-\yb_k}_{\pw}\widehat{e}_{k+1}$.
      \end{itemize}
  \end{lem}
  \noindent The inequality in $(a)$ follows the proof of Lemma \ref{lem_A4}$(a)$ with $C_{\rm c}(\alpha^{-3}):=\alpha^{-2}C_{\rm emb}C_{\rm PI} C_{\rm AP}$; $(b)$ follows the proof of Lemma \ref{lem_A4}$(c)$ with $C_{\rm s}(\alpha^{-2}):=\max\{2C_\rI, C_\rI C_{\rm PI}|\Omega|^{2-\gamma},\alpha^{-1}C_{\rm PI}C_{\rm AP}\}$, and $(c)$ follows the proof of  Lemma \ref{lem_A4}$(b)$ with $C_{\rm a}(\alpha^{-1}):=n C_{\rm emb}^2 (1+ C_{\rm PI})C_{\rm AP}$. 
\begin{thm}[verification of axioms]
    For all $\T,\widehat{\T}\in {\mathbb{T}}(\delta)$ with $\zeta\in \D$,  {\bf (A1)}- {\bf (A3)}, and $({\rm \bf  A4}_{\varepsilon})$ hold. 
\end{thm}
The proofs of the axioms {\bf{(A1)}} and {\bf{(A2)}} 
 follow using similar arguments as in Section \ref{A1}. Similarly,  \eqref{cora12} in  Remark \ref{a12_rem} follows by redefining $\mu^2:=\sum_{T\in \T}\eta^2_{S,T}$. \\
\noindent {\it Sketch of the proof of {\bf (A3).}}~Let $(\yb_\rM,\pb_\rM,\bar{u}_\h)\in V_\rM \times V_\rM \times U_{\rm ad}$ (resp. $(\widehat{\yb}_\rM,\widehat{\pb}_\rM,\widehat{\bar{u}}_\h)\in \widehat{V}_\rM \times \widehat{V}_\rM \times U_{\rm ad}$) solve \eqref{eq2:discrete} with respect to  $\T$ (resp. $\widehat{\T}$) and $\widehat{\ty}_\rM,\widehat{\tp}_\rM$ solve 
\begin{align}
a(\widehat{\ty}_\rM,\widehat{\phi}_\rM)=(f+\ub_\h,\widehat{\phi}_\rM) \text{ and }a(\widehat{\phi}_\rM,\widehat{\tp}_\rM)=\sum_{\zeta\in\D}(\yb_\rM(\zeta)-y_{\zeta\rm,d}) \langle\delta_\zeta ,\widehat{\phi}_\rM\rangle ~\text{for all }\widehat{\phi}_\rM\in \widehat{V}_\rM.\no
\end{align}
Following the steps in Lemma \ref{disequiv}, we have the following error bound: 
   $$ \norm{\widehat{\bar{u}}_\h-\bar{u}_\h}+\enorm{\widehat{\yb}_\rM-\yb_\rM}_\pw+\enorm{\widehat{\pb}_\rM-\pb_\rM}_\pw\approx \enorm{\widehat{\ty}_\rM-\yb_\rM}_\pw+\enorm{\widehat{\tp}_\rM-\pb_\rM}_\pw. $$

\noindent Now follow the estimates in Step 2 in the proof of $({\rm \bf A3}_{\varepsilon})$ in Section \ref{disrel} to obtain
$\enorm{\widehat{\ty}_\rM-\yb_\rM}_\pw\lesssim \eta_S(\mathcal R\mathcal{(\cT,\widehat{\cT})}).$
The results in Step 1 in the proof of $({\rm \bf A3}_{\varepsilon})$ in Section \ref{disrel} yield $\enorm{\widehat{\tp}_\rM-\pb_\rM}_\pw\lesssim \eta_A(\mathcal R\mathcal{(\cT,\widehat{\cT})}).$
These two estimates conclude the proof of {\bf(A3)}. \qed


\noindent {The proof of $({\rm \bf  A4}_{\varepsilon})$ remains same as in Theorem~\ref{thma4e} by redefining $\osc_\ell:=\osc(f+\ub_\ell,\T_\ell)$.}


\section{Numerical experiments}\label{sec:num}
In this section, we present numerical experiments that verify the theoretical results for \eqref{opt1}-\eqref{opt2n} defined over convex and non-convex domains using uniform and adaptive mesh refinements.

\medskip
\noindent Each numerical experiment is initialized with a sufficiently fine triangulation to ensure adaptive convergence, with specific details provided in each example. To compute the discrete solution, we employ an iterative projected gradient method as described in \cite[Algorithm 4.2]{HM2005}. This method requires an initial guess for the control variable $\bar u_h$ and is followed by the computation of $\bar y_\M$ and $\bar p_\M$. The initial guess for the control variable is obtained by averaging the control constraints within the admissible space. In all the examples, we assume that the source points act on the nodes of the initial triangulation.
For uniform mesh refinement, the red-refinement technique is used. For adaptive mesh refinement, we adopt the standard adaptive FEM loop that consists of SOLVE, ESTIMATE, MARK, and REFINE.
    %
%
In the estimate step, we use the a~posteriori error estimators derived in Theorem~\ref{aposteriori} for \eqref{opt1}, Theorem~\ref{OCPap}$(b)$ for \eqref{opt1n}, and Theorem~\ref{OCPapn}$(b)$ for \eqref{opt2n}. For marking, the D{\"o}rfler's marking strategy \cite{WD96} is applied with a bulk parameter of 0.2. The adaptive refinement of the marked triangles is then carried out using the newest-vertex bisection method.

\medskip
\noindent We compute the errors in the state, adjoint, and control variables over both uniform and adaptive mesh refinements. The errors in the state and adjoint variables are measured in the energy norm, while the control error is evaluated in the Euclidean (or $L^2$) norm. The total error is calculated as the sum of these three individual errors. Convergence rates are determined with respect to the number of degrees of freedom (NDOF) over triangulation $\T$, where NDOF denotes only Morley degrees of freedom. The numerical experiments for \eqref{opt1}–\eqref{opt2n} and their corresponding results are presented in Subsections \ref{Ex1_OCP}–\ref{pwOCP2}, respectively. 

\medskip
\noindent Many interesting examples are constructed to validate the theoretical results and observe the behavior of the cost functions, as benchmark examples are {\it not }available in the literature. The details of the examples are summarized in the Table \ref{tab:examples-summary}.
\begin{table}[htbp]
\centering
\scalebox{0.5}{
\renewcommand{\arraystretch}{1.2} 
\setlength{\tabcolsep}{8pt} 

\newcolumntype{Y}{>{\centering\arraybackslash}m{2.8cm}} 
\newcolumntype{Z}{>{\centering\arraybackslash}m{1.8cm}} 
\newcolumntype{A}{>{\centering\arraybackslash}m{1cm}} 
\newcolumntype{B}{>{\centering\arraybackslash}m{2cm}} 
\newcolumntype{C}{>{\centering\arraybackslash}m{2.2cm}} 

\begin{tabularx}{\textwidth}{|A|B|C|X|Y|Z|}
\hline
\textbf{Sl. No.} & \textbf{Example} & \textbf{Problem} & \textbf{Domain} & \textbf{Figure with $\Z(\times)$ and $\D(\circ)$} & \textbf{Exact Solution} \\
\hline
1 & Example \ref{ex1} & \eqref{opt1} & \vspace{-1cm} Circular domain with one source point & \vspace{0.1cm}
\begin{tikzpicture}[scale=2, baseline={(0,-0.3)}]
  \draw[thick] (0,0) circle(0.6);
  \node at (0,0) {\large $\times$};
\end{tikzpicture} \vspace{0.1cm}& Available \\
\hline 
2 & Example \ref{ex2} & \eqref{opt1} &\vspace{-0.8cm} Square domain with 9 source points & \vspace{0.2cm}
\begin{tikzpicture}[scale=1.5, baseline={(0,-0.3)}]
  \draw[thick] (0,0) rectangle (1.2,1.2);
  \foreach \x in {0.3,0.6,0.9}
    \foreach \y in {0.3,0.6,0.9}
      \node at (\x,\y) {\large $\times$};
\end{tikzpicture} & Not available \\
\hline
3 & Example \ref{ex3} & \eqref{opt1} & \vspace{-0.5cm}$\Gamma$-shaped domain with source point at $(0.5, 0.5)$ & \vspace{0.1cm}
\begin{tikzpicture}[scale=1, baseline={(0,-0.3)}]
  \draw[thick]
    (-1,-1) -- (-1,1) -- (1,1) -- (1,0) -- (0,0) -- (0,-1) -- cycle;
  \node at (0.5,0.5) {\large $\times$};
\end{tikzpicture} \vspace{0.0cm}& Not available \\
\hline
4 & Example \ref{ex4} & \eqref{opt1n} & \vspace{-1cm}Circular domain with state solution approaching desired state,  ${\bf y}_{\rm d}$ at $(0,0)$ and distributed control & \vspace{0.3cm}
\begin{tikzpicture}[scale=2, baseline={(0,-0.3)}]
  \draw[thick] (0,0) circle(0.6);
  \node at (0,0) {\large $\circ$};
\end{tikzpicture} & Available \\
\hline
5 & Example \ref{ex5} & \eqref{opt1n} & \vspace{-1.2cm}$\Gamma$-shaped domain, with changing locations of points in $\D$, $\D_1=(-0.5,-0.5)$, $\D_2=(-0.5,0.5)$, $\D_3=(0.5,0.5)$ & \vspace{0.26cm}
\begin{tikzpicture}[scale=1.2, baseline={(0,-0.3)}]
  \draw[thick]
    (-1,-1) -- (-1,1) -- (1,1) -- (1,0) -- (0,0) -- (0,-1) -- cycle;
  \node at (-0.5,-0.5) {\large $\circ$};
  \node at (-0.5,0.5) {\large $\circ$};
  \node at (0.5,0.5) {\large $\circ$};
\end{tikzpicture} & Not available \\
\hline
6 & Example \ref{ex6} & \eqref{opt2n} & \vspace{-0.8cm}Increasing the distance between $\mathcal{Z}$ and $\mathcal{D}$ & \vspace{0.1cm}\begin{tikzpicture}[scale=1.8]
  \def\s{1.2}
  \draw[thick] (0,0) rectangle (\s,\s);
  \node at (0.25*\s,0.25*\s) {\large $\times$};
  \node at (0.75*\s,0.75*\s) {\large $\circ$};
\end{tikzpicture} & Not available \\
\hline
7 & Example \ref{ex7} & \eqref{opt2n} &\vspace{-0.9cm} Increasing the cardinality of $\Z$ while keeping $\Z$ fixed & \vspace{0.1cm}\begin{tikzpicture}[scale=1.8]
  \def\s{1.2}
  \draw[thick] (0,0) rectangle (\s,\s);
  \node at (0.5*\s,0.5*\s) {\large $\circ$};
  \node at (0.25*\s,0.25*\s) {\large $\times$};
  \node at (0.75*\s,0.25*\s) {\large $\times$};
  \node at (0.25*\s,0.75*\s) {\large $\times$};
  \node at (0.75*\s,0.75*\s) {\large $\times$};
\end{tikzpicture} & Not available \\
\hline
8 & Example \ref{ex8} & \eqref{opt2n} & \vspace{-0.9cm}Increasing the cardinality of $\Z$ while keeping $\D$ fixed & \vspace{0.2cm}\begin{tikzpicture}[scale=1.8]
  \def\s{1.2}
  \draw[thick] (0,0) rectangle (\s,\s);
  \node at (0.5*\s,0.5*\s) {\large $\times$};
  \node at (0.25*\s,0.25*\s) {\large $\circ$};
  \node at (0.75*\s,0.25*\s) {\large $\circ$};
  \node at (0.75*\s,0.75*\s) {\large $\circ$};
  \node at (0.25*\s,0.75*\s) {\large $\circ$};
\end{tikzpicture} & Not available \\
\hline
\end{tabularx}
}
\caption{Summary of examples, problem settings, and availability of exact solutions.}
\label{tab:examples-summary}
\end{table}

\subsection{OCP with point sources \eqref{opt1}}\label{Ex1_OCP}

This subsection presents three examples with point sources: the first two examples are defined on convex domains, while the third focuses on a non-convex domain. For the first example, the exact solution is known, thus allowing for a direct computation of the error. However, for the remaining two examples, where exact solutions are unavailable, errors are estimated using a reference solution computed on a highly refined mesh.
\begin{example}{(\bf Circular domain with one point source).}\label{ex1}
   Consider \eqref{opt1} defined on a circular domain $\Omega \subset \mathbb{R}^2$ centered at the origin with radius $r = e^{-1/2}$, where $r^2 := x^2 + y^2$. The point source is located at the center of the circle. The desired state is given by  
$\displaystyle
y_{\rm d} = -e^{-2} \frac{r^2}{8\pi} \log(r) - \frac{1}{16\pi} e^{-3} - 0.064.
$ 
The regularization parameter $\alpha$ is chosen as $10^{-3}$ and control constraints are selected as ${\bf a} = \{-1\}$, ${\bf b} = \{0\}$.  The exact solutions for the state and adjoint equations  (constructed based on the fundamental solution of the biharmonic equation) and exact control are given by
$\displaystyle
\bar{y} = -e^{-2} \frac{r^2}{8\pi} \log(r) - \frac{1}{16\pi} e^{-3} , \:
\bar{p} = 10^{-3} (e^{-1} - r^2)^2 , \text{ and }\bar{\bu}~=~-e^{-2},
$ respectively.

\medskip
\noindent This test case is defined on a smooth circular domain, which implies that the elliptic regularity index is $\sigma = 1$  with $\gamma = 1$. Theorem~\ref{apriori} predicts a quadratic convergence rate for the control error, and a linear convergence rate for the state and adjoint errors in the energy norm with respect to mesh-size. In this numerical experiment, the initial mesh is generated by applying one red refinement to the four-triangle mesh in diamond shape. At each mesh refinement, we project the boundary nodes to the circle to approximate the circular domain. The errors and their convergence rates for uniform mesh refinement are reported in Table~\ref{tab:OCP_Ex1_UniErr}, showing that the error in the control variable converges optimally at order 1, while the total error converges at the optimal rate of 0.5 with respect to the NDOF, thus validating the theoretical predictions in Theorem~\ref{apriori}.

 \begin{table}[ht!]
     	\centering \footnotesize
     	\begin{tabular}{|c|c|c|c|c|c|c|c|c|}
      \hline
      NDOF & $\enorm{\yb-\yb_\rM}_\pw$ & Order &$\enorm{\pb-\pb_\rM}_\pw$ & Order & $\norm{{\bar{\bu}}-{\bar{\bu}_\h}}_{\mathbb{R}^n}$ & Order & Total error & Order  \\
\hline\hline
             $25$ & $0.02529$ & -- & $0.00217$ & -- & $0.11859$ & -- & $0.14606$ & -- \\
     			$113$ & $0.00948$ & $0.65$ & $0.00117$ & $0.41$ & $0.03082$ & $0.89$ & $0.04147$ & $0.83$ \\
     			$481$ & $0.00448$ & $0.52$ & $0.00060$ & $0.46$ & $0.00779$ & $0.95$ & $0.01287$ & $0.81$ \\
     			$1985$ & $0.00234$ & $0.46$ & $0.00030$ & $0.48$ & $0.00195$ & $0.98$ & $0.00459$ & $0.73$ \\
     			$8065$ & $0.00125$ & $0.45$ & $0.00015$ & $0.49$ & $0.00049$ & $0.99$ & $0.00189$ & $0.63$ \\
     			$32513$ & $0.00067$ & $0.45$ & $0.00008$ & $0.50$ & $0.00012$ & $0.99$ & $0.00086$ & $0.56$ \\
        $130561$ & $0.00035$ & $0.46$ & $0.00003$ & $0.50$ & $0.00003$ & $1.00$ & $0.00042$ & $0.52$ \\
     			$523265$ & $0.00018$ & $0.46$ & $0.00002$ & $0.50$ & $0.00001$ & $1.00$ & $0.00021$ & $0.49$ \\
        \hline
     	\end{tabular}
       \caption{Errors and their convergence rates with uniform refinement in Example \ref{ex1}}
\label{tab:OCP_Ex1_UniErr}
     \end{table} 

\medskip
\noindent In case of adaptive mesh refinement, the a~posteriori error estimator in \eqref{elemest} is considered. Figure~\ref{fig1Ex1_OCP} shows the adaptive mesh at the 20th iteration (Iter) along with the corresponding state and adjoint solutions. As expected, the mesh exhibits increased refinement around the origin, since the point source is located at the origin. Table~\ref{tab:OCP_Ex1_Err} presents the errors in the state, adjoint, and control variables, along with the a~posteriori error estimators and their corresponding convergence rates. These results confirm the optimal convergence rates 0.5 for total error and estimator, as predicted in Theorem~\ref{aposteriori}. {It is observed that in the total error, the error in the state solution is dominating the errors obtained from the adjoint and control solutions. }
\begin{figure}[ht!] 
\begin{tabular}{ccc} 
\includegraphics[width=5cm,height=4cm]{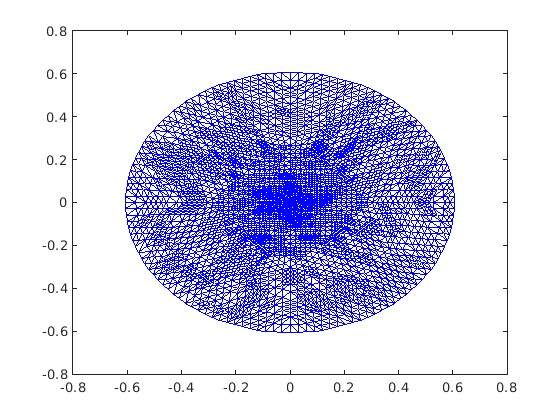} 
& \includegraphics[width=5cm,height=4cm]{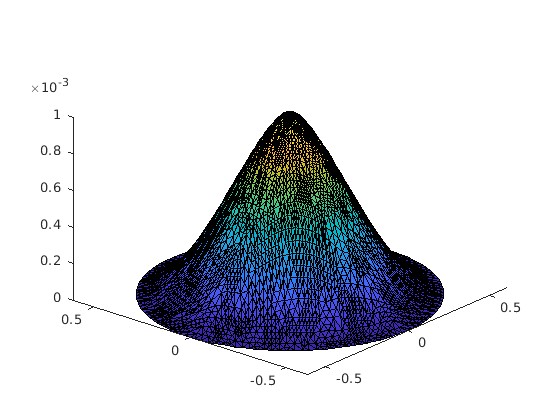}
& \includegraphics[width=5cm,height=4cm]{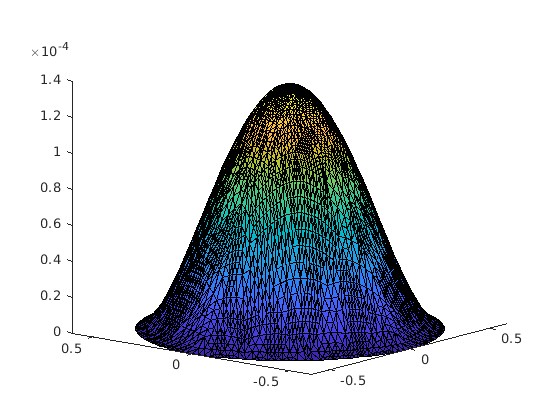}\\
(a) Adaptive Mesh &
(b) State solution &
(c) Adjoint solution 
	 \end{tabular}   	
  \caption{\small   The adaptive mesh at the $20{\rm th}$ iteration and corresponding discrete solution in Example \ref{ex1}}
	\label{fig1Ex1_OCP}
\end{figure}
 \begin{table}[ht!]
     	\centering\scalebox{0.8}{
     	\begin{tabular}{|c|c|c|c|c|c|c|c|c|c|c|c|}
      \hline
    Iter&  NDOF & $\enorm{\yb-\yb_\rM}_\pw$ & Order &$\enorm{\pb-\pb_\rM}_\pw$ & Order & $\norm{{\bar{\bu}}-{\bar{\bu}_\h}}_{\mathbb{R}^n}$ & Order & Total error & Order & $\eta$ & Order  \\
\hline\hline
1&$25$ & $0.02529$ & -- & $0.00217$ & -- & $0.11859$ & -- & $0.14606$ & -- & 0.07906 & --\\
5& $123$ & $0.01047$ & $1.23$ & $0.00159$ & $1.09$ & $0.05402$ & $1.74$ & $0.06608$ & $1.65$ &$0.02947$ &$0.98$\\
9& $463$ & $0.00423$ & $0.55$ & $0.00094$ & $0.36$ & $0.01372$ & $0.91$ & $0.01889$ & $0.81$ &$0.01321$ &0.54\\
13&$1723$ & $0.00199$ & $0.51$ & $0.00048$ & $0.45$ & $0.00357$ & $0.87$ & $0.00603$ & $0.72$ & 0.00670& 0.48\\
17& $5751$ & $0.00107$ & $0.44$ & $0.00027$ & $0.31$ & $0.00106$ & $1.00$ & $0.00240$ & $0.69$ &$0.00371$&$0.45$\\
21& $18684$ & $0.00059$ & $0.47$ & $0.00014$ & $0.35$ & $0.00034$ & $0.93$ & $0.00107$ & $0.61$ & $0.00208$ & $0.47$ \\
25&$58677$ & $0.00033$ & $0.47 $ & $0.00008$ & $0.56 $ & $0.00011$ & $1.09 $ & $0.00052$ & $0.62 $& $0.00119$&$0.47$ \\
29& $175719$ & $0.00019$ & $0.51$ & $0.00005$ & $0.44$ & $0.00004$ & $0.74$ & $0.00028$ & $0.53$ & $0.00069$ & $0.50$ \\
33 & $ 520861$ &  $0.00011$ &  $0.54$&   $0.00003$ & $ 0.71$& $ 0.00001$ &  $ 1.22$&   $0.00011$&  $0.62$& $0.00041$& $0.53$\\
        \hline
     	\end{tabular}}
      \caption{Errors and their convergence rates with adaptive  refinement in Example \ref{ex1}}
\label{tab:OCP_Ex1_Err}
     \end{table}

\medskip
\noindent The discrete costs function, $\cJ(\yb_\rM, \bar{\bu}_\h)$, computed using the discrete solution over adaptive meshes, is presented in Table \ref{tab:horizon_table}. The exact cost is $\cJ(\yb, \bar{\bu})= 0.0023756409164699$. The discrete cost is found to approach the exact cost value under adaptive mesh refinements. 
\begin{table}[ht!]
\centering\scalebox{0.8}{
\begin{tabular}{|c||c|c|c|c|c|c|c|c|}
\hline 
\multicolumn{9}{|c|}{\text{ADAPTIVE REFINEMENT}} \\ \hline \hline
Iter &  9 & 13 & 17 & 21 & 25 & 29 & 33 & 37\\ \hline
NDOF & $463$ & $1723$ & $5751$ & $18684$ & $58677$ & $175719$ & $520861$ & $1461945$\\ \hline
$10^{2}\times\cJ(\yb_\rM, \bar{\bu}_\h)$ & $0.2294468$ & $0.2329082$ & $0.2336079$ & $0.2338840$ & $0.2339504$ & $0.2339669$ & $0.2339716$ & $0.2339736$\\ \hline \hline
\multicolumn{9}{|c|}{\text{Exact cost is $\cJ(\yb, \bar{\bu})= 0.0023756409164699$.}} \\ \hline  
\end{tabular}}
\caption{The values of $\cJ(\yb_\rM, \bar{\bu}_\h)$ with respect to adaptive mesh refinements in Example \ref{ex1}}
\label{tab:horizon_table}
\end{table}

\end{example}

\begin{example}{(\bf Square domain with 9 point sources).}\label{ex2}
     This example is dedicated to \eqref{opt1} with multiple point sources. Consider \eqref{opt1} over $\Omega=[0,1]^2$ with nine source points located at $\Z=\{(x_i,x_j): i,j=1,2,3\}$
     with $x_1=1/6$, $x_2=1/2$, $x_3=5/6$, and ${\bf a}$, ${\bf b} \in \mathbb{R}^9$ with all the entries $-1$ and $2$, respectively. The desired state is given by $y_{\rm d}=10^4x^2y^2(1-x)^2(1-y)^2((0.25-x)^2(0.25-y)^2+(0.75-x)^2(0.25-y)^2+(0.25-x)^2(0.75-y)^2+(0.75-x)^2(0.75-y)^2)$ and the regularization parameter is $\alpha=10^{-3}$.

\medskip
\noindent The numerical solutions and errors are computed on the adaptive mesh, and the numerical experiment begins with an initial triangulation with uniform 36 triangles to ensure all the points in $\Z$ are at the nodes. Since the exact solution of the OCP is unknown in this case, errors are computed with respect to a reference solution obtained on the mesh after 25 adaptive refinement steps. Figure~\ref{fig1Ex2_OCP}(a) shows the adaptive mesh refinement at the 20th adaptive iteration.
At the 20th adaptive refinement, the control attains a magnitude of 0.84 at the center, 0.10 at the corners, and 0.28 at the remaining locations. This is also reflected in the adaptive mesh, where finer refinements are observed at the points in $\Z$ corresponding to higher control magnitudes. Figure~\ref{fig1Ex2_OCP}(b) presents the convergence rates of the error in the control variable and the total error, with the convergence behavior of the associated estimators. Optimal convergence rates are observed in Figure \ref{fig1Ex2_OCP}: linear convergence for the control variable and a rate of 0.5 for the error and estimators. Recall that the complete estimator (Estm) is a combination of the state estimator (St Estm) and the adjoint estimator (Adj Estm). It is worth noting that the plots of the state estimator and the complete estimator overlap, indicating that the contribution from the state estimator dominates the adjoint contributions.
\begin{figure}[ht!] 
\centering
\begin{tabular}{cc} 
\includegraphics[width=6cm,height=4.4cm]{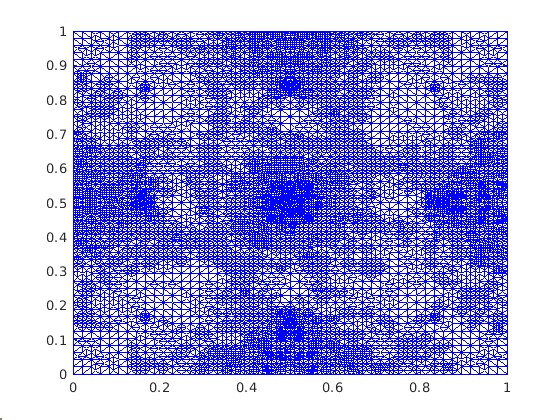} 
& \includegraphics[width=6cm,height=4.4cm]{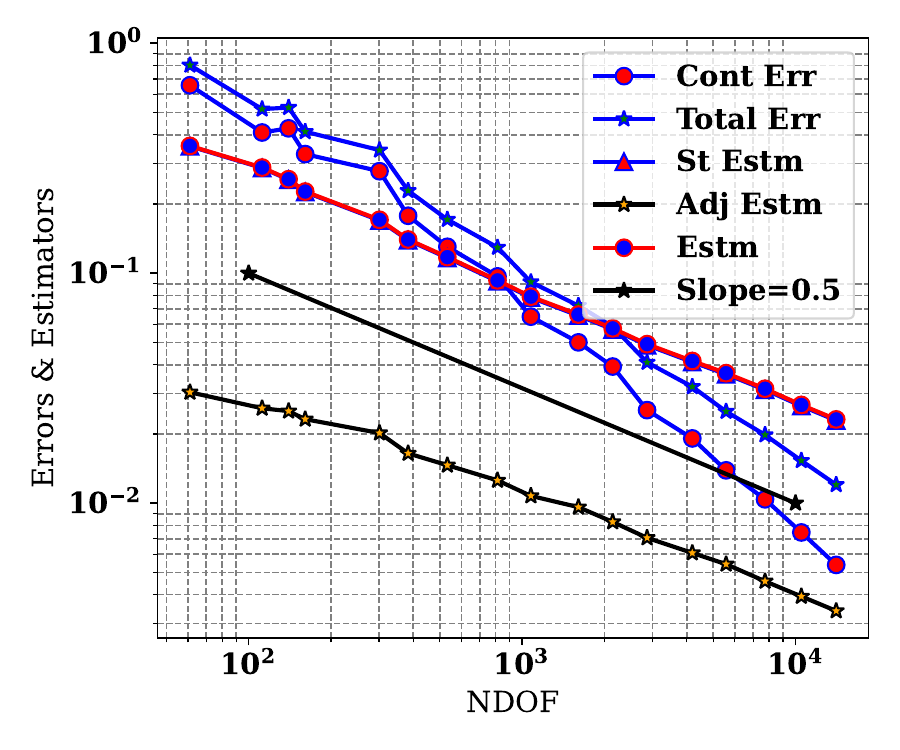}\\
(a) Adaptive Mesh &
(b) Error and Estimator 
	 \end{tabular}   	
  \caption{\small   The adaptive mesh at $20{\rm th}$ iteration and  convergence of the errors and estimators in Example \ref{ex2}}
	\label{fig1Ex2_OCP}
\end{figure}
\end{example}

\begin{example}{(\bf Non-convex domain with single point source).}\label{ex3}
 Consider \eqref{opt1} over $\Gamma$-shaped domain $\Omega=[-1,1]^2\setminus ((0,1]\times [-1,0))$,  ~$\Z=\{(0.5,0.5)\}$, $\alpha=10^{-3}$, and ${\bf a}:=\{-2\},~{\bf b}:=\{2\}$ with the desired state defined as $y_{\rm d}=(1-x^2)^2(1-y^2)^2$.
   \end{example}

 \medskip 
 \noindent The initial mesh is obtained by performing two successive red-refinements on the $\Gamma$-shaped domain, which consists of six isosceles triangles. For a coarse mesh, we do not observe convergence in the first few refinements. To achieve the expected convergence, we start with a fine mesh as predicted in {\it Quasi-orthogonality} {\bf{(A4)}}. Figure \ref{fig1Ex3_OCP}(a) shows the adaptive mesh refinement at the 15th iteration, which demonstrates more refinements at the origin and the location of the source point. In Figure \ref{fig1Ex3_OCP}(b)-(c), the sum of absolute values of the entries in the Hessian of solutions to the state and adjoint equations (that is, for all $T\in \T$, $\norm{(D^2\yb_\rM)|_{T}}_{\ell^1({\mathbb R}^4)}$ and $\norm{(D^2\pb_\rM)|_{T}}_{\ell^1({\mathbb R}^4)}$) are shown. It can be observed that the refinement around the source points is due to a high Hessian of the state solution, which is captured in the state edge estimators, while the higher Hessian of the adjoint solution leads to refinements near $(0,0)$ as we consider a non-convex domain. The errors for the state, adjoint, control variables, and total error, along with their corresponding convergence rates, are presented in Table \ref{tab:OCP_Ex3_Err}. As the exact solution of the OCP is unknown, the errors are calculated with respect to the discrete solution in the 30th adaptive mesh refinement. Optimal convergence rates are observed, which are 0.5 for state and adjoint errors in the energy norm and linear for control error in the Euclidean norm.

\medskip
 \noindent  In Table \ref{tab:OCP_Ex3_Estm}, we present the contributions arising from the state and adjoint estimators within the complete estimator, along with their corresponding convergence rates. In the beginning, the state estimator is dominant, but with more adaptive iterations, the contribution of the adjoint estimator increases. Observe that the complete estimator has an optimal convergence rate of 0.5.
\begin{figure}[ht!] 
\scalebox{0.9}{
\begin{tabular}{ccc} 
\includegraphics[width=5cm,height=4cm]{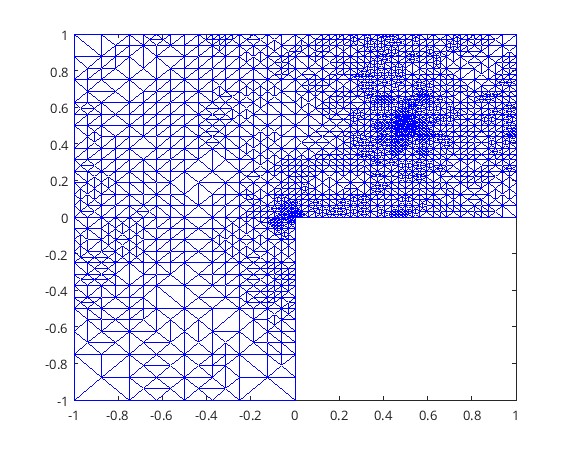} 
&\includegraphics[width=5cm,height=4cm]{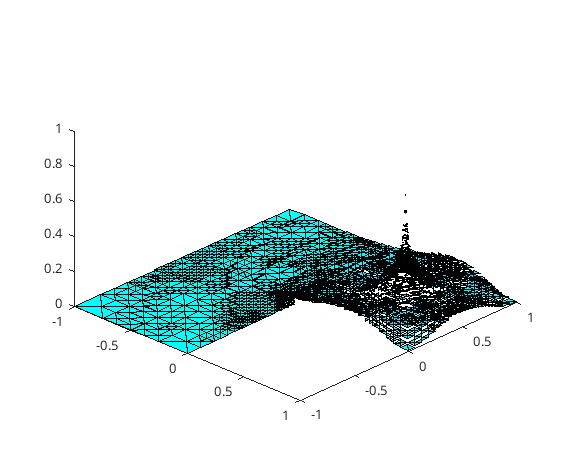} 
&\includegraphics[width=5cm,height=4cm]{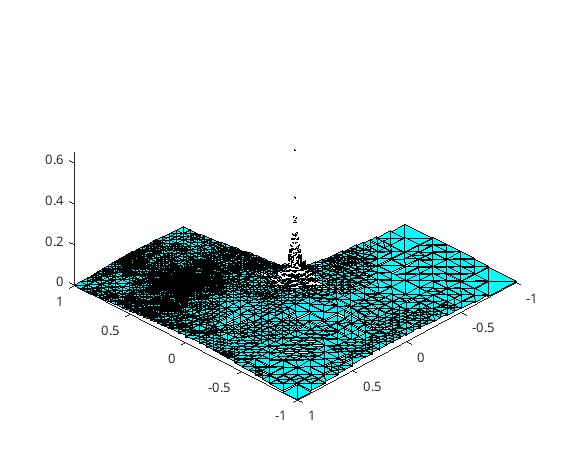}\\
(a) Adaptive Mesh &
(b) State Hessian spikes at the source &
 (c) Adjoint Hessian spikes at $(0,0)$
	 \end{tabular} }  	
  \caption{\small  The adaptive mesh at 15th iteration and corresponding Hessian of the discrete solutions in Example \ref{ex3}.}
	\label{fig1Ex3_OCP}
\end{figure}
 
\begin{table}[ht!]
\centering\footnotesize
\begin{tabular}{|c|c|c|c|c|c|c|c|c|c|}
\hline
 Iter & NDOF & $\enorm{\yb_{30}-\yb_\rM}_\pw$ & Order &$\enorm{\pb_{30}-\pb_\rM}_\pw$ & Order & $\norm{{\bar{\bu}_{30}}-{\bar{\bu}_\h}}_{\mathbb{R}^n}$ & Order & Total error & Order  \\
\hline\hline
1 & 161 & 0.10289 & -- & 0.02431 & -- & 0.55312 & -- & 0.68032 & -- \\
4 & 258 & 0.05249 & 1.35 & 0.02310 & 0.17 & 0.30274 & 1.46 & 0.37833 & 1.37 \\
7 & 632 & 0.02417 & 0.60 & 0.01387 & 0.71 & 0.09027 & 1.06 & 0.12832 & 0.94 \\
10 & 1654 & 0.01367 & 0.58 & 0.00967 & 0.39 & 0.03946 & 0.86 & 0.06280 & 0.73 \\
13 & 4471 & 0.00768 & 0.66 & 0.00592 & 0.33 & 0.01294 & 1.15 & 0.02654 & 0.84 \\
16 & 11584 & 0.00479 & 0.52 & 0.00353 & 0.52 & 0.00508 & 1.13 & 0.01339 & 0.76 \\
19 & 27914 & 0.00303 & 0.51 & 0.00224 & 0.54 & 0.00210 & 1.19 & 0.00737 & 0.73 \\
23 & 88461 & 0.00159 & 0.54 & 0.00119 & 0.63 & 0.00059 & 1.03 & 0.00338 & 0.66 \\
\hline 
\end{tabular}
\caption{Errors in state, adjoint, and control solutions and their convergence rates in Example \ref{ex3}}
\label{tab:OCP_Ex3_Err}
\end{table}
 \vspace{-0.5cm}
\begin{table}[ht!]
\centering\footnotesize
\begin{tabular}{|c|c|c|c|c|c|c|c|c|}
\hline
Iter&NDOF & $\eta_S$ & Order & $\eta_A$ & Order & $\eta$ & Order \\
\hline \hline
1&$161$ & $0.32566$ & -- & $0.08613$ & -- & $0.33686$ & -- \\
4&$258$ & $0.14885$ & $1.50$ & $0.07962$ & $0.26$ & $0.16881$ & $1.27$ \\
7&$632$ & $0.07492$ & $0.53$ & $0.04496$ & $0.82$ & $0.08737$ & $0.61$ \\
10&$1654$ & $0.04353$ & $0.52$ & $0.03063$ & $0.38$ & $0.05322$ & $0.47$ \\
13&$4471$ & $0.02548$ & $0.61$ & $0.01899$ & $0.32$ & $0.03178$ & $0.51$ \\
16&$11584$ & $0.01627$ & $0.49$ & $0.01177$ & $0.49$ & $0.02008$ & $0.49$ \\
19&$27914$ & $0.01057$ & $0.48$ & $0.00773$ & $0.50$ & $0.01309$ & $0.49$ \\
23&$88461$ & $0.00595$ & $0.46$ & $0.00441$ & $0.54$ & $0.00741$ & $0.49$ \\    
\hline
\end{tabular}
\caption{Estimators and their convergence rates in Example \ref{ex3}}
\label{tab:OCP_Ex3_Estm}
\end{table}

 \medskip
\subsection{Point-wise tracking \eqref{opt1n}}\label{pwOCP1}
In this section, the numerical experiments for the point-wise tracking OCP in \eqref{opt1n} are considered. Here, the state solution needs to approach the desired state at certain locations while control is applied to the whole domain. In the previous subsection, the force control was applied at only a few points within the domain, while the desired state was imposed over the whole domain.

\medskip
\noindent Recall that the set $\D$ denotes the set of points where we expect the state solution to approach the desired state, ${\bf y}_{\rm d}$. Here, two numerical experiments are conducted on convex and non-convex domains. First, we consider a circular domain where the exact solution of the OCP is known. For this problem, the errors, estimators, and their corresponding convergence rates in both uniform and adaptive mesh refinements are presented. In the second experiment, we examine the effect of the various locations of the point $\D$ on the adaptive mesh and the discrete solution in a $\Gamma$-shaped domain. 
\begin{example}{(\bf Convex domain)}\label{ex4}
Let $\Omega\subset {\mathbb{R}}^2$ be the circular domain as in Example \ref{ex1}. Consider \eqref{opt1n} with ${\bf y}_{\rm d}=\{10^{-3}(e^{-2}-1)\}$, $\alpha=10^{-3}$, $\D=\{(0,0)\}$, and the control constraints ${u_a}=-1,~{u_b}=1$. The exact solutions to the state and adjoint equations are $\yb=10^{-3}(e^{-1}-r^2)^2$ and $\pb=10^{-3}(\frac{r^2}{8\pi}{\rm log}(r)+\frac{1}{16\pi}e^{-1})$. The exact control is $\bar{u}$ can be obtained by the formula, $\Pi_{[u_a,u_b]}(-\alpha^{-1}\pb)$ and the force, $f=0.064-\ub$.

\medskip
\noindent Since the exact solutions are with elliptic regularity index $\sigma=1$, Theorem~\ref{OCPap} establishes an optimal convergence rate of 1 for the control variable and 0.5 for the total error. This behavior is confirmed in Table \ref{tab:Ex1_table}. Additionally, the results validate the a~posteriori error estimates in Theorem~\ref{OCPap}$(b)$, demonstrating that the total error and the complete estimator achieve the optimal order of convergence. 
The values of discrete cost functions at some iterations are presented in Table \ref{tab2:horizon_table} for both uniform and adaptive mesh refinements, which are close to the exact cost value.
\end{example}
\begin{table}[ht!]
\centering
\footnotesize
\begin{tabular}{|c|c|c|c|c|c|c|c|c|c|c|}
\hline
 \multicolumn{5}{|c|}{Uniform refinement} & \multicolumn{6}{c|}{Adaptive refinement} \\ \hline
 NDOF & Total error & Order & $\eta$ & Order & Iter & NDOF & Total error & Order & $\eta$ & Order  \\ \hline \hline
 $5$ & $0.01380$ & -- & $0.00752$ & -- & $5$ & $161$ & $0.00171$ & $0.29$ & $0.00301$ & $0.32$ \\
$25$ & $0.00518$ & $0.61$ & $0.00594$ & $0.15$ & $8$ & $415$ & $0.00101$ & $0.48$ & $0.00211$ & $0.36$ \\
$113$ & $0.00198$ & $0.64$ & $0.00359$ & $0.33$ & $11$ & $1049$ & $0.00059$ & $0.61$ & $0.00145$ & $0.59$ \\
$481$ & $0.00082$ & $0.61$ & $0.00194$ & $0.42$ & $14$ & $2907$ & $0.00037$ & $0.35$ & $0.00094$ & $0.38$ \\
$1985$ & $0.00037$ & $0.57$ & $0.00101$ & $0.46$ & $17$ & $7315$ & $0.00021$ & $0.52$ & $0.00060$ & $0.49$ \\
$8065$ & $0.00017$ & $0.54$ & $0.00051$ & $0.48$ & $20$ & $18256$ & $0.00013$ & $0.69$ & $0.00039$ & $0.52$ \\
$32513$ & $0.00008$ & $0.52$ & $0.00026$ & $0.49$ & $23$ & $44492$ & $0.00008$ & $0.42$ & $0.00026$ & $0.42$ \\
$130561$ & $0.00004$ & $0.51$ & $0.00013$ & $0.50$ & $26$ & $99288$ & $0.00005$ & $0.53$ & $0.00017$ & $0.53$ \\
$523265$ & $0.00002$ & $0.50$ & $0.00006$ & $0.50$ & $29$ & $234354$ & $0.00003$ & $0.49$ & $0.00011$ & $0.50$ \\
 \hline
\end{tabular}
\caption{Errors, estimators, and their orders of convergence in Example \ref{ex4}}
\label{tab:Ex1_table}
\end{table}

\begin{table}[ht!]
\centering\scalebox{0.8}{
\begin{tabular}{|c||c|c|c|c|c|c|c|c|c|c|}
\hline
\multicolumn{11}{|c|}{\text{UNIFORM REFINEMENT}} \\ \hline 
NDOF & 5 & 25 & 113 & 481 & 1985 & 8065 & 32513 & 130561 & 523265 & 2095105\\ \hline
$10^{7}\times\cJ(\yb_\rM, \ub_\h)$ & 0.73537 & 0.19754 & 0.09163 & 0.06856 & 0.06309 & 0.06174 & 0.06140 & 0.6132 & 0.06130 & 0.06129\\ \hline \hline 
\multicolumn{11}{|c|}{\text{ADAPTIVE REFINEMENT}} \\ \hline 
Iter & 1 & 5 & 9 & 13 & 17 & 21 & 25 & 29 &33 &37 \\
\hline
NDOF & $25$ & $161$ & $611$ & $1967$ & $7315$ & $24202$ & $76820$ & $234354$ & $663867$ & $1867454$\\ \hline
$10^{7}\times\cJ(\yb_\rM, \ub_\h)$  & 0.19754 & 0.08903 & 0.6798 & $0.06350$ & $0.06198$ & $0.06152$ & $0.06137$ & $0.06132$ & $0.06130$ & $0.06129$\\ \hline
\multicolumn{11}{|c|}{\text{Exact cost is $6.129158721563519\times 10^{-9}$.}} \\
\hline
\end{tabular}}
\caption{The values of $\cJ(\yb_\rM, \ub_\h)$ for uniform and adaptive mesh refinements in Example \ref{ex4}}
\label{tab2:horizon_table}
\end{table}  
  \begin{example}{(\bf Non-convex domain).}\label{ex5}
   Consider the OCP in \eqref{opt1n} with $\Omega=[-1,1]^2\setminus ((0,1]\times [-1,0))$, $\yd=\{1\}$, $\alpha=10^{-3}$, $u_a=0,~u_b=5$, and $f=0$. The numerical experiments are conducted separately for three different locations of $\D$. In the first case, the state solution is required to drive towards the desired state at the point \(\D1=\{(-0.5, -0.5)\}\), in the second case at \(\D2=\{(-0.5, 0.5)\}\), and in the third case at \(\D3=\{(0.5, 0.5)\}\).     
  \end{example}
 
\begin{figure}[ht!] 
\begin{tabular}{ccc}
\includegraphics[width=5cm,height=4cm]{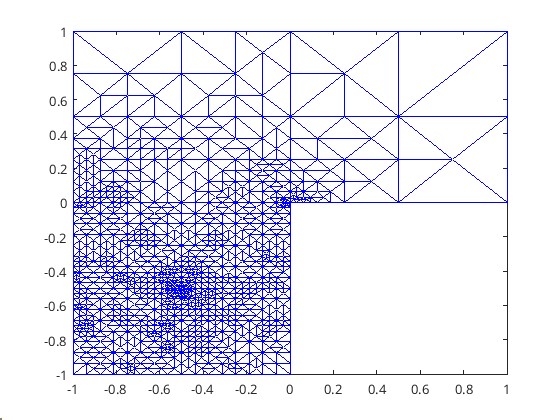} 
& \includegraphics[width=5cm,height=4cm]{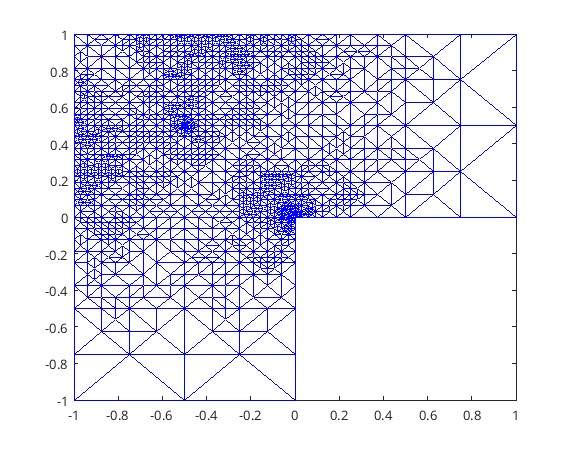}
& \includegraphics[width=5cm,height=4cm]{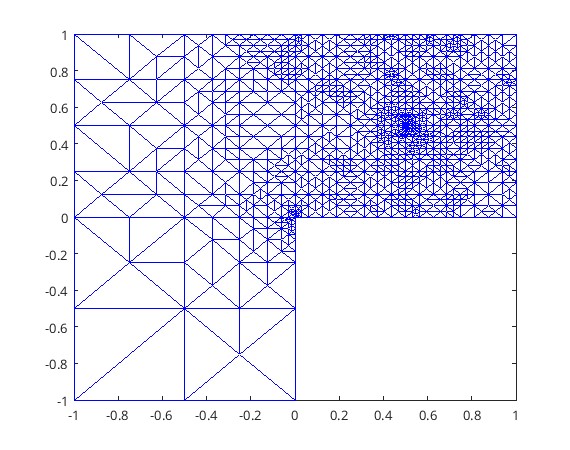}\\
(a) Adaptive mesh $(\D1)$& (b) Adaptive mesh $(\D2)$& (c) Adaptive mesh $(\D3)$\\
\includegraphics[width=5cm,height=4cm]{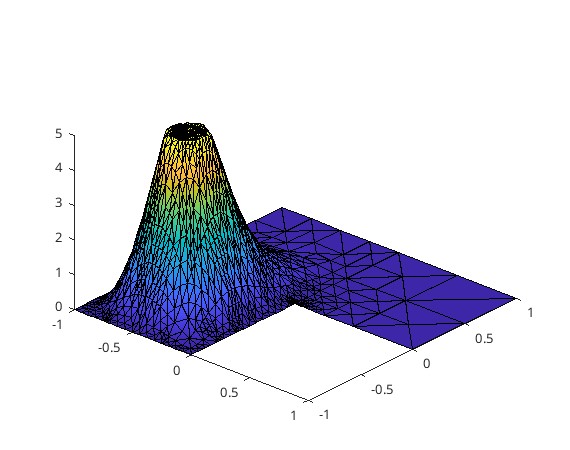} 
& \includegraphics[width=5cm,height=4cm]{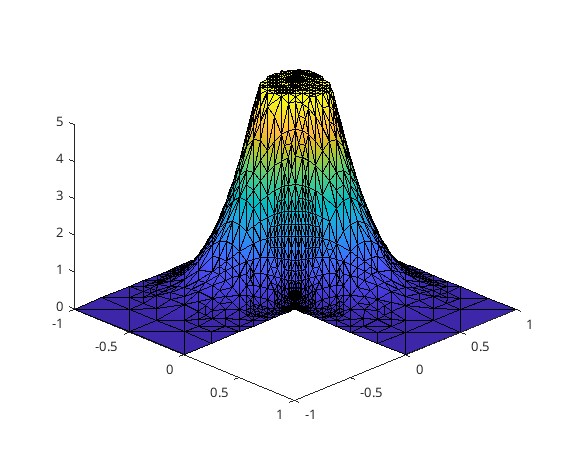}
& \includegraphics[width=5cm,height=4cm]{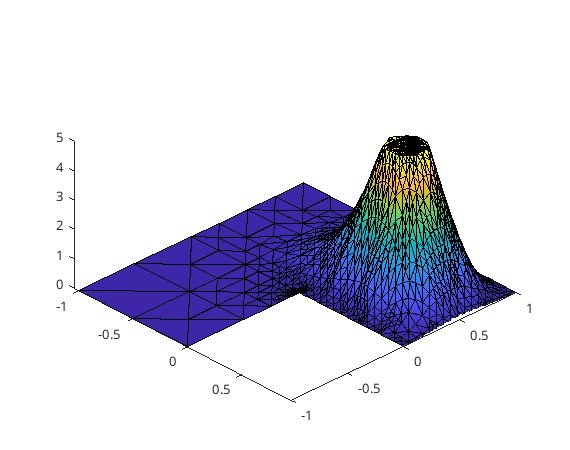}\\
(d) Control solution $(\D1)$ &
(e) Control solution $(\D2)$ &
(f) Control solution $(\D3)$\\
\includegraphics[width=5cm,height=4cm]{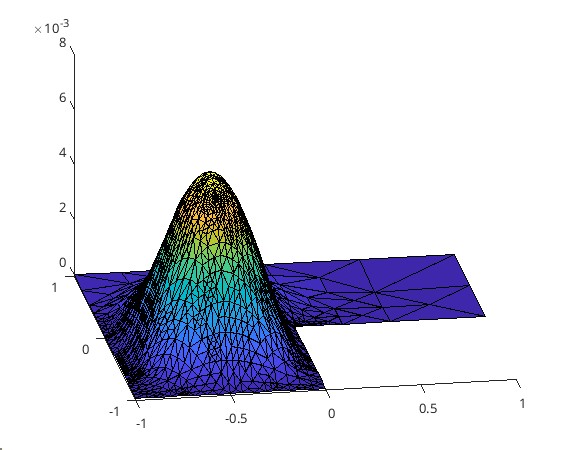} 
& \includegraphics[width=5cm,height=4cm]{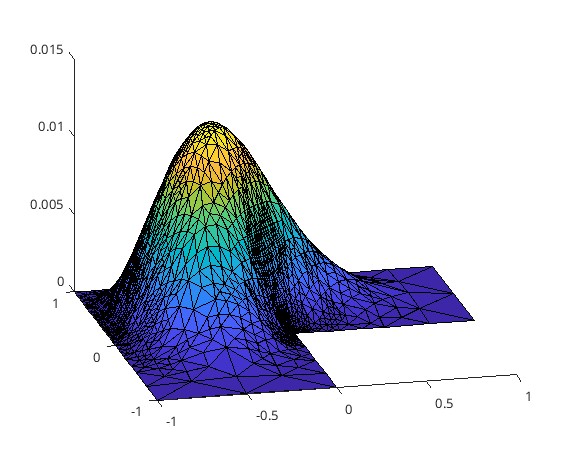}
& \includegraphics[width=5cm,height=4cm]{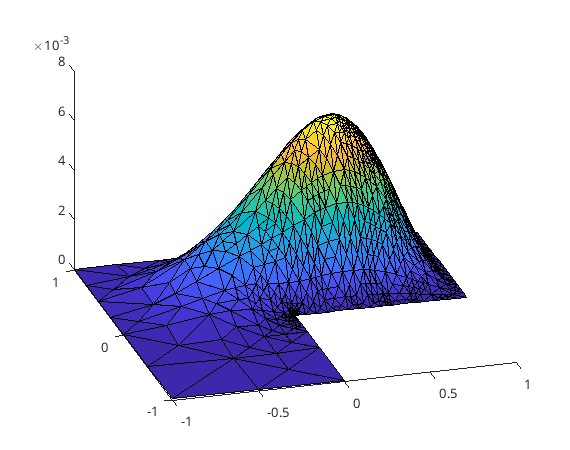}\\
(g) State solution $(\D1)$ &
(h) State solution $(\D2)$ &
(i) State solution $(\D3)$
	 \end{tabular}   	
  \caption{\small The adaptive mesh (15th iteration) and the respective discrete solutions $\ub_\h$ and $\yb_\rM$ in Example \ref{ex5}}
\label{fig1Ex2_OCPnew}
\end{figure}
\noindent In this example, the discrete control and state solutions, and adaptive mesh will be compared for different locations of points in $\D$. Figure \ref{fig1Ex2_OCPnew} illustrates the adaptive mesh refinements at the 15th iteration along with the discrete control and state solutions, \(\ub_\h\) and \(\yb_\rM\), respectively, for \(\D1\)-\(\D3\). The figures highlight the impact of varying the locations of points in \(\D\) on the control and state solutions. Specifically, increased refinements are observed at different locations of \(\D\) and at the corner \((0,0)\) due to the consideration of a nonconvex domain. Furthermore, it is evident that the state solutions approach the desired state at the location of \(\D\). This confirms the effectiveness of the a~posteriori error estimator in accurately identifying critical regions. 

\begin{figure}[ht!] 
\begin{tabular}{cc} 
\includegraphics[width=7cm,height=5.4cm]{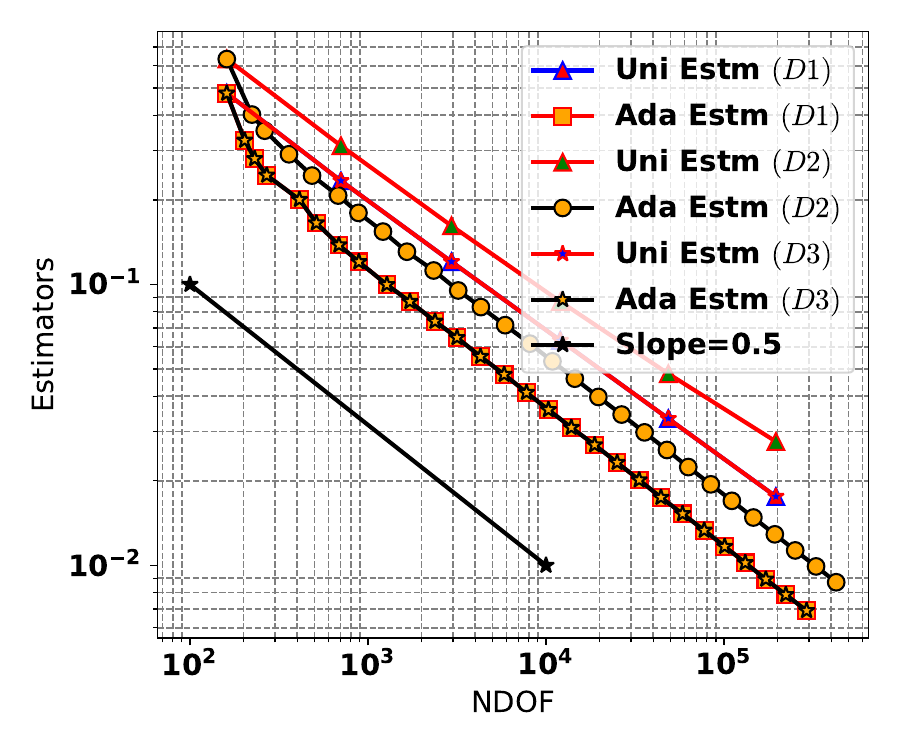} & \includegraphics[width=7cm,height=5.4cm]{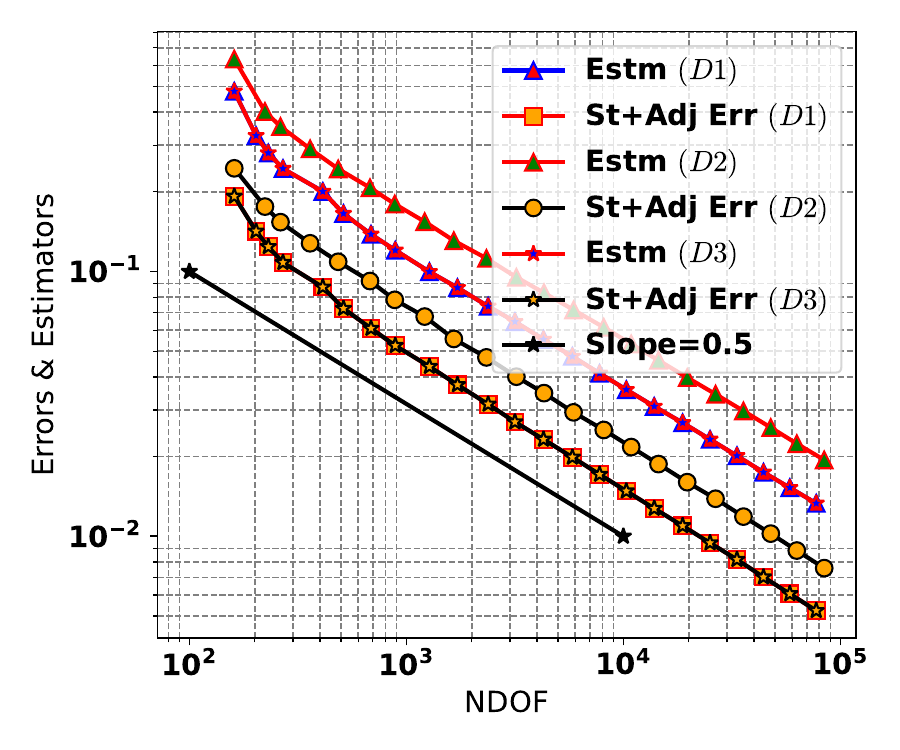}
	 \end{tabular}   	
  \caption{\small The rate of convergence of the errors and the estimators for all locations of $\D$ in Example \ref{ex5}}
	\label{fig1Ex2_OCP2}
\end{figure}

\medskip
\noindent In Figure \ref{fig1Ex2_OCP2}, we have observed the behavior of errors and estimators with $\D1$, $\D2$, and $\D3$ in the adaptive mesh refinements. As the analytical solution is unknown, the state and adjoint errors are calculated with respect to the discrete solutions at the 33rd adaptive iteration (\(\enorm{\yb_{33}-\yb_\rM}_\pw+\enorm{\pb_{33}-\pb_\rM}_\pw\)). 
The first figure in Figure \ref{fig1Ex2_OCP2} demonstrates the convergence rates of complete estimators in both uniform and adaptive mesh refinements. The convergence rate of 0.45 is observed in the case of uniform refinements for all the choices of $\D$, while in adaptive mesh refinement, the optimal convergence rate, 0.5, is achieved. In the second figure of Figure \ref{fig1Ex2_OCP2}, the optimal convergence rate is observed for the sum of state and adjoint errors and the complete estimator for adaptive refinement. Furthermore, we observed that the errors and estimator values for cases $\D1$ and $\D3$ are similar. That is due to the diagonally symmetric positions of the points in $\D1$ and $\D3$ within the domain, resulting in the plots that are completely overlapping.
\subsection{Point-wise tracking OCP with point sources \eqref{opt2n}}\label{pwOCP2}
 Recall that the set $\D$ is the set of points associated with the desired state, ${\bf y}_{\rm d}$, as discussed in Section \ref{pwOCP1}, and $\Z$ is the set of source points (see Section \ref{OCP_new2}). In this section, we analyze the cost function in three examples by: increasing the distance between $\D$ and $\Z$, increasing the cardinality of $\Z$ while keeping $\D$ fixed, and increasing the cardinality of $\D$ while keeping $\Z$ fixed.
Throughout this section, we consider \eqref{opt2n} over the domain $\Omega=[0,1]^2$, ${\bf y}_{\rm d}=\{1\}^m$, $\alpha=10^{-3}$, and ${\bf a}:=\{0\}^n$ and ${\bf b}:=\{7\}^n$, where \(m=|\D|\) and \(n=|\Z|\).

\begin{example}{(\bf{Different locations of $\D$ and $\Z$})}\label{ex6}
     In this example, we have considered 6 different pairs of locations for $\D$ and $\Z$: ($\D1=\{(0.5,0.5)\}$,~$\Z1=\{(0.5,0.5)\}$), ($\D2=\{(0.25,0.25)\}$,~$\Z2=\{(0.25,0.25)\}$), ($\D3=\{(0.25,0.25)\}$,~$\Z3=\{(0.5,0.5)\}$), ($\D4=\{(0.5,0.5)\}$,~$\Z4=\{(0.25,0.25)\}$), ($\D5=\{(0.25,0.75)\}$,~$\Z5=\{(0.25,0.25)\}$), and ($\D6=\{(0.75,0.75)\}$,~$\Z6=\{(0.25,0.25)\}$) and computed the numerical solutions in each cases. The impact of the distance between the points in \(\D\) and \(\Z\) on the cost function is analyzed.
\end{example}
\begin{figure}[ht!] 
\begin{tabular}{ccc}
\includegraphics[width=4.5cm,height=3.5cm]{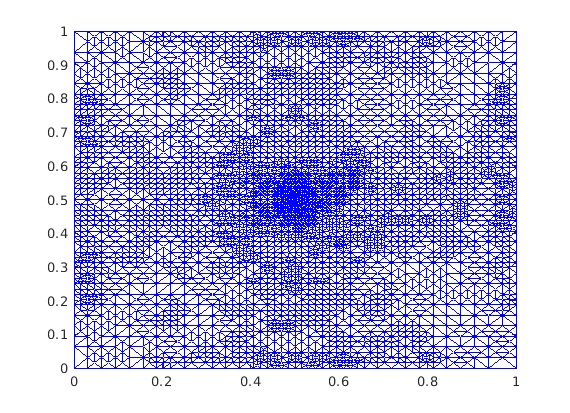} 
& \includegraphics[width=4.5cm,height=3.5cm]{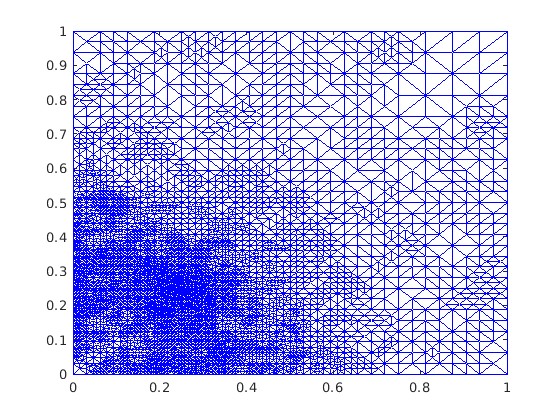}
& \includegraphics[width=4.5cm,height=3.5cm]{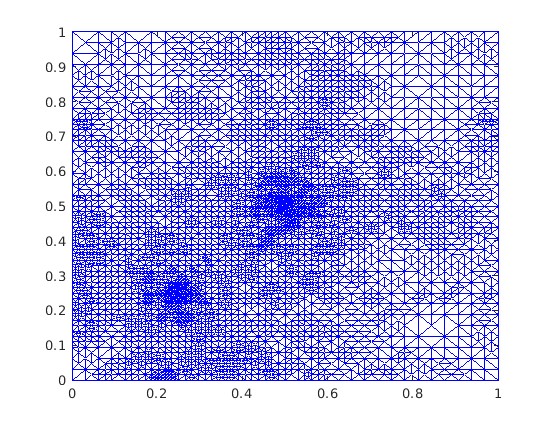}\\
 \footnotesize{(a) $\D1=\{0.5,0.5\}$,~$\Z1=\{0.5,0.5\}$} &
 \footnotesize{(b) $\D2=\{0.25,0.25\}$,~$\Z2=\{0.25,0.25\}$}  &
 \footnotesize{(c) $\D3=\{0.25,0.25\}$,~$\Z3=\{0.5,0.5\}$} \\
\includegraphics[width=4.5cm,height=3.5cm]{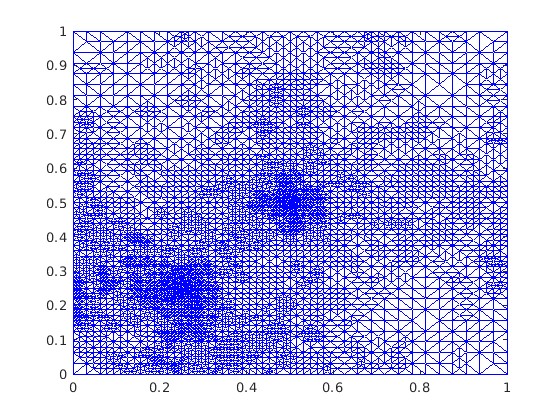} 
& \includegraphics[width=4.5cm,height=3.5cm]{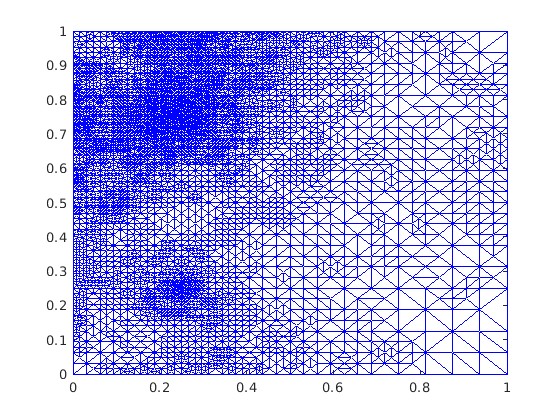}
& \includegraphics[width=4.5cm,height=3.5cm]{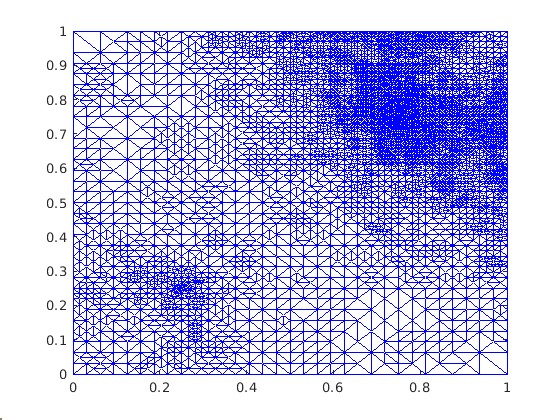}\\
 \footnotesize{(d) $\D4=\{0.5,0.5\}$,~$\Z4=\{0.25,0.25\}$} &
 \footnotesize{(e) $\D5=\{0.25,0.75\}$,~$\Z5=\{0.25,0.25\}$}  &
 \footnotesize{(f) $\D6=\{0.75,0.75\}$,~$\Z6=\{0.25,0.25\}$} 
	 \end{tabular}   	
  \caption{\small  The adaptive meshes at $15{\rm th}$ iteration for each pair in Example \ref{ex6}}
\label{Ex1_OCPnew}
\end{figure}

\begin{table}[ht!]
\centering
\footnotesize
\begin{tabular}{|c|c|c|c|c|c|c|c|c|c|}
\hline
 \multicolumn{1}{|c|}{}&\multicolumn{3}{c|}{\scalebox{0.9}{$\D1=\{(0.5,0.5)\}$,~$\Z1=\{(0.5,0.5)\}$}} & \multicolumn{3}{c|}{\scalebox{0.9}{$\D2=\{(0.25,0.25)\}$,~$\Z2=\{(0.25,0.25)\}$}} & \multicolumn{3}{c|}{\scalebox{0.9}{$\D3=\{(0.25,0.25)\}$,~$\Z3=\{(0.5,0.5)\}$}} \\ 
\hline 
Iter& NDOF &  $\cJ(\yb_\rM,\bar{\bu}_\h)$ & $\bar{\bu}_\h$ & NDOF & $\cJ(\yb_\rM,\bar{\bu}_\h)$ & $\bar{\bu}_\h$ & NDOF & $\cJ(\yb_\rM,\bar{\bu}_\h)$ & $\bar{\bu}_\h$ \\
\hline
1  & 25    & 0.4673400 & 7.0000000 & 25    & 0.4909403 & 7.0000000 & 25    & 0.4913056 & 4.1335805 \\
5  & 133   & 0.4972165 & 7.0000000 & 147   & 0.4950323 & 3.1363585 & 103   & 0.4979709 & 2.0104099 \\
9  & 502   & 0.4809692 & 6.0508541 & 655   & 0.4968692 & 2.4944567 & 447   & 0.4990126 & 1.4039051 \\
13 & 1906  & 0.4837230 & 5.6119778 & 2480  & 0.4972641 & 2.3327974 & 1758  & 0.4991923 & 1.2699722 \\
17 & 6633  & 0.4844498 & 5.4893668 & 8862  & 0.4973663 & 2.2890292 & 6703  & 0.4992346 & 1.2362865 \\
21 & 22254 & 0.4846488 & 5.4552490 & 28794 & 0.4973944 & 2.2768344 & 22825 & 0.4992454 & 1.2275927 \\
25 & 70018 & 0.4847064 & 5.4453291 & 89372 & 0.4974027 & 2.2732179 & 72052 & 0.4992484 & 1.2251527 \\
\hline
\end{tabular}
\caption{ The cost function and the discrete control for (a)-(c) in Figure \ref{Ex1_OCPnew} in Example \ref{ex6}}
\label{tab:Ex1_OCPnew}
\end{table}
 
\begin{table}[ht!]
\centering
\footnotesize
\begin{tabular}{|c|c|c|c|c|c|c|c|c|c|}
\hline
\multicolumn{1}{|c|}{}& \multicolumn{3}{c|}{\scalebox{0.9}{$\D4=\{(0.5,0.5)\}$,~$\Z4=\{(0.25,0.25)\}$}} & \multicolumn{3}{c|}{ \scalebox{0.9}{$\D5=\{(0.25,0.75)\}$,~$\Z5=\{(0.25,0.25)\}$}} & \multicolumn{3}{c|}{\scalebox{0.9}{$\D6=\{(0.75,0.75)\}$,~$\Z6=\{(0.25,0.25)\}$}} \\ 
\hline 
Iter & NDOF &  $\cJ(\yb_\rM,\bar{\bu}_\h)$ & $\bar{\bu}_\h$ & NDOF & $\cJ(\yb_\rM,\bar{\bu}_\h)$ & $\bar{\bu}_\h$ & NDOF & $\cJ(\yb_\rM,\bar{\bu}_\h)$ & $\bar{\bu}_\h$ \\
\hline
 1  & 25 & 0.4913056 & 4.1335805 & 25 & 0.4997017 & 0.7722038 & 25 & 0.4997417 & 0.7185069 \\
5  & 164 & 0.4984822 & 1.7396733 & 158 & 0.4999127 & 0.4177588 & 139 & 0.4999615 & 0.2775687 \\
9  & 590 & 0.4990660 & 1.3654542 & 664 & 0.4999466 & 0.3267888 & 611 & 0.4999789 & 0.2056170 \\
13 & 2189 & 0.4992070 & 1.2583849 & 2544 & 0.4999538 & 0.3040960 & 2428 & 0.4999854 & 0.1709136 \\
17 & 76970 & 0.4992371 & 1.2343197 & 9130 & 0.4999559 & 0.2968740 & 8607 & 0.4999866 & 0.1639613 \\
21 & 25566 & 0.4992460 & 1.2270944 & 30705 & 0.4999565 & 0.2951099 & 28940 & 0.4999868 & 0.1622927 \\
25 & 79916 & 0.4992485 & 1.2250408 & 97142 & 0.4999566 & 0.2945760 & 90680 & 0.4999869 & 0.1617794 \\
\hline
\end{tabular}
\caption{The cost function and the discrete control for (d)-(f) in Figure \ref{Ex1_OCPnew} in Example \ref{ex6}}
\label{tab:Ex1_OCPnew1}
\end{table}
\noindent  In Figure \ref{Ex1_OCPnew} and Table \ref{tab:Ex1_OCPnew}-\ref{tab:Ex1_OCPnew1}, we have presented the adaptive mesh refinements along with the values of the cost functions and the discrete control for the above pairs of locations of $\D$ and $\Z$. In Figure \ref{Ex1_OCPnew}, more adaptive refinements are observed at the locations of $\D$ and $\Z$ in all the examples. In the first case, the point source in \(\Z\) is positioned in the center, and the state solution \(\yb_\rM\) is required to drive toward the desired state \(\mathbf{y}_{\rm d}\), also at the center (\(\D\)), as shown in Figure \ref{Ex1_OCPnew}(a). However, in the second case, both \(\D\) and \(\Z\) are placed at a corner (0.25,0.25), as illustrated in Figure \ref{Ex1_OCPnew}(b). In particular, the first case results in a lower cost compared to the second (see Table \ref{tab:Ex1_OCPnew}).  When we move both the points locations in \(\D\) and \(\Z\) from the center to another location as shown in Figure \ref{Ex1_OCPnew}(a)-(b), the value of the cost function increases (ref. Table \ref{tab:Ex1_OCPnew}). As the distance between the points in \(\D\) and \(\Z\) increases, as shown in Figure \ref{Ex1_OCPnew}(b) and Figures \ref{Ex1_OCPnew}(d)-(f), the values of the cost functions gradually increase, which is reflected in Table \ref{tab:Ex1_OCPnew}-\ref{tab:Ex1_OCPnew1}. 
 
 \begin{example}{(\bf Increasing the cardinality of $\Z$ for fixed $\D$).}\label{ex7}
     In this example, the point in $\D$ is fixed, while the number of point sources in $\Z$ gradually increased as illustrated below in Figure \ref{fig:squares1}, to observe its effects on the cost functions. In Figure \ref{fig:squares1}, the location of point sources indicated by blue dots and red dots shows the coordinates in $\D$. Fix \(\D=\{(0.5,0.5)\}\) and gradually add point sources at the corners \((0.25,0.25)\), \((0.75,0.25)\), \((0.25,0.75)\), and \((0.75,0.75)\).
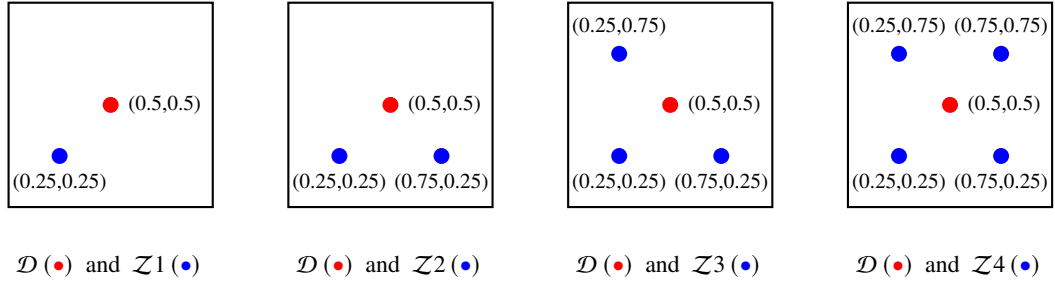
\begin{figure}[h!]
    \centering
    \begin{tikzpicture}
        \def\squaresize{2.7} 
        \def\gap{1} 
        \def\dotradius{3pt} 

        \def\descone{ $\D \,(\textcolor{red}{\bullet})$ \text{ and } $\Z1 \,(\textcolor{blue}{\bullet})$}
        \def\desctwo{ $\D \,(\textcolor{red}{\bullet})$ \text{ and } $\Z2 \,(\textcolor{blue}{\bullet})$}
        \def\descthree{ $\D \,(\textcolor{red}{\bullet})$ \text{ and } $\Z3 \,(\textcolor{blue}{\bullet})$}
        \def\descfour{ $\D \,(\textcolor{red}{\bullet})$ \text{ and } $\Z4 \,(\textcolor{blue}{\bullet})$}

        \foreach \i in {0,1,2,3} {
            \pgfmathsetmacro\x{\i * (\squaresize + \gap)}

            \pgfmathsetmacro\cx{\x + 0.5*\squaresize}
            \pgfmathsetmacro\cy{0.5*\squaresize}

            \pgfmathsetmacro\blx{\x}              
            \pgfmathsetmacro\bly{0}               
            \pgfmathsetmacro\brx{\x + \squaresize} 
            \pgfmathsetmacro\bry{0}               
            \pgfmathsetmacro\tlx{\x}              
            \pgfmathsetmacro\tly{\squaresize}     
            \pgfmathsetmacro\trx{\x + \squaresize} 
            \pgfmathsetmacro\try{\squaresize}     

            \pgfmathsetmacro\bxL{(\blx + \cx) / 2}
            \pgfmathsetmacro\byL{(\bly + \cy) / 2}

            \pgfmathsetmacro\bxR{(\cx + \brx) / 2}
            \pgfmathsetmacro\byR{(\cy + \bry) / 2}

            \pgfmathsetmacro\bxT{(\cx + \tlx) / 2}
            \pgfmathsetmacro\byT{(\cy + \tly) / 2}

            \pgfmathsetmacro\bxTR{(\cx + \trx) / 2}
            \pgfmathsetmacro\byTR{(\cy + \try) / 2}

            \draw[thick] (\x cm,0cm) --++ (\squaresize cm,0cm) --++ (0cm,\squaresize cm) --++
                         (-\squaresize cm,0cm) -- cycle;

            \fill[red] (\cx cm, \cy cm) circle (\dotradius);

            \fill[blue] (\bxL cm, \byL cm) circle (\dotradius);
            \node[below] at (\bxL cm, \byL cm - 0.1cm) {\scriptsize (0.25,0.25)};

            \ifnum\i>0
                \fill[blue] (\bxR cm, \byR cm) circle (\dotradius);
                \node[below] at (\bxR cm, \byR cm - 0.1cm) {\scriptsize (0.75,0.25)};
            \fi

            \ifnum\i>1
                \fill[blue] (\bxT cm, \byT cm) circle (\dotradius);
                \node[above] at (\bxT cm, \byT cm + 0.1cm) {\scriptsize (0.25,0.75)};
            \fi
 
            \ifnum\i=3
                \fill[blue] (\bxTR cm, \byTR cm) circle (\dotradius);
                \node[above] at (\bxTR cm, \byTR cm + 0.1cm) {\scriptsize (0.75,0.75)};
            \fi

            \node[right] at (\cx cm + 0.1cm, \cy cm) {\scriptsize (0.5,0.5)};

            \ifnum\i=0
                \node[below] at (\cx cm, -0.5cm) {\descone};
            \fi
            \ifnum\i=1
                \node[below] at (\cx cm, -0.5cm) {\desctwo};
            \fi
            \ifnum\i=2
                \node[below] at (\cx cm, -0.5cm) {\descthree};
            \fi
            \ifnum\i=3
                \node[below] at (\cx cm, -0.5cm) {\descfour};
            \fi
        }

    \end{tikzpicture}
    \caption{Illustration of locations of $\D$ and $\Z$ with red and blue dots, respectively in Example \ref{ex7}}
    \label{fig:squares1}
\end{figure}
\end{example}
\begin{figure}[ht!] 
\begin{tabular}{cccc}
\includegraphics[width=3.5cm,height=3.5cm]{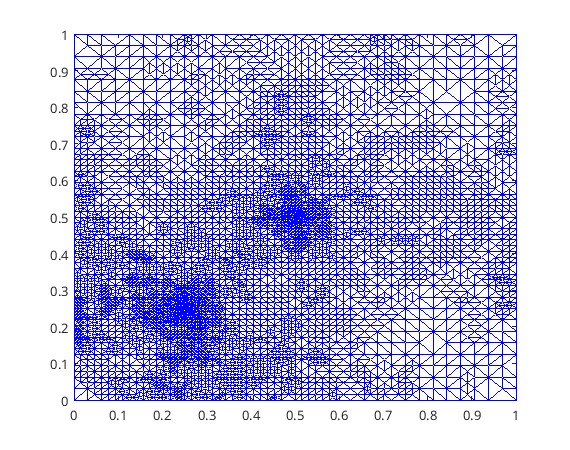} 
& \includegraphics[width=3.5cm,height=3.5cm]{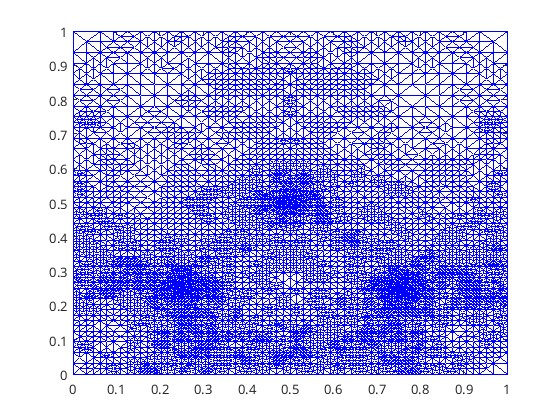}
& \includegraphics[width=3.5cm,height=3.5cm]{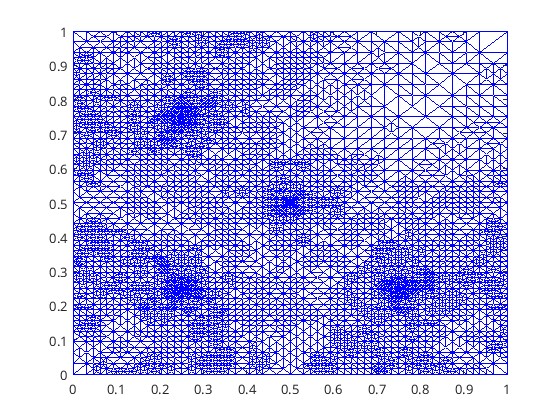}
& \includegraphics[width=3.5cm,height=3.5cm]{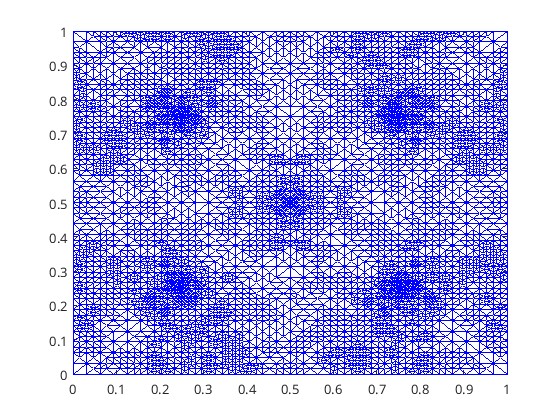}\\
 \footnotesize{(a) Adaptive mesh $(\D,\Z1)$} &
 \footnotesize{(b) Adaptive mesh $(\D,\Z2)$} &
 \footnotesize{(c) Adaptive mesh $(\D,\Z3)$} &
 \footnotesize{(d) Adaptive mesh $(\D,\Z4)$}
	 \end{tabular}   	
  \caption{\small  The adaptive meshes at the $15{\rm th}$ iteration in Example \ref{ex7}}
\label{Ex2_OCP3}
\end{figure}

\medskip
\noindent In Figure \ref{Ex2_OCP3}, the triangulations at the 15th iteration are presented for each case,  which shows a larger number of refinements near the points in $\D$ and $\Z$. In all cases, we observed that the state estimator $\eta_S$ gradually increases as we increase the cardinality of $\Z$ and dominates the adjoint estimator.
\begin{table}[ht!]
\centering
\footnotesize
\begin{tabular}{|c|c|c|c|c|c|c|c|c|}
\hline
\multicolumn{1}{|c|}{}& \multicolumn{2}{c|}{$\Z1$} & \multicolumn{2}{c|}{$\Z2$} & \multicolumn{2}{c|}{$\Z3$}
 & \multicolumn{2}{c|}{$\Z4$}\\ 
\hline 
Iter & NDOF &  $\cJ(\yb_\rM,\bar{\bu}_\h)$  & NDOF & $\cJ(\yb_\rM,\bar{\bu}_\h)$ &  NDOF & $\cJ(\yb_\rM,\bar{\bu}_\h)$ & NDOF & $\cJ(\yb_\rM,\bar{\bu}_\h)$  \\
\hline
1 & 25  & 0.4913056 & 25  & 0.4829083 & 25  & 0.4747933 & 25  & 0.4669466 \\
4 & 91  & 0.4981494 & 113 & 0.4962486 & 63  & 0.4922649 & 75  & 0.4921652 \\
7 & 288 & 0.4989332 & 326 & 0.4978780 & 303 & 0.4964251 & 354 & 0.4952236 \\
10 & 820 & 0.4991356 & 1086 & 0.4982860 & 843 & 0.4973521 & 817 & 0.4964078 \\
13 & 2189 & 0.4992070 & 2963 & 0.4984304 & 2137 & 0.4976077 & 2167 & 0.4968123 \\
16 & 5601 & 0.4992327 & 7760 & 0.4984736 & 5458 & 0.4976978 & 5661 & 0.4969383 \\
19 & 14030 & 0.4992430 & 19691 & 0.4984912 & 13641 & 0.4977313 & 14447 & 0.4969803 \\
22 & 34372 & 0.4992469 & 46253 & 0.4984969 & 33146 & 0.4977462 & 35423 & 0.4970010 \\
25 & 79916 & 0.4992485 & 106419 & 0.4984998 & 78632 & 0.4977517 & 85472 & 0.4970069 \\
\hline
\end{tabular}
\caption{Comparison of the cost functions for different pairs of $\D$ and $\Z$ as in Figure \ref{fig:squares1} in Example \ref{ex7}}
\label{tab:Ex1_OCPnew2}
\end{table}

\medskip
\noindent 

\noindent Table \ref{tab:Ex1_OCPnew2} shows that introducing additional point sources leads to a decrease in the values of the cost functions. This demonstrates that increasing the number of point sources contributes to a more effective control, thereby lowering the value of the cost functional.
\begin{example}{(\bf Increasing the cardinality of $\D$ for fixed $\Z$).}\label{ex8}
     Here, we fix the location of the point source $(\Z=\{(0.5,0.5)\})$, and gradually we force the state to drive towards the desired state at more number of points $(\D)$, and compare its effects on the cost functions. {The locations of points in $\D $ are same as position of $\Z $ in Figure \ref{fig:squares1}.}
 \end{example}
\begin{figure}[ht!] 
\begin{tabular}{cccc}
\includegraphics[width=3.5cm,height=3.5cm]{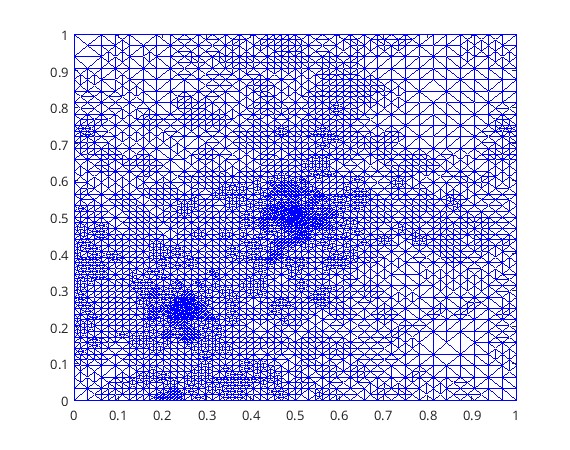} 
& \includegraphics[width=3.5cm,height=3.5cm]{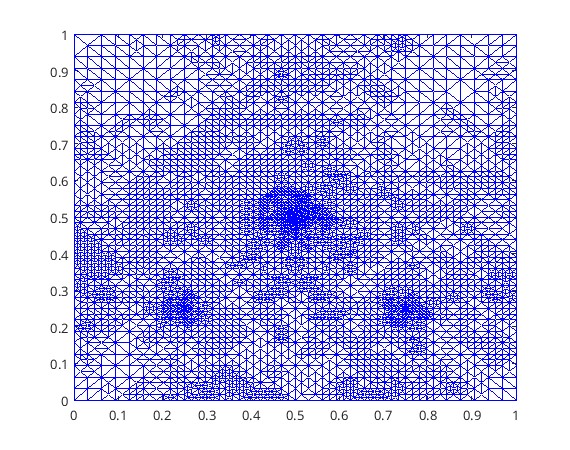}
& \includegraphics[width=3.5cm,height=3.5cm]{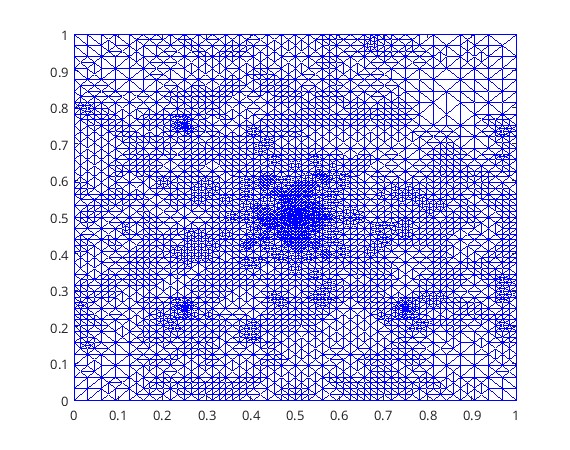}
& \includegraphics[width=3.5cm,height=3.5cm]{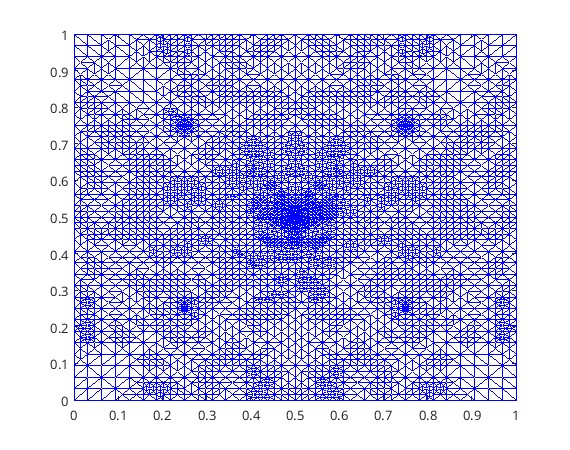}\\
 \footnotesize{(a) Adaptive mesh $(\D1,\Z)$} &
 \footnotesize{(b) Adaptive mesh $(\D2,\Z)$} &
 \footnotesize{(c) Adaptive mesh $(\D3,\Z)$} &
 \footnotesize{(d) Adaptive mesh $(\D4,\Z)$}
	 \end{tabular}   	
  \caption{\small  The adaptive meshes at the $15{\rm th}$ iteration in Example \ref{ex8}}
\label{Ex3_OCP3}
\end{figure}
\begin{table}[ht!]
\centering
\scalebox{0.8}{
\begin{tabular}{|c|c|c|c|c|c|c|c|c|c|c|c|c|}
\hline
\multicolumn{1}{|c|}{}& \multicolumn{3}{c|}{$\D1$} & \multicolumn{3}{c|}{$\D2$} & \multicolumn{3}{c|}{$\D3$}
 & \multicolumn{3}{c|}{$\D4$}\\ 
\hline 
Iter & NDOF &  $\cJ(\yb_\rM,\bar{\bu}_\h)$ & $\bar{\bu}_\h$ & NDOF & $\cJ(\yb_\rM,\bar{\bu}_\h)$ & $\bar{\bu}_\h$ &  NDOF & $\cJ(\yb_\rM,\bar{\bu}_\h)$ & $\bar{\bu}_\h$ & NDOF & $\cJ(\yb_\rM,\bar{\bu}_\h)$ & $\bar{\bu}_\h$ \\
\hline
1  & 25  & 0.49131 & 4.13 & 25  & 0.99527 & 7.00 & 25  & 1.48065 &7.00 & 25 & 1.96604 &7.00 \\
        4  & 89 & 0.49626 & 2.72 & 89  & 0.98516 & 5.41  & 89  & 1.49582 &7.00 & 99  & 1.98906 &7.00 \\
        7  & 213 & 0.49868 & 1.62 & 221 & 0.99452 &3.30 & 250 & 1.48687 &5.10 & 253 & 1.97674  &6.78\\
        10 & 632 & 0.49909 & 1.35 & 680 & 0.99633 &2.70 & 621 & 1.49123  &4.18& 659 & 1.98461 &5.53 \\
        13 & 1758 & 0.49919 & 1.27 & 1830 & 0.99677 &2.54 & 1827 & 1.49253 &3.85 & 1775 & 1.98683 &5.11 \\
        16 & 4976 & 0.49923 & 1.24 & 4916 & 0.99692 &2.48 & 4754 & 1.49303 &3.72 & 4943 & 1.98762 &4.96 \\
        19 & 12639 & 0.49924 & 1.23 & 11868 & 0.99697 &2.46 & 11612 & 1.49318 &3.69 & 12059 & 1.98789 &4.91 \\
        22 & 30262 & 0.49925 & 1.23 & 28979 & 0.99699 &2.45 & 28182 & 1.49323 &3.67 & 29519 & 1.98798 & 4.89 \\
        25 & 72052 & 0.49925 & 1.23 & 68502 & 0.99700 &2.45 & 67320 & 1.49325 &3.67 & 71138 & 1.98802 &4.88 \\
\hline
\end{tabular}}
\caption{Comparison of the cost functions for different pairs of $\D$ and $\Z$ as in Example \ref{ex8}}
\label{tab:Ex2_OCPnew2}
\end{table}
\medskip
\noindent The adaptive mesh at 15th iteration is illustrated in Figure \ref{Ex3_OCP3}. It shows more refinements near the points in $\D$ and $\Z$. We have observed (in the undisplayed result) that both state and adjoint estimators gradually increase as we increase the cardinality of $\D$, and the state estimator dominates the adjoint estimator in all the cases. Table \ref{tab:Ex2_OCPnew2} indicates that the values of the cost functions and discrete control solutions are gradually rising. This shows that when we apply force at a fixed point and aim to minimize the cost function at multiple locations, it results in higher overall costs and necessitates more force \((u_\z)\) at that point.

\medskip\noindent {\bf Acknowledgments.}
The authors sincerely thank the three anonymous referees for their careful reviews and valuable corrections.
Asha K. Dond acknowledges the support  provided by the Anusandhan National Research Foundation (ANRF) through the ARG grant ANRF/ARG/2025/009285/MS. {Neela Nataraj and Subham Nayak acknowledge the JC Bose grant ANRF/JBG/2025/000209/HAA. } 

\medskip\noindent {\bf Data Availability}{ Data sharing is not applicable to this article as no datasets were generated or analyzed.}

\medskip\noindent {\bf Declarations}\\
{ \textbf{Conflict of interest} The authors declare no conflict of interest during the current study.}
 
 {\footnotesize{
\bibliographystyle{plain}
\bibliography{vKeBib} }}
\pagestyle{empty}
\section*{Appendix}

\appendix

\renewcommand{\thesection}{\Alph{section}}

\counterwithin{equation}{section}
\renewcommand{\theequation}{\thesection.\arabic{equation}}

\section{Proof of first-order discrete optimality condition}\label{ap:1stcndn}
\begin{proof}[Proof of \ref{optimality}]
     Define the map $S: \mathbb{R}^n\rightarrow L^2(\Omega) $ such that $S\bu:=y_\rM $ with 
     $(D^2_\pw y_\rM,D^2_\pw v_\rM)_{L^2(\Omega)}=\sum_{{\rm z}\in \Z}u_{\rm z}\langle \delta_{\rm z},Jv_\rM \rangle \fl v_\rM\in V_\rM,$ 
     where $\bu=\{u_\z\}_{\z\in\Z}$. The reduced functional corresponding to the OCP in \eqref{discrete_cost} is 
     $\displaystyle f(\bu):=\frac{1}{2}\norm{S\bu-\yd}_{L^2(\Omega)}^2+\frac{\alpha}{2}\sum_{{\rm z}\in \Z}|u_{\rm z}|^2.$ 
     The map $f$ is lower semi-continuous and $U_{\rm ad}$ is a closed, bounded, and convex subset of $\mathbb{R}^n$. Hence, the existence and uniqueness of the OCP in \eqref{discrete_cost} follows (see \cite[Theorem 2.14]{TF2010}). Let $\bar{\bu}_\h:=\{\ub_{\rm z,h}\}_{\z\in\Z}$ be the optimal control. The convexity of $U_{\rm ad}$ leads to the following variational inequality: 
      $   f'(\bar{\bu}_\h)(\bu-\bar{\bu}_\h)\geq 0 \fl \bu\in U_{\rm ad}.$
     The Fr\'echet differentiation of the reduced functional, $f$ at $\ub_{\rm z,h}$ reveals
     \begin{align}
         f'(\bar{\bu}_\h)(\bu-\bar{\bu}_\h)=(S\bar{\bu}_\h-\yd,S(\bu-\bar{\bu}_\h))+\alpha (\bar{\bu}_\h,\bu-\bar{\bu}_\h)\geq 0.\label{dpl3}
     \end{align}
     Consider \eqref{dis_cost_b} with respect to $\bu$ and $\bar{\bu}_\h$ to arrive at
      $  a_\pw( y_\rM-\yb_\rM , v_\rM)=\sum_{{\rm z}\in \Z}(u_{\rm z}-\ub_{\rm z,h})\langle\delta_{\rm z},Jv_\rM\rangle$ $\fl v_\rM\in V_\rM.$
     Utilize $v_\rM=\pb_\rM$ as the test function in this equation to obtain
         $a_\pw( y_\rM-\yb_\rM ,  \pb_\rM)=\sum_{{\rm z}\in \Z}(u_{\rm z}-\ub_{\rm z,h})\langle\delta_{\rm z},J\pb_\rM\rangle.$
     Consider $\phi_\rM=y_\rM-\yb_\rM$ in \eqref{dadj} to arrive at
      $   a_\pw( y_\rM-\yb_\rM , \pb_\rM)=(\yb_\rM-\yd,y_\rM-\yb_\rM).$
     The last two equations show $(\yb_\rM-\yd,y_\rM-\yb_\rM)=\sum_{{\rm z}\in \Z}(u_{\rm z}-\ub_{\rm z,h})J\pb_\rM({\rm z})$.  Substitute this in  \eqref{dpl3} to arrive at 
     \begin{align}
        \sum_{{\rm z}\in \Z} (u_{\rm z}-\ub_{\rm z,h})J\pb_\rM({\rm z})+\sum_{{\rm z}\in \Z}(\alpha \ub_{\rm z,h},u_{\rm z}-\ub_{\rm z,h})=\sum_{{\rm z}\in \Z}(J\pb_\rM({\rm z})+\alpha \ub_{\rm z,h})(u_{\rm z}-\ub_{\rm z,h})\geq 0.\no
     \end{align}
     This concludes the proof.
 \end{proof} 

 \section{Point-wise tracking OCP with point sources {\bf (P3)}}\label{OCP_new2}
This section discusses the point-wise tracking OCP governed by the biharmonic equation with point sources. The state variable will drive towards the desired state at points specified in $\D$ by applying force at some points denoted as $\Z$ in the domain.
{Recall the OCP in \eqref{opt2n} with the desired state ${\bf y}_{\rm d}:=\{y_{ \zeta\rm,d}\}_{\zeta\in\D}$} that minimizes the cost functional
\begin{align} 
\cJ(y,\bu):=\frac{1}{2}\sum_{\zeta\in\D}|y(\zeta)-y_{ \zeta\rm,d}|^2+\frac{\alpha}{2}\norm{\bu}_{\mathbb{R}^n}^2\no
\end{align} 
in the admissible set $U_{\rm ad}=\big\{\bu=\{u_{\rm z}\}_{{\rm z}\in\Z}\in \mathbb{R}^n:~ a_{\rm z}\leq u_{\rm z} \leq b_{\rm z}\fl {\rm z}\in \Z \big\}$ subject to 
\begin{align}
\Delta^2 y=\sum_{{\rm z}\in\Z}u_{\rm z}\delta_{\rm z} \text{ in }\Omega\text{ and } y=\frac{\partial y}{\partial \nu} =0 \text{ on } \partial \Omega.\no
\end{align}
For all $\phi\in V$, $\phi_\rM\in V_\rM$, and ${\bf{v}}:=\{v_{\rm z}\}_{{\rm z}\in \Z}\in U_{\rm ad}$, the continuous and discrete optimality systems are:
\begin{equation}
\label{eq3:discrete}
\left.
\begin{array}{l@{\hspace{1cm}}l}
\text{Seek } (\yb,\pb,\bar{\bu})\in V\times V\times U_{\rm ad} \text{ such that} &
\text{Seek } (\bar{y}_\rM, \bar{p}_\rM, \bar{\bu}_{\h})\in V_\rM \times V_\rM \times U_{\rm ad} \text{ such that} \\[1.2ex]
a(\yb,\phi) = \sum_{{\rm z}\in\Z} \ub_{\rm z}\langle \delta_{\rm z}, \phi \rangle &
a_\pw(\yb_\rM,\phi_\rM) = \sum_{{\rm z}\in\Z} \ub_{\rm z,h} \langle \delta_{\rm z}, \phi_\rM \rangle \\[1.2ex]
a(\phi,\pb) = \sum_{\zeta\in\D} (\yb(\zeta) - y_{\zeta\rm,d}) \langle \delta_\zeta, \phi \rangle &
a_\pw(\phi_\rM, \pb_\rM) = \sum_{\zeta\in\D} (\yb_\rM(\zeta) - y_{\zeta\rm,d}) \langle \delta_\zeta, \phi_\rM \rangle \\[1.2ex]
\sum_{\z\in\Z} (\pb(\z) + \alpha \ub_{\z}) (v_{\z} - \ub_{\z}) \geq 0. &
\sum_{\z\in\Z} (\pb_\rM(\z) + \alpha \ub_{\rm z,h}) (v_{\z} - \ub_{{\rm z},h}) \ge 0.
\end{array}
\right\}
\end{equation}

\noindent {\it The atoms of the Dirac measure, $\z\in \Z$ and $\zeta\in \D$ are considered at the nodes of the initial triangulation, ${\mathcal{N}}^i(\T_0)$.} This is to ensure that the expression on the right-hand side of the discrete state and adjoint equation is well-defined. The existence and uniqueness of the solution of the OCP follow from \cite[Section 5.3]{AAS21}.

\noindent Define the auxiliary equations for this problem as in \eqref{aux1} by replacing the loads according to the discrete optimality system as follows. Let $\ty,\tp\in V$ satisfy
$$a(\ty,\phi)=\sum_{{\rm z}\in\Z}\ub_{\rm z,h}\langle \delta_{\rm z},\phi\rangle \text{ and }a(\phi,\tp)=\sum_{\zeta\in\D}(\yb_\rM(\zeta)-y_{\zeta\rm,d}) \langle\delta_\zeta ,\phi\rangle \text { for all }\phi\in V.$$

\noindent {For \eqref{opt1}, the discrete state equation \eqref{state_eq1} involves a Dirac measure on the right-hand side, whereas the discrete adjoint equation \eqref{dadj} has a source term in $L^2(\Omega)$. On the other hand, in the case of \eqref{opt2n}, the right-hand side terms in both the state and adjoint equations include finite sums of Dirac measures (see \eqref{eq3:discrete}). As a result, the strategy used to estimate the state error in \eqref{opt1} is employed for estimating both the state and adjoint errors in \eqref{opt2n}.}

\medskip
\noindent {Following analogous steps as  Theorem~\ref{equivalence}, we arrive at the result that the complete error
is equivalent to the sum of energy norm errors in the auxiliary equations.}
This, with the a~priori and a~posteriori error control for auxiliary PDEs defined above, will play the key role in establishing the following a~priori and a~posteriori error estimates for \eqref{opt2n}.
 \begin{tcolorbox}[colback=gray!5, colframe=black, boxrule=0.4pt, arc=2pt,
  left=6pt, right=6pt, top=6pt, bottom=6pt]
For any $T \in \T$, define the edge estimators as 
$$\displaystyle
\eta_{S,\E(T)}^2 := \sum_{E \in \E(T)} h_T \norm{[D_{\rm pw}^2 \yb_\rM]_E \tau_E}_{L^2(E)}^2, \qquad
\eta_{A,\E(T)}^2 := \sum_{E \in \E(T)} h_T \norm{[D_{\rm pw}^2 \pb_\rM]_E \tau_E}_{L^2(E)}^2.
$$
The a~posteriori error estimator reads  
$\displaystyle 
\eta^2 := \eta_S^2 + \eta_A^2
$
with
$$ \displaystyle 
\eta_S^2 = \sum_{T \in \T} \eta_{S,\E(T)}^2, \text{ and } 
\eta_A^2 = \sum_{T \in \T} \eta_{A,\E(T)}^2.
$$
\end{tcolorbox}
 \begin{thm}[error control] \label{OCPapn}
   The continuous and discrete solutions $(\yb,\pb,{\bar\bu})$ and
   $(\yb_\rM,\pb_\rM,\bar{\bu}_\h)$
   to \eqref{eq3:discrete} satisfy
    (a) the a~priori error estimates below for $(\yb,\pb,{\bar\bu})
   \in  (V \cap H^{4-\gamma}(\Omega))^2\times U_{\ad}$:
    \begin{align*}
     & \norm{\bar{\bu}-\bar{\bu}_\h}_{\mathbb{R}^n}+\enorm{\yb-\yb_\rM}_\pw+\enorm{\pb-\pb_\rM}_\pw\leq {C_{\rm AE}h^{2-\gamma}\big(\max\{\norm{\bf a}_{L^\infty(\mathbb{R}^n)},\norm{\bf b}_{L^\infty(\mathbb{R}^n)}\}+\norm{\by_{\rm d}}_{L^\infty(\mathbb{R}^m)}\big),}\\
     &  \norm{\bar{\bu}-\bar{\bu}_\h}_{\mathbb{R}^n}+\norm{\yb-\yb_\rM}_{H^{\gamma}(\T)} +\norm{\pb-\pb_\rM}_{H^{\gamma}(\T)} \leq  C_{\rm AP} h^{2-\gamma}\big(\norm{\bar{\bu}-\bar{\bu}_\h}_{\mathbb{R}^n} +\enorm{\yb-\yb_\rM}_\pw+\enorm{\pb-\pb_\rM}_\pw\big).
    \end{align*}
    %
%
(b) reliable and efficient a posteriori error estimates that read:
\begin{align*}
C_{\rm REL}^{-2}(\norm{\bar{\bu}-\bar{\bu}_\h}^2_{\mathbb{R}^n}+\enorm{\yb -\yb_\rM}_{\pw}^2+\enorm{\pb -\pb_\rM}_{\pw}^2) \leq  \eta^2 \leq  C_{\rm EFF}^2\big( \norm{\bar{\bu}-\bar{\bu}_\h}^2_{\mathbb{R}^n}+\enorm{\yb -\yb_\rM}_{\pw}^2 + \enorm{\pb-\pb_\rM}_{\pw}^2\big).
\end{align*}
\end{thm} 
\noindent The proof is skipped as it is analogous to the proofs of Theorem \ref{apriori}-Theorem \ref{aposteriori}. 

\medskip
\noindent Now we discuss the quasi-optimality of AFEM. The axioms {\bf{(A1)}} and {\bf{(A2)}} follow similar to the discussions in Section \ref{A1}. Here, brief details of the proofs of the verification of {\bf{(A3)}} -$({\rm \bf  A4}_{\varepsilon})$ are provided which conclude the proof of Theorem~\ref{thm:2.2} (see \cite{CMMD2014} for details). Consider the distance function $\dd(\T,\widehat{\T}):=\left(\norm{\widehat{\bar{\bu}}_\h-\bar{\bu}_\h}_{\mathbb{R}^n}^2+\enorm{\widehat{\yb}_\rM-\yb_\rM}_{\pw}^2 +\enorm{\widehat{\pb}_\rM-\pb_\rM}_{\pw}^2\right)^{1/2}$. \\
  
\noindent {\it Sketch of the proof of {\bf (A3).}}~
Let $(\yb_\rM,\pb_\rM,\bar{\bu}_\h)\in V_\rM \times V_\rM \times U_{\rm ad}$ (resp. $(\widehat{\yb}_\rM,\widehat{\pb}_\rM,\widehat{\bar{\bu}}_\h)\in \widehat{V}_\rM \times \widehat{V}_\rM \times U_{\rm ad}$) solve \eqref{eq3:discrete} with respect to  $\T$ (resp. $\widehat{\T}$) and $(\widehat{\ty}_\rM,\widehat{\tp}_\rM)$ solve the auxiliary problems $$a(\widehat{\ty}_\rM,\widehat{\phi}_\rM)=\sum_{{\rm z}\in\Z}\ub_{\rm z,h}\langle \delta_{\rm z},\widehat{\phi}_\rM\rangle  \text{ and }a(\widehat{\phi}_\rM,\widehat{\tp}_\rM)=\sum_{\zeta\in\D}(\yb_\rM(\zeta)-y_{\zeta\rm,d}) \langle\delta_\zeta ,\widehat{\phi}_\rM\rangle ~\text{for all }\widehat{\phi}_\rM\in \widehat{V}_\rM.$$
Now apply the steps in the proof of Lemma \ref{disequiv} to obtain
   $$ \norm{\widehat{\bar{\bu}}_\h-\bar{\bu}_\h}_{\mathbb{R}^n}+\enorm{\widehat{\yb}_\rM-\yb_\rM}_\pw+\enorm{\widehat{\pb}_\rM-\pb_\rM}_\pw\approx \enorm{\widehat{\ty}_\rM-\yb_\rM}_\pw+\enorm{\widehat{\tp}_\rM-\pb_\rM}_\pw . $$
To estimate the above two terms on the right side of the above inequality, follow  Step 1 in the proof of $({\rm \bf A3}_{\varepsilon})$ in Section \ref{disrel} to arrive at
$\enorm{\widehat{\ty}_\rM-\yb_\rM}_\pw\lesssim \eta_S(\mathcal R\mathcal{(\cT,\widehat{\cT})})\text{ and }
\enorm{\widehat{\tp}_\rM-\pb_\rM}_\pw\lesssim \eta_A(\mathcal R\mathcal{(\cT,\widehat{\cT})}).$ \qed
\begin{lem}[intermediate results for $({\rm \bf  A4}_{\varepsilon})$]\label{P3:lem_A4}
Let $(\yb_\ell,\pb_\ell,{\bf\ub}_\ell)$ be the solution of the discrete optimality system \eqref{eq3:discrete} with respect to  $\T_\ell\in\mathbb{T}(\delta)$. Then there exist $C_{\rm c}$, $C_{\rm s}$, and $C_{\rm a}$ $>0$ such that 
\begin{itemize}
\item[(a)] $ ({\bf\ub}_{k+1}-{\bf\ub}_k,{\bf\ub}_{k+1}-{\bf\ub}) \leq C_{\rm c}   \delta^{2-\gamma}\enorm{\pb_{k+1}-\pb_k}_{\pw}{e}_{k+1}$,
\item[(b)] $a_\pw(\yb_{k+1}-\yb_k,\yb_{k+1}-\yb) \leq C_{\rm s} \delta^{2-\gamma}\enorm{\pb_{k+1}-\pb_k}_{\pw}{e}_{k+1}$ ,  
\item[(c)] $a_\pw(\pb_{k+1}-\pb_k,\pb_{k+1}-\pb)\leq C_{\rm a} \delta^{2-\gamma}\enorm{\yb_{k+1}-\yb_k}_{\pw}{e}_{k+1}$.
\end{itemize}
\end{lem}
  \noindent The proof of $(a)$ follows the proof of Lemma \ref{lem_A4}$(a)$ with $C_{\rm c}(\alpha^{-3}):=n^2\alpha^{-2}C_{\rm PI}C_{\rm emb} C_{\rm AP}$. The proofs of $(b)$-$(c)$ follow the proof of Lemma \ref{lem_A4}$(c)$ with $C_{\rm s}(\alpha^{-2}):=n\alpha^{-1} C_{\rm emb}^2 (1+ C_{\rm PI})C_{\rm AP}$ and $C_{\rm a}(\alpha^{-1}):=n C_{\rm emb}^2 (1+ C_{\rm PI})C_{\rm AP}$.\\
  
\noindent {\bf Proof of $({\rm \bf  A4}_{\varepsilon})$.} 
For any $k\in \mathbb{N}\cup \{0\}$, a rearrangement of the terms, the definitions of $\dd_{k,k+1}$, $e_k$, and Lemma \ref{P3:lem_A4} result in  
\begin{align}
\dd_{k,k+1}^2+e_{k+1}^2-e_k^2 &=2({\bf\ub}_k-{\bf\ub}_{k+1},{\bf\ub}-{\bf\ub}_{k+1})+2a_\pw(\yb_{k+1}-\yb_k,\yb_{k+1}-\yb)+2a_\pw(\pb_{k+1}-\pb_k,\pb_{k+1}-\pb)\no \\
     &\leq C_{\rm sa} \delta^{2-\gamma}\dd_{k,k+1}{e}_{k+1}\no
\end{align}
with $C_{\rm sa}:= 2^{3/2}\max\{C_{\rm c},C_{\rm s},C_{\rm a}\}$. Utilize the Young's inequality with $\epsilon=\varepsilon$, $a=C_{\rm sa}C_{\rm REL} \delta^{2-\gamma}\dd_{k,k+1}$ and $b=C_{\rm REL}^{-1}{e}_{k+1}$ to obtain 
$\dd_{k,k+1}^2+2(e_{k+1}^2-e_k^2)\leq {\varepsilon}C_{\rm REL}^{-2}{e}_{k+1}^2$
by applying $C_{\rm sa}^2C_{\rm REL}^2\delta^{4-2\gamma}\leq 2\varepsilon$. A partial telescoping summation from $\ell$ to $\ell+m$ with $\varepsilon\leq 2C_{\rm REL}^2$ and the reliability of the a~posteriori error estimator in Theorem~\ref{OCPapn}$(b)$ show
$$\sum_{k=\ell}^{\ell+m}\dd_{k,k+1}^2\leq 2e_\ell^2+{\varepsilon}C_{\rm REL}^{-2} \sum_{k=\ell}^{\ell+m-1}{e}_{k+1}^2\leq 2C_{\rm REL}^2\eta_\ell^2+{\varepsilon}\sum_{k=\ell}^{\ell+m}{\eta}_{k}^2.$$ This implies that for every $\varepsilon\leq \varepsilon_0:=2C_{\rm REL}^2$ choosing the largest $\delta$ that satisfies $2^{-1}C_{\rm sa}^2C_{\rm REL}^2\delta^{4-2\gamma}\leq \varepsilon$, we obtain $({\rm \bf  A4}_{\varepsilon})$ with $\Lambda_{4(\varepsilon)}:=2C_{\rm REL}^2$. This concludes the proof.\qed
\begin{remark}[$\zeta,\z\in {\mathcal{N}^i(\T_0)}$]
Pointwise tracking OCP governed by biharmonic equations with either $L^2(\Omega)$ source terms as in \eqref{opt1n}, or point sources, as in \eqref{opt2n}, have been analyzed under the assumption that the Dirac measure atoms are located at the nodes of the initial triangulation. To relax this assumption in the context of the OCP described in \eqref{opt1}, a companion operator is employed in \eqref{discrete1}. However, for pointwise tracking OCP in \eqref{opt1n} and \eqref{opt2n}, the introduction of the companion operator complicates the well-posedness of the problem. 
\end{remark}

\end{document}

%% file: command_ccnn.tex
\newtheorem{thm}{Theorem}[section]
\newtheorem{remark}{Remark}[section]

\newtheorem{lem}[thm]{Lemma}

\theoremstyle{definition}

\theoremstyle{remark}

\newtheorem{example}{\bf Example}[section]

\numberwithin{equation}{section}

\newcommand{\bT}{\mathbb T}

\newcommand{\cT}{\mathcal T}

\newcommand{\cJ}{\mathcal J}

\newcommand{\fl}{\;\text{ for all }}

\newcommand{\dx}{{\rm\,dx}}
\newcommand{\dy}{{\rm\,dy}}

\DeclareMathOperator{\E}{\mathcal{E}}

\newcommand{\T}{\mathcal{T}}

